\documentclass[11pt,reqno]{amsart}
 \usepackage[margin=1in]{geometry} 
\usepackage{amsmath,amsthm,amssymb,amsfonts,graphicx, fullpage, amsmath, url, verbatim, amssymb}

\usepackage{bbm}
\usepackage{color}
\usepackage{afterpage}
\newcommand{\lap}{\Delta}

\newcommand{\vphi}{\varphi}
\newcommand{\R}{\mathbb{R}} 
\newcommand{\N}{\mathbb{N}}

\newcommand{\inb}{\partial_{\text{in}}}
\newcommand{\outb}{\partial_{\text{out}}}
\newcommand{\supp}{\text{supp}}
\newcommand{\interior}{\text{int}}

\numberwithin{equation}{section}
\newtheorem{thm}{Theorem}[section]
\newtheorem{lem}[thm]{Lemma}
\newtheorem{cor}[thm]{Corollary}
\newtheorem{prop}[thm]{Proposition}
\newtheorem{definition}[thm]{Definition}
\newtheorem{rmk}[thm]{Remark}

\def\XXint#1#2#3{{\setbox0=\hbox{$#1{#2#3}{\int}$ }
\vcenter{\hbox{$#2#3$ }}\kern-.6\wd0}}

\newtheorem{theorem}{Theorem}

\newtheorem{question}{Question}

\newcommand{\al}{\alpha}

\renewcommand{\bar}{\overline} 

\definecolor{purple}{rgb}{0.65, 0, 1}
\definecolor{orange}{rgb}{1,.6,0}
\definecolor{purple}{rgb}{0.5, 0, 0.9}
\definecolor{green}{rgb}{0, 1, 0}
\definecolor{orange}{rgb}{1,.5,0}
\definecolor{gray}{rgb}{.6,.6,.6}

\begin{document}

\title{Symmetry in stationary and uniformly-rotating solutions of active scalar equations}

\author{Javier G\'omez-Serrano, Jaemin Park, Jia Shi and Yao Yao}

\begin{abstract}

In this paper, we study the radial symmetry properties of stationary and uniformly-rotating solutions of the 2D Euler and gSQG equations, both in the smooth setting and the patch setting. For the 2D Euler equation, we show that any smooth stationary solution with compactly supported and nonnegative vorticity must be radial, without any assumptions on the connectedness of the support or the level sets. In the patch setting, for the 2D Euler equation we show that every uniformly-rotating patch $D$ with angular velocity $\Omega \leq 0$ or $\Omega\geq \frac{1}{2}$ must be radial, where both bounds are sharp. For the gSQG equation we obtain a similar symmetry result for $\Omega\leq 0$ or $\Omega\geq \Omega_\alpha$ (with the bounds being sharp), under the additional assumption that the patch is simply-connected. These results settle several open questions in Hmidi \cite{Hmidi:trivial-solutions-rotating-patches} and de la Hoz--Hassainia--Hmidi--Mateu \cite{delaHoz-Hassainia-Hmidi-Mateu:vstates-disk-euler} on uniformly-rotating patches. Along the way, we close a question on overdetermined problems for the fractional Laplacian \cite[Remark 1.4]{Choksi-Neumayer-Topaloglu:anisotropic-liquid-drop-models}, which may be of independent interest. The main new ideas come from a calculus of variations point of view.

\color{black}
\vskip 0.3cm


\end{abstract}

\maketitle

\section{Introduction}

Let us start by considering the initial value problem for the two-dimensional incompressible Euler equation in vorticity form. Here the evolution of the vorticity $\omega$ is given by
\begin{equation}\label{euler}
\begin{cases}
\partial_t \omega+ u \cdot \nabla \omega = 0 \ &\text{ in } \mathbb{R}^2 \times \mathbb{R}_+,  \\
 u(\cdot,t) = -\nabla^{\perp}(-\Delta)^{-1}\omega(\cdot,t) &\text{ in }\mathbb{R}^2,\\
 \omega(\cdot,0) = \omega_0&\text{ in }\mathbb{R}^2,
\end{cases}
\end{equation}
where $\nabla^\perp := (-\partial_{x_2}, \partial_{x_1})$. Note that we can express $u$ as 
$
u(\cdot,t) = \nabla^\perp(\omega(\cdot,t) * \mathcal{N}),
$
where $\mathcal{N}(x) := \frac{1}{2\pi}\ln |x|$ is the Newtonian potential in two dimensions.  More generally, the 2D Euler equation belongs to the following family of active scalar equations indexed by a parameter $\alpha$, ($0 \leq \alpha < 2$), known as the generalized Surface Quasi-Geostrophic (gSQG) equations:
\begin{equation}\label{gSQG}
\begin{cases}
\partial_t \omega+ u \cdot \nabla \omega = 0 \ &\text{ in } \mathbb{R}^2 \times \mathbb{R}_+,  \\
 u(\cdot,t) = -\nabla^{\perp}(-\Delta)^{-1+\frac{\alpha}{2}}\omega(\cdot,t) &\text{ in }\mathbb{R}^2,\\
 \omega(\cdot,0) = \omega_0&\text{ in }\mathbb{R}^2.
\end{cases}
\end{equation}
Here we can also express the Biot--Savart law as 
\begin{equation}\label{u_sqg}
u(\cdot,t) = \nabla^\perp(\omega(\cdot,t) * K_\alpha),
\end{equation}
where $K_\alpha$ is the fundamental solution for $-(-\Delta)^{-1+\frac{\alpha}{2}}$, that is, 
\begin{equation}\label{def_k_alpha}
K_\alpha(x) = \begin{cases}\frac{1}{2\pi} \ln |x|& \text{ for }\alpha = 0,\\ 
-C_\alpha |x|^{-\alpha} & \text{ for } \alpha \in (0,2),\end{cases}
\end{equation} where $C_\alpha = \frac{1}{2\pi} \frac{\Gamma(\frac{\alpha}{2})}{2^{1-\alpha}\Gamma(1-\frac{\alpha}{2})} $ is a positive constant only depending on $\alpha$. 

In this paper we will be focusing on establishing radial symmetry properties for stationary and uniformly-rotating solutions to equations \eqref{euler} and \eqref{gSQG}. 
We either work with the \emph{patch} setting, where $\omega(\cdot,t)=1_{D(t)}$ is an indicator function of a bounded set that moves with the fluid, or the smooth setting, where $\omega(\cdot,t)$ is smooth and compactly-supported in $x$. For well-posedness results for patch solutions, see the global well-posedness results  \cite{Bertozzi-Constantin:global-regularity-vortex-patches,Chemin:persistance-structures-fluides-incompressibles} for \eqref{euler}, and  local well-posedness results \cite{Rodrigo:evolution-sharp-fronts-qg,Gancedo:existence-alpha-patch-sobolev,Kiselev-Yao-Zlatos:local-regularity-sqg-patch-boundary,Cordoba-Cordoba-Gancedo:uniqueness-sqg-patch}  for \eqref{gSQG} with $\alpha\in (0,2)$.

 Let us begin with the definition of a stationary/uniformly-rotating solution in the patch setting. For a bounded domain $D\subset \mathbb{R}^2$ with $C^1$ boundary, we say $\omega = 1_D$ is a \emph{stationary patch} solution to \eqref{gSQG} for some $\alpha\in[0,2)$ if $u(x)\cdot\vec{n}(x)=0$ on $\partial D$, with $u$ given by \eqref{u_sqg}. This leads to the integral equation
 \begin{equation}\label{eq_stat}
 1_D * \mathcal{K_\alpha} \equiv C_i \quad\text{ on } \partial D, 
 \end{equation}
 where the constant $C_i$ can differ on different connected components of $\partial D$. And if $\omega(x,t) = 1_D(R_{\Omega t} x)$ is a \emph{uniformly-rotating patch} solution with angular velocity $\Omega$ (where $R_{\Omega t}x$  rotates a vector $x\in\mathbb{R}^2$ counter-clockwise by angle $\Omega t$ about the origin), then $1_D$ becomes stationary in the rotating frame with angular velocity $\Omega$, that is,
$\left(\nabla^\perp(\omega(\cdot,t) * K_\alpha) - \Omega x^\perp\right) \cdot \vec{n}(x) = 0
$ on $\partial D$. As a result we have 
\begin{equation}\label{eq_rotating20}
 1_D * \mathcal{K_\alpha} - \frac{\Omega}{2}|x|^2 \equiv C_i \quad\text{ on }\partial D, 
\end{equation}
where $C_i$ again can take different values along different connected components of $\partial D$.  Note that a stationary patch $D$ also satisfies \eqref{eq_rotating20} with $\Omega=0$, and it can be considered as a special case of uniformly-rotating patch with zero angular velocity.

Likewise, in the smooth setting, if $\omega(x,t)=\omega_0(R_{\Omega t}x)$ is a uniformly-rotating solution of \eqref{gSQG} with angular velocity $\Omega$ (which becomes a stationary solution in the $\Omega=0$ case), then we have $(\nabla^\perp(\omega_0*K_\alpha)-\Omega x^{\perp})\cdot \nabla \omega_0=0$. As a result, $\omega_0$ satisfies
\begin{equation}\label{eq_smooth}
\omega_0 * \mathcal{K_\alpha} - \frac{\Omega}{2}|x|^2 \equiv C_i \quad\text{ on each connected component of a regular level set of }\omega_0,
\end{equation}
where $C_i$ can be different if a regular level set $\{\omega_0=c\}$ has multiple connected components.

Clearly, every radially symmetric patch/smooth function automatically satisfies \eqref{eq_rotating20}/\eqref{eq_smooth} for all $\Omega \in \mathbb{R}$. The goal of this paper is to address the complementary question, which can be roughly stated as following:

\begin{question}In the patch or smooth setting, under what condition must a stationary/uniformly-rotating solution be radially symmetric?
\label{question1}
\end{question}

Below we summarize the previous literature related to this question, and state our main results. We will first discuss the 2D Euler equation in the patch and smooth setting respectively, then discuss the gSQG equation with $\alpha\in(0,2)$.

\subsection{2D Euler in the patch setting} Let us deal with the patch setting first. So far affirmative answers to Question~\ref{question1} have only been only obtained for simply-connected patches, for angular velocities $\Omega=0$, $\Omega<0$ (under some additional convexity assumptions), and $\Omega=\frac{1}{2}$. For stationary patches ($\Omega=0$), Fraenkel \cite[Chapter 4]{Fraenkel:book-maximum-principles-symmetry-elliptic} proved that if $D$ satisfies \eqref{eq_rotating20} (where $K_\alpha=\mathcal{N}$) with the \emph{same} constant $C$ on the whole $\partial D$, then $D$ must be a disk. The idea is that in this case the stream function $\psi=1_D*\mathcal{N}$ solves a semilinear elliptic equation $\Delta \psi = g(\psi)$ in $\mathbb{R}^2$ with $g(\psi)=1_{\{\psi<C\}}$, where the monotonicity of the discontinuous function $g$  allows one to apply the moving plane method developed in \cite{Serrin:symmetry-moving-plane,Gidas-Ni-Nirenberg:symmetry-maximum-principle} 
to obtain the symmetry of $\psi$. As a direct consequence, every simply-connected stationary patch must be a disk. But if $D$ is not simply-connected, \eqref{eq_rotating20} gives that $\psi=C_i$  on different connected components of $\partial D$, thus $\psi$ might not solve a single semilinear elliptic equation in $\mathbb{R}^2$. Even if $\psi$ satisfies $\Delta \psi = g(\psi)$, $g$ might not have the right monotonicity. For these reasons, whether a non-simply-connected stationary patch must be radial still remained an open question.

For $\Omega< 0$, Hmidi \cite{Hmidi:trivial-solutions-rotating-patches} used the moving plane method to show that if a simply-connected uniformly-rotating patch $D$ satisfies some additional convexity assumption (which is stronger than star-shapedness but weaker than convexity), then $D$ must be a disk. In the special case $\Omega=\frac{1}{2}$, Hmidi \cite{Hmidi:trivial-solutions-rotating-patches} also showed that a simply-connected uniformly-rotating patch $D$ must be a disk, using the fact that $1_D * \mathcal{N} - \frac{\Omega}{2}|x|^2$ becomes a harmonic function in $D$ when $\Omega=\frac{1}{2}$.

On the other hand, it is known that there can be non-radial uniformly-rotating patches for $\Omega \in (0,\frac{1}{2})$. The first example dates back to the Kirchhoff ellipse \cite{Kirchhoff:vorlesungen-math-physik}, where it was shown that any ellipse $D$ with semiaxes $a,b$ is a uniformly-rotating patch with angular velocity $\frac{ab}{(a+b)^2}$. Deem--Zabusky \cite{Deem-Zabusky:vortex-waves-stationary} numerically found families of patch solutions of \eqref{euler} with $m$-fold symmetry by bifurcating from a disk at explicit angular velocities $\Omega^{0}_{m} = \frac{m-1}{2m}$  and coined the term V-states. Further numerics were done in \cite{Wu-Overman-Zabusky:steady-state-Euler-2d,Elcrat-Fornberg-Miller:stability-vortices-cylinder,LuzzattoFegiz-Williamson:efficient-numerical-method-steady-uniform-vortices,Saffman-Szeto:equilibrium-shapes-equal-uniform-vortices}. Burbea gave the first rigorous proof of their existence by using (local) bifurcation theory arguments close to the disk \cite{Burbea:motions-vortex-patches}. There have been many recent developments in a series of works by Hmidi--Mateu--Verdera and de la Hoz--Hmidi--Mateu--Verdera \cite{Hmidi-Mateu-Verdera:rotating-vortex-patch,Hmidi-delaHoz-Mateu-Verdera:doubly-connected-vstates-euler,Hmidi-Mateu-Verdera:rotating-doubly-connected-vortices} in different settings and directions (regularity of the boundary, different topologies, etc.). In particular, \cite{Hmidi-delaHoz-Mateu-Verdera:doubly-connected-vstates-euler} showed the existence of $m$-fold doubly-connected non-radial patches bifurcating at any angular velocity $\Omega \in \left(0,\frac12\right)$ from some annulus of radii $b\in(0,1)$ and $1$.

There are many other interesting perspectives of the V-states, which we briefly review below, although they are not directly related to Question~\ref{question1}.
Hassainia--Masmoudi--Wheeler \cite{Hassainia-Masmoudi-Wheeler:global-bifurcation-vortex-patches} were able to perform global bifurcation arguments and study the whole branch of V-states. Other scenarios such as the bifurcation from ellipses instead of disks have also been studied: first numerically by Kamm \cite{Kamm:thesis-shape-stability-patches} and later theoretically by Castro--C\'ordoba--G\'omez-Serrano \cite{Castro-Cordoba-GomezSerrano:analytic-vstates-ellipses} and Hmidi--Mateu \cite{Hmidi-Mateu:bifurcation-kirchhoff-ellipses}. See also the work of Carrillo--Mateu--Mora--Rondi--Scardia--Verdera \cite{Carrillo-Mateu-Mora-Rondi-Scardia-Verdera:dislocation-ellipses} for variational techniques applied to other anisotropic problems related to vortex patches. Love \cite{Love:stability-ellipses} established linear stability for ellipses of aspect ratio bigger than $\frac{1}{3}$ and linear instability for ellipses of aspect ratio smaller than $\frac{1}{3}$. Most of the efforts have been devoted to establish nonlinear stability and instability in the range predicted by the linear part. Wan \cite{Wan:stability-rotating-vortex-patches}, and Tang \cite{Tang:nonlinear-stability-vortex-patches} proved the nonlinear stable case, whereas Guo--Hallstrom--Spirn \cite{Guo-Hallstrom-Spirn:dynamics-unstable-kirchhoff-ellipse} settled the nonlinear unstable one. See also \cite{Constantin-Titi:evolution-nearly-circular-vortex-patches}. In \cite{Turkington:corotating-vortices}, Turkington considered $N$ vortex patches rotating around the origin in the variational setting, yielding solutions of the problem which are close to point vortices.

Our first main result is summarized in the following Theorem~\ref{thmA}, which gives a complete answer to Question~\ref{question1} for 2D Euler in the patch setting. Note that $D$ is allowed to be disconnected, and each connected component can be non-simply-connected. Figure~\ref{fig_bifurcation1} illustrates a comparison of our result (in red color) with the previous results (in black color).

\begin{theorem}[\textbf{= Corollary~\ref{cor_multiple_steady}, Theorems~\ref{thm_omega<0} and \ref{clockwise rotating_patch}}]
\label{thmA}
 Let $D\subset \mathbb{R}^2$ be a bounded domain with $C^1$ boundary. Assume $D$ is a stationary/uniformly-rotating patch of \eqref{euler}, in the sense that $D$ satisfies \eqref{eq_rotating20} (with $K_\alpha =\mathcal{N}$) for some $\Omega\in \mathbb{R}$. Then $D$ must be radially symmetric if $\Omega \in (-\infty,0)\cup[\frac{1}{2},\infty)$, and radially symmetric up to a translation if $\Omega=0$.
\end{theorem}
\begin{figure}[h!]
\begin{center}
\includegraphics[scale=1]{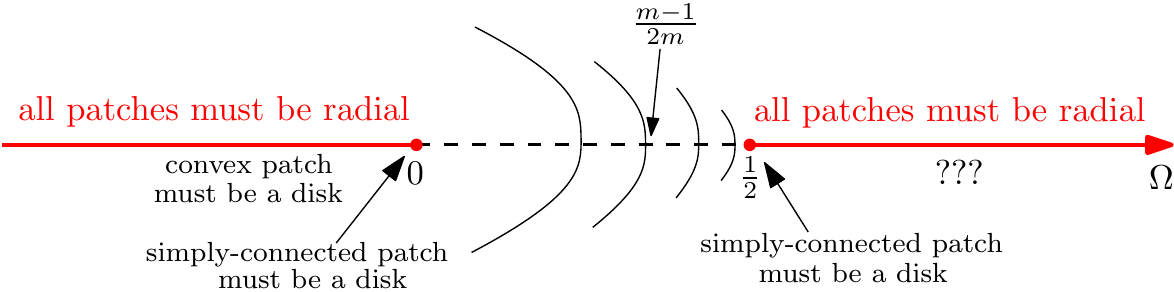}
\caption{\label{fig_bifurcation1}For 2D Euler in the patch setting, previous results on Question~\ref{question1} are summarized in black color. Our results in Theorem A are colored in red.}
\end{center}
\end{figure}

\subsection{2D Euler in the smooth setting} One of the main motivations of this paper is to find sufficient rigidity conditions in terms of the vorticity, such that the only stationary/uniformly-rotating solutions are radial ones. Heuristically speaking, this belongs to the broader class of ``Liouville Theorem'' type of results, which show that solutions satisfying certain conditions must have a simpler geometric structure, such as being constant (in one direction, or all directions) or being radial. In the literature we could not find any conditions on 2D Euler that leads to radial symmetry, although several other Liouville-type results have been established for 2D fluid equations: For 2D Euler, Hamel--Nadirashvili \cite{Hamel-Nadirashvili:liouville-euler,Hamel-Nadirashvili:shear-flow-euler-strip-halfspace} proved that any stationary solution without a stagnation point must be a shear flow. (But note that this result does not apply to our setting \eqref{eq_smooth}, since the velocity field $u$ associated with any compactly-supported $\omega_0$ must have a stagnation point). See also the Liouville theorem by Koch--Nadirashvili--Seregin--\v{S}ver\'ak for the 2D Navier--Stokes equations \cite{Koch-Nadirashvili-Seregin-Sverak:liouville-navier-stokes}. 

Let us briefly review some results on the characterization of stationary solutions to 2D Euler, although they are not directly related to Question~\ref{question1}.  Nadirashvili \cite{Nadirashvili:stationary-2d-euler} studied the geometry and the stability of stationary solutions, following the works of Arnold \cite{Arnold:geometrie-differentielle-dimension-infinie,Arnold:apriori-estimate-hydrodynamic-stability,Arnold-Khesin:topological-methods-hydrodynamics}. Izosimov--Khesin \cite{Izosimov-Khesin:characterization-steady-solutions-2d-euler} characterized stationary solutions of 2D Euler on surfaces.  Choffrut--\v{S}ver\'ak \cite{Choffrut-Sverak:local-structure-steady-euler} showed that locally near each stationary smooth solution there exists a manifold of stationary smooth solutions transversal to the foliation, and Choffrut--Sz\'ekelyhidi \cite{Choffrut-Szekelyhihi:weak-solutions-stationary-euler} showed that there is an abundant set of stationary weak ($L^{\infty}$) solutions near a smooth stationary one. Shvydkoy--Luo \cite{Luo-Shvydkoy:2d-homogeneous-euler,Luo-Shvydkoy:addendum-homogeneous-euler} classified the set of stationary smooth solutions of the form $v = \nabla^{\perp}(r^{\gamma}f(\omega))$, where $(r,\omega)$ are polar coordinates. In a different direction, Turkington \cite{Turkington:stationary-vortices} used variational methods to construct stationary vortex patches of a prescribed area in a bounded domain, imposing that the patch is a characteristic function of the set $\{\Psi > 0\}$, and also studied the asymptotic limit of the patches tending to point vortices. Long--Wang--Zeng \cite{Long-Wang-Zeng:concentrated-steady-vortex-patches} studied their stability, as well as the regularity in the smooth setting (see also \cite{Cao-Wang:nonlinear-stability-patches-bounded-domains}). For other variational constructions close to point vortices, we refer to the work of Cao--Liu--Wei \cite{Cao-Liu-Wei:regularization-point-vortices}, Cao--Peng--Yan \cite{Cao-Peng:planar-vortex-patch-steady} and Smets--van Schaftingen \cite{Smets-VanSchaftingen:desingularization-vortices-euler}. We remark that these results do not rule out that those solutions may be radial. Musso--Pacard--Wei \cite{Musso-Pacard-Wei:stationary-solutions-euler} constructed nonradial smooth stationary solutions without compact support in $\omega$.  The (nonlinear $L^1$) stability of circular patches was proved by Wan--Pulvirenti \cite{Wan-Pulvirenti:stability-circular-patches} and later Sideris--Vega gave a shorter proof \cite{Sideris-Vega:stability-L1-patches}. See also Beichman--Denisov \cite{Beichman-Denisov:stability-rectangular-strip} for similar results on the strip.

Lately, Gavrilov \cite{Gavrilov-stationary-euler-3d,Gavrilov:stationary-euler-helix} provided a remarkable construction of nontrivial stationary solutions of 3D Euler with compactly supported velocity. See also Constantin--La--Vicol for a simplified proof with extensions to other fluid equations \cite{Constantin-La-Vicol:remarks-gavrilov-stationary}.

Regarding uniformly-rotating smooth solutions $(\Omega\neq 0)$ for 2D Euler, Castro--C\'ordoba--G\'omez-Serrano \cite{Castro-Cordoba-GomezSerrano:uniformly-rotating-smooth-euler} were able to desingularize a vortex patch to produce a smooth $m$-fold V-state with $\Omega \sim \frac{m-1}{2m} > 0$ for $m \geq 2$. Recently Garc\'ia--Hmidi--Soler \cite{Garcia-Hmidi-Soler:non-uniform-vstates-euler} studied the construction of V-states bifurcating from other radial profiles (Gaussians and piecewise quadratic functions).

Our second main result is the following theorem, which gives radial symmetry of compactly supported stationary/uniformly-rotating solutions in the smooth setting for $\Omega\leq 0$, under the additional assumption $\omega_0\geq 0$:

\begin{theorem}[\textbf{= Theorem \ref{theorem3} and Corollary~\ref{cor_conti_rotating}}]
\label{thmB}
 Let $\omega_0 \geq 0$ be smooth and compactly-supported. Assume $\omega(x,t)=\omega_0(R_{\Omega t}x)$ is a stationary/uniformly-rotating solution of \eqref{euler} with $\Omega \leq 0$, in the sense that it satisfies \eqref{eq_smooth} with $K_\alpha =\mathcal{N}$. Then $\omega_0$ must be radially symmetric if $\Omega<0$, and radially symmetric up to a translation if $\Omega=0$.
 \end{theorem}




Although the extra assumption $\omega_0\geq 0$ might seem unnatural at the first glance, in a forthcoming work \cite{GomezSerrano-Park-Shi-Yao:nonradial-stationary-solutions} we will show that it is indeed necessary:  if we allow $\omega_0$ to change sign, then by applying bifurcation arguments to sign-changing radial patches, we are able to show that there exists a compactly-supported, sign-changing smooth stationary vorticity $\omega_0$ that is non-radial.

\subsection{The gSQG case ($0<\alpha<2$)}  Recall that in the patch setting, a stationary/uniformly-rotating patch satisfies \eqref{eq_rotating20} with $K_\alpha$ given in \eqref{def_k_alpha}.
Even though the kernels $K_\alpha$ are qualitatively similar for all $\alpha\in[0,2)$, 
 there is a key difference on the symmetry v.s. non-symmetry results  between the cases $\alpha=0$ and $\alpha>0$: For the 2D Euler equation ($\alpha=0$), we proved in Theorem~A that any rotating patch $D$ with $\Omega \leq 0$ must be radial, even if $D$ is not simply-connected. However, this result is not true for any $\alpha \in (0,2)$: de la Hoz--Hassainia--Hmidi--Mateu \cite{delaHoz-Hassainia-Hmidi:doubly-connected-vstates-gsqg} showed that there exist non-radial patches bifurcating from annuli at $\Omega < 0$ and G\'omez-Serrano \cite{GomezSerrano:stationary-patches} constructed non-radial, doubly connected stationary patches ($\Omega = 0$). Therefore we cannot expect a non-simply-connected rotating patch $D$ with $\Omega\leq 0$ to be radial for $\alpha\in(0,2)$.

However, if $D$ is a simply-connected stationary patch, then radial symmetry results were obtained in a series of works for $\alpha \in[0,\frac{5}{3})$, which we review below. These works consider \eqref{eq_rotating20} in a more general context not limited to dimension 2: Let $K_{\alpha, d}$ be the fundamental solution of $-(-\Delta)^{-1+\frac{\alpha}{2}}$ in $\mathbb{R}^d$ for $d\geq 2$, given by
\begin{equation}
K_{\alpha,d} := -C_{\alpha, d} |x|^{-d+2-\alpha} 
\end{equation}  for some $C_{\alpha, d}>0$; except that in the special case $-d+2-\alpha=0$ it becomes $K_{\alpha,d}=C_d \ln|x|$ for some $C_d>0$. Note that $K_{\alpha,d} \in L_{loc}^1(\mathbb{R}^d)$ for all $\alpha<2$. 
Consider the following question:

\begin{question} Let $\alpha\in[0,2)$. Assume $D\subset \mathbb{R}^d$ is a bounded domain such that 
\begin{equation}\label{multi_d}
K_{\alpha,d} * 1_D - \frac{\Omega}{2} |x|^2 = \text{const} \quad\text{ on }\partial D
\end{equation}
 for some $\Omega \leq 0$, where the constant is the same along all connected components of $\partial D$. Must $D$ be a ball in $\mathbb{R}^d$? 
 \label{question2}
 \end{question}

Positive answers to Question~\ref{question2} were obtained in the $\Omega=0$ case for $\alpha < \frac{5}{3}$ in the following works. As we discussed before, Fraenkel \cite{Fraenkel:book-maximum-principles-symmetry-elliptic} proved that $D$ must be a ball for $\alpha=0$. Also using the moving plane method,  
 Reichel \cite[Theorem 2]{Reichel:balls-riesz-potentials}, Lu--Zhu \cite{Lu-Zhu:overdetermined-riesz-potential} and Han--Lu--Zhu \cite{Han-Lu-Zhu:characterization-balls-bessel-potentials} generalized this result to $\alpha \in[0,1)$.  Here \cite{Lu-Zhu:overdetermined-riesz-potential} also covered generic radially increasing potentials not too singular at the origin (which include all Riesz potentials $K_{\alpha,d}$ with $\alpha \in[0,1)$). Recently, Choksi--Neumayer--Topaloglu \cite[Theorem 1.3]{Choksi-Neumayer-Topaloglu:anisotropic-liquid-drop-models} further pushed the range to $\alpha \in[0,\frac53)$, leaving the range $\alpha \in [\frac53,2)$ an open problem.   We point out that in all these results for $\alpha>0$, $\partial D$ was assumed to be at least $C^1$. All above results were obtained using the moving plane method.
 
In our third main result, we use a completely different approach to give an affirmative answer to Question~\ref{question2} for all $\Omega\leq 0$ and $\alpha\in [0,2)$, under a weaker assumption on the regularity of $\partial D$.

\begin{theorem}[\textbf{= Theorem \ref{thm-generic}}]
\label{thmC}
Let $D$ be a bounded domain in $\mathbb{R}^d$ with Lipschitz boundary (and if $d=2$ we only require $\partial D$ to be rectifiable). If $D$ satisfies \eqref{multi_d} for some $\Omega\leq 0$ and $\alpha\in[0,2)$, then it must be a ball in $\mathbb{R}^d$.
 \end{theorem}

As a directly consequence, Theorem C implies that for the gSQG equation with $\alpha\in[0,2)$, any simply-connected rotating patch with $\Omega\leq 0$ must be a disk (see Theorem~\ref{thm-euler}). In addition, in the smooth setting  \eqref{eq_smooth}, we prove a similar result in Corollary~\ref{thm-euler-smooth} for uniformly-rotating solutions with $\Omega\leq 0$ for all $\alpha\in[0,2)$: if the super level-sets $\{\omega_0>h\}$ are all simply-connected for all $h>0$, then $\omega_0$ must be radially decreasing.

Next we review the previous literature on uniformly-rotating solutions for the gSQG equation. Note that the case of $\alpha\in(0,2)$ is more challenging than the 2D Euler case, since the velocity is more singular and this produces obstructions to the bifurcation theory when it comes to the choice of spaces and the regularity of the functionals involved in the construction. Hassainia--Hmidi \cite{Hassainia-Hmidi:v-states-generalized-sqg} showed the existence of V-states with $C^k$ boundary regularity in the case $0 < \alpha < 1$, and in \cite{Castro-Cordoba-GomezSerrano:existence-regularity-vstates-gsqg}, Castro--C\'ordoba--G\'omez-Serrano upgraded the result to show existence and $C^{\infty}$ boundary regularity in the remaining open cases: $\al \in [1,2)$ for the existence, $\al \in (0,2)$ for the regularity. In that case, the solutions bifurcate at angular velocities given by $\Omega_{m}^{\al} := 2^{\al-1}\frac{\Gamma(1-\al)}{\Gamma\left(1-\frac{\al}{2}\right)^2}\left(\frac{\Gamma\left(1+\frac{\al}{2}\right)}{\Gamma\left(2-\frac{\al}{2}\right)} - \frac{\Gamma\left(m+\frac{\al}{2}\right)}{\Gamma\left(m+1-\frac{\al}{2}\right)}\right)$. This boundary regularity was subsequently improved to analytic in \cite{Castro-Cordoba-GomezSerrano:analytic-vstates-ellipses}. See also \cite{Hmidi-Mateu:existence-corotating-counter-rotating} for another family of rotating solutions, \cite{delaHoz-Hassainia-Hmidi:doubly-connected-vstates-gsqg,Renault:relative-equlibria-holes-sqg} for the doubly connected case and \cite{Castro-Cordoba-GomezSerrano:global-smooth-solutions-sqg} for a construction in the smooth setting. 

One can check that $\Omega^{\alpha}_{m}$ are increasing functions of $m$ for any $\alpha$, whose limit is a finite number $\Omega^{\alpha}:= 2^{\al-1}\frac{\Gamma(1-\al)}{\Gamma\left(1-\frac{\al}{2}\right)^2}\frac{\Gamma\left(1+\frac{\al}{2}\right)}{\Gamma\left(2-\frac{\al}{2}\right)}$ for $\alpha\in[0,1)$, and $+\infty$ if $\alpha\geq 1$. It is then a natural question to ask whether  there exist V-states (with area $\pi$) that rotate with angular velocity faster than $\Omega_\alpha$ for $\alpha\in(0,1)$.  Our fourth main theorem answers this question among all simply-connected patches:

\begin{theorem}[\textbf{= Theorem \ref{thm_large_omega}}]
\label{thmD}
For $\alpha\in (0,1)$, let $1_{D}$ be a simply connected V-state of area $\pi$ and let its angular velocity be $\Omega \geq \Omega^{\al}$. Then $D$ must be the unit disk.
\end{theorem}

Finally, we illustrate a comparison of our results in Theorem C and D (in red color) with the previous results (in black color) in Figure~\ref{fig_bifurcation2}.
\begin{figure}[h!]
\begin{center}
\includegraphics[scale=1]{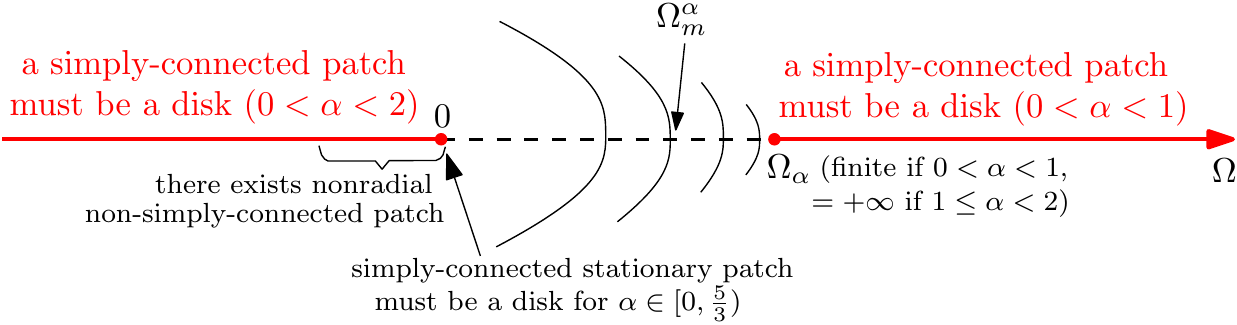}
\caption{\label{fig_bifurcation2}For gSQG in the patch setting, previous results on Question~\ref{question1} are summarized in black color, with our results in Theorem C and D colored in red.}
\end{center}
\end{figure}

\subsection{Structure of the proofs}

While all the previous symmetry results on Question~\ref{question1} and \ref{question2} \cite{Fraenkel:book-maximum-principles-symmetry-elliptic, Lu-Zhu:overdetermined-riesz-potential, Han-Lu-Zhu:characterization-balls-bessel-potentials, Hmidi:trivial-solutions-rotating-patches,  Reichel:balls-riesz-potentials, Choksi-Neumayer-Topaloglu:anisotropic-liquid-drop-models} are done by moving plane methods, our approaches are completely different, which have a variational flavor.

Theorem~\ref{thmA} is based on computing the first variation of the energy functional
$$ \mathcal{E}[1_D] = - \frac12 \int_{\mathbb{R}^2}1_D(x) (1_D * \mathcal{N})(x)  - \frac{\Omega}{2}|x|^2 1_D(x) dx $$ in two different ways,
as we deform $D$ along a carefully chosen vector field that is divergence-free in $D$. On the one hand, we show the first variation should be 0 if $D$ is a stationary/rotating patch with angular velocity $\Omega$; on the other hand, we show that the first variation must be non-zero if $\Omega \leq 0$ or $\Omega\geq \frac{1}{2}$, leading to  a contradiction.  If $D$ is simply-connected, we give a very short proof in Section~\ref{sec_21}, where a rearrangement inequality due to Talenti \cite{Talenti:rearrangements} is crucial to get a sign condition. For a non-simply-connected patch $D$, the choice of the right vector field is more involved. Since $1_D * \mathcal{N} - \frac{\Omega}{2}|x|^2$ now takes different constant values on different connected components of $\partial D$, in order to keep the first variation to be 0, we have to modify our perturbation vector field such that it also preserves the area of each hole. We then prove a new version of a rearrangement inequality for this modified vector field in a similar spirit as Talenti's result, leading to a non-zero first variation if $D$ is non-radial and $\Omega \leq 0$ or $\Omega\geq \frac{1}{2}$.

The smooth setting in Theorem~\ref{thmB} is based on a similar idea, but technically more difficult. The point of view is to approximate a smooth function by step functions and consider the above perturbation in each set where the step function is constant. To do this we need to obtain some quantitative (stability) estimates on our version of Talenti's rearrangement inequality, in particular in terms of the Fraenkel asymmetry of the domain in the spirit of Fusco--Maggi--Pratelli  \cite{Fusco-Maggi-Pratelli:stability-faber-krahn}. 

Theorem~\ref{thmC} is also based on a variational approach, but we need a different perturbation from the vector field in Theorem~\ref{thmA}, which heavily relies on the Newtonian potential, and fails for general Riesz potential $K_\alpha$. The key ingredient to prove Theorem~\ref{thmC} is to perturb $D$ using the continuous Steiner symmetrization \cite{Brock:continuous-steiner-symmetrization}, which has been successfully applied in other contexts by Carrillo--Hittmeir--Volzone--Yao \cite{Carrillo-Hittmeir-Volzone-Yao:nonlinear-aggregation-symmetry-asymptotics} (nonlinear aggregation models) or Morgan \cite{Morgan:ball-minimizes-grativational-energy} (minimizers of the gravitational energy). This method is much more flexible and allows to treat more singular kernels than in the existing papers using moving plane methods. 
Due to the low regularity of the kernels, instead of computing the derivative of the energy under the perturbation, we work with finite differences instead. 

Theorem~\ref{thmD} uses maximum principles and monotonicity formulas for nonlocal equations. The idea is to find the smallest disk $B(0,R)$ containing $D$ (which intersects $\partial D$ at some $x_0$), then use two different ways to compute $\nabla (1_{B(0,R)\setminus D}*K_\alpha)$ at $x_0$, and obtain a contradiction if $\Omega \geq \Omega_\alpha$ and $D$ is not a disk.  The proof works for the full range of $\alpha \in [0,2)$, thus closing the problem raised by Hmidi \cite{Hmidi:trivial-solutions-rotating-patches} and de la Hoz--Hassainia--Hmidi--Mateu \cite{delaHoz-Hassainia-Hmidi-Mateu:vstates-disk-euler} among all simply-connected patches.

\subsection{Organization of the paper}
The paper is split into sections according to the cases $\alpha = 0$ (Euler) and $\alpha \neq 0$ (gSQG). Sections 2 and 3 are devoted to prove the symmetry results for $\alpha = 0$, in the patch setting (Section 2) and in the smooth setting (Section 3). Sections 4 and 5 deal with the gSQG equations with $0<\alpha<2$. Section 4 is concerned with the case $\Omega \leq 0$, whereas Section 5 handles the case $\Omega \geq \Omega_c$.

\vspace*{-0.2cm}

\subsection{Notations}
Through Section \ref{sec2}--\ref{sec3} of this paper, we use the following notations.

For a simple closed curve $\Gamma$, denote $\interior(\Gamma)$ by its interior, which is the bounded connected component of $\R^{2}$ separated by the curve $\Gamma$. Note that the Jordan--Schoenflies theorem guarantees that $\interior(\Gamma)$ is open and simply connected. 

We say that two disjoint simple closed curves $\Gamma_1$ and $\Gamma_2$ are nested if $\Gamma_1\subset\interior(\Gamma_2)$ or vice versa. We say that two connected domains $D_1, D_2$ are nested if one is contained in a hole of the other one.

For a bounded connected domain $D\subset \mathbb{R}^2$, we denote by $\outb D$ its outer boundary. And if $D$ is doubly-connected, we denote by $\inb D$ its inner boundary, 

For a set $D$, we use $1_D(x)$ to denote its indicator function. And for a statement $S$, we let $\mathbbm{1}_S = \begin{cases} 1 & \text{ if $S$ is true}\\0 & \text{ if $S$ is false}\end{cases}$. (e.g. $\mathbbm{1}_{\pi<3}=0$).

For a domain $U\subset \mathbb{R}^2$, in the boundary integral $\int_{\partial U} \vec{n} \cdot \vec{f} d\sigma$, the vector $\vec{n}$ is taken as the outer normal of the domain $U$ in that integral. 

\section{Radial symmetry of steady/rotating patches for 2D Euler equation}\label{sec2}

Throughout this section we work with the 2D Euler equation \eqref{euler} in the patch setting. For a stationary or uniformly-rotating patch $D$ with angular velocity $\Omega\in \mathbb{R}$, let 
\[
f_{\Omega}(x) := (1_D * \mathcal{N})(x) - \frac{\Omega}{2}|x|^2.
\]
Recall that in \eqref{eq_rotating20} we have shown that $f_\Omega\equiv C_i$ on each connected component of $\partial D$, where the constants can be different on different connected components.

Our goal in this section is to prove Theorem~\ref{thmA}, which completely answers Question~\ref{question1} for 2D Euler patches. As we described in the introduction, our proof has a variational flavor, which is done by perturbing $D$ by a carefully chosen vector field, and compute the first variation of an associated energy functional in two different ways. In Section~\ref{sec_21}, we will define the energy functional and the perturbation vector field, and give a one-page proof in Theorem~\ref{thm_simply_connected} that answers Question~\ref{question1} among simply-connected patches. (Note that even among simply-connected patches, it is an open question whether every rotating patch with $\Omega>\frac{1}{2}$ or $\Omega<0$ must be a disk.) In the following subsections, we further develop this method, and modify our perturbation vector field to cover non-simply-connected patches.

\subsection{Warm-up: radial symmetry of simply-connected rotating patches}\label{sec_21}
We begin by providing a sketch and some motivations of our approach, and then give a rigorous proof afterwards in Theorem~\ref{thm_simply_connected}.  
Suppose that $D$ is a $C^1$ simply-connected rotating patch with angular velocity $\Omega$ that is \emph{not} a disk. We perturb $D$ in 
``time'' (here the ``time'' $t$ is just a name for our perturbation parameter, and is irrelevant with the actual time in the Euler equation) 
with a velocity field $\vec{v}(x) \in C^1(D) \cap C(\bar D)$ that is divergence-free in $D$, which we will fix later. That is, consider the transport equation
\[
\rho_t + \nabla\cdot(\rho \vec{v}) = 0
\]
with $\rho(\cdot,0) = 1_D$. We then investigate how the ``energy functional''
\[
\mathcal{E}[\rho] := - \int_{\mathbb{R}^2} \frac{1}{2} \rho(x) (\rho * \mathcal{N})(x) -  \frac{\Omega}{2}|x|^2 \rho(x) dx
\]
 changes in time under the perturbation. Formally, we have
 \begin{equation}\label{time_derivative}
 \begin{split}
 \frac{d}{dt}\mathcal{E}[\rho]\Big|_{t=0} &= -\int_{\mathbb{R}^2} \rho_t(x,0) \Big((\rho(\cdot,0)*\mathcal{N})(x) - \frac{\Omega}{2}|x|^2 \Big) dx \\
 &=  -\int_D\vec{v}(x) \cdot\nabla\Big((1_D*\mathcal{N})(x) - \frac{\Omega}{2}|x|^2\Big) dx.
 \end{split}
 \end{equation}
The above transport equation and the energy functional only serve as our motivation, and will not appear in the proof. In the actual proof we only focus on the right hand side of \eqref{time_derivative}, which is an integral that is well-defined by itself:
\begin{equation}\label{def_I}
\mathcal{I} := - \int_D\vec{v}(x) \cdot\nabla\Big((1_D*\mathcal{N})(x) - \frac{\Omega}{2}|x|^2 \Big) dx = -\int_D\vec{v}\cdot\nabla f_\Omega \,dx.
\end{equation}
We will use two different ways to compute $\mathcal{I}$, and show that if $D$ is not a disk, the two ways lead to a contradiction for $\Omega \leq 0$ or $\Omega \geq \frac12$.

 On the one hand, since $f_\Omega$ is a constant on $\partial D$ (denote it by $c$), the divergence theorem yields the following for \emph{every} $\vec v\in C^1(D) \cap C(\bar D)$ that is divergence-free in $D$:
 \begin{equation}\label{i=0}
 \mathcal{I}=-c\int_{\partial D} \vec{n}\cdot\vec{v} d\sigma + \int_D (\nabla\cdot\vec{v}) f_\Omega \,dx =-c\int_D \nabla\cdot\vec{v} dx + \int_D (\nabla\cdot\vec{v}) f_\Omega \,dx = 0.
 \end{equation} On the other hand, we fix $\vec v$ as follows, which is at the heart of our proof. Let $\vec v(x) := -\nabla \varphi(x)$ in $D$, where 
 \begin{equation}
 \label{def_phi1}
 \varphi(x) := \frac{|x|^2}{2} + p(x) \quad \text{ in }D,
 \end{equation}
 with $p(x)$ being the solution to Poisson's equation
 \begin{equation}
 \label{def_p1}
 \begin{cases}
 \Delta p(x) = -2 & \text{ in } D\\
 p(x) = 0 & \text{ on }\partial D.
 \end{cases}
 \end{equation}
 Note that $\varphi$ is harmonic in $D$, thus $\vec{v}$ is indeed divergence-free in $D$. This definition of $\vec v$ is motivated by the fact that among all divergence-free vector fields in $D$, such $\vec v$ is the closest one to $-\vec x$ in the $L^2(D)$ distance. (In fact, such $\vec v$ is connected to the gradient flow of $\int_D \frac{|x|^2}{2} dx$ in the metric space endowed by 2-Wasserstein distance, under the constraint that $|D(t)|$ must remain constant \cite{Maury-RoudneffChupin-Santambrogio:macroscopic-crowd-gradient-flow, Maury-RoudneffChupin-Santambrogio-Venel:handling-congestion-crowd, Alexander-Kim-Yao:quasistatic-evolution-crowd}.) Formally, one expects that $D$ becomes ``more symmetric'' as we perturb it by $\vec v$, which inspires us to consider the first variation of $\mathcal{E}$ under such perturbation.

 In the proof we will show that with such choice of $\vec{v}$, we can compute $\mathcal{I}$ in another way and obtain that $\mathcal{I} > 0$ for $\Omega\leq 0$ and $\mathcal{I} < 0$ for $\Omega \geq \frac{1}{2}$. Therefore in both cases, we obtain a contradiction with $\mathcal{I}=0$ in \eqref{i=0}.
 
Our proof makes use of a rearrangement inequality for solutions to elliptic equations, which is due to Talenti \cite{Talenti:rearrangements}. Below is the form that we will use; the original theorem works for a more general class of elliptic equations.



\begin{prop}[\cite{Talenti:rearrangements}, Theorem 1]\label{talenti} Let $D\subset \mathbb{R}^2$ be a bounded domain with $C^1$ boundary, and let $p$ be defined as in \eqref{def_p1}. Let $B$ be a disk centered at the origin with $|B|=|D|$, and let $p_B$ solve \eqref{def_p1} in $B$. Then we have $p^* \leq p_B$ pointwise in $B$, where $p^*$ is the radial decreasing rearrangement of $p^*$. As a consequence, we have
\[
\int_D p(x) dx \leq \frac{1}{4\pi} |D|^2,
\]
and the equality is achieved if and only if $D$ is a disk.
\end{prop}


Now we are ready to prove the following theorem, saying that any simply-connected stationary/rotating patch with $\Omega \leq 0$ or $\Omega \geq \frac{1}{2}$ must be a disk. Interestingly, the same proof can treat the two disjoint intervals $\Omega \leq 0$ and $\Omega \geq \frac{1}{2}$ all at once.

\begin{thm}\label{thm_simply_connected}
Let $D$ be a simply-connected domain with $C^1$ boundary. If $D$ is a rotating patch solution with angular velocity $\Omega$, where $\Omega \leq 0$ or $\Omega\geq \frac{1}{2}$, then $D$ must be a disk, and it must be centered at the origin unless $\Omega=0$.
\end{thm}



\begin{proof}
Let $D$ be a rotating patch with $\Omega \in (-\infty,0]\cup[\frac12,\infty)$. As we described above, in this theorem we will use two different ways to compute the integral $\mathcal{I}$ defined in \eqref{def_I}, where we fix $\vec v(x) := -\nabla \varphi(x)$, with $\varphi$ and $p$ defined as in \eqref{def_phi1} and \eqref{def_p1}. 

On the one hand, we have that $\vec v$ is divergence free in $D$, and elliptic regularity theory immediately yields that $\vec v\in C^1(D) \cap C(\bar D)$. Using the assumption that $D$ is a rotating patch, we know $f_\Omega$ is a constant on $\partial D$. (Note that $\partial D$ is a connected closed curve since we assume $D$ is simply-connected). Thus the computation in \eqref{i=0} directly gives that $\mathcal{I}=0$.

On the other hand, we compute $\mathcal{I}$ as follows:
\begin{equation}\label{comp_i1}
\mathcal{I} =  -\int_D\vec{v}\cdot\nabla f_\Omega dx =  \int_D \nabla\varphi \cdot\nabla f_\Omega dx = \underbrace{\int_D x \cdot \nabla f_\Omega dx}_{=: \mathcal{I}_1} + \underbrace{ \int_D \nabla p \cdot \nabla f_\Omega dx}_{=: \mathcal{I}_2}.
\end{equation}
For $\mathcal{I}_1$, we have
\begin{equation}\label{i1_identity}
\begin{split}
\mathcal{I}_1 &=  \int_D x \cdot \nabla(1_D * \mathcal{N})dx - \int_D x \cdot \Omega x dx\\
&= \frac{1}{2\pi} \int_{D}\int_{D} \frac{x\cdot (x-y)}{|x-y|^2} dydx - \Omega \int_D |x|^2 dx\\
&=  \frac{1}{4\pi} \int_{D}\int_{D}\frac{x\cdot (x-y) - y\cdot(x-y)}{|x-y|^2}dydx - \Omega \int_D |x|^2 dx\\
&= \frac{1}{4\pi} |D|^2 -\Omega \int_D |x|^2 dx,
\end{split}
\end{equation}
where the third equality is obtained by exchanging $x$ with $y$ in the first integral, then taking average with the original integral. To compute $\mathcal{I}_2$, using the divergence theorem (and the fact that $p = 0$ on $\partial D$), we have
\begin{equation}\label{comp_i2}
\begin{split}
\mathcal{I}_2 = -\int_D p \Delta f_\Omega dx = (2\Omega-1) \int_D p dx.
\end{split}
\end{equation}
Plugging \eqref{i1_identity} and \eqref{comp_i2} into \eqref{comp_i1} gives
\begin{equation}\label{I_eq2}
\mathcal{I} =  \frac{1}{4\pi} |D|^2 -\Omega \int_D |x|^2 dx +  (2\Omega-1) \int_D p dx.
\end{equation}
When $\Omega=0$, Proposition~\ref{talenti} directly gives that $\mathcal{I}> 0$ if $D$ is not a disk, contradicting $\mathcal{I}=0$. 

When $\Omega \in (-\infty, 0) \cup[\frac12, \infty)$, let $B$ be a disk centered at the origin with the same area as $D$. Towards a contradiction, assume $D \neq B$. Among all sets with the same area as $D$, the disk $B$ is the unique one that minimizes the second moment, thus we have \[
\int_D |x|^2 dx > \int_B |x|^2 dx = \frac{1}{2\pi} |D|^2, 
\] where the last step follows from an elementary computation. Plugging this into \eqref{I_eq2} gives the following inequality for $\Omega \in [\frac12, \infty)$:
\[
\mathcal{I} < \frac{1}{4\pi}|D|^2 - \frac{\Omega}{2\pi} |D|^2  +  (2\Omega-1) \int_D p dx = (1-2\Omega) \Big(\frac{1}{4\pi}|D|^2 -  \int_D p dx \Big) \leq 0.
\]
 On the other hand, for $\Omega \in (-\infty, 0)$, we have 
\[
\mathcal{I} > \frac{1}{4\pi}|D|^2 - \frac{\Omega}{2\pi} |D|^2  +  (2\Omega-1) \int_D p dx = (1-2\Omega) \Big(\frac{1}{4\pi}|D|^2 -  \int_D p dx \Big) > 0,
\]
and we get a contradiction to $\mathcal{I}=0$ in all the cases, thus the proof is finished.
\end{proof}

\begin{rmk}\label{rmk_disconnected}
In fact, one can easily check that Theorem~\ref{thm_simply_connected} holds for a bounded disconnected patch $D = \dot\cup_{i=1}^N D_i$ with $C^1$ boundary, as long as each connected component $D_i$ is simply-connected. Here the proof remains the same, except a small change in the $\mathcal{I}=0$ proof: since now we have $f_\Omega=c_i$ on $\partial D_i$, \eqref{i=0} should be replaced by
\[
 \mathcal{I}=-\sum_{i=1}^N\left( c_i\int_{\partial D_i} \vec{n}\cdot\vec{v} d\sigma + \int_{D_i} (\nabla\cdot\vec{v}) f_\Omega dx \right)= 0.
\]
\end{rmk}

Even in the regime $\Omega \in (0,\frac{1}{2})$, where  non-radial rotating patches are known to exist (recall that there exist patches bifurcating from a disk at $\Omega_{m} = \frac{m-1}{2m}$ for all $m\geq 2$),  our approach still allows us to obtain the following quantitative estimate, saying that if a simply-connected patch $D$ rotates with angular velocity $\Omega \in (0,\frac{1}{2})$ that is very close to $\frac{1}{2}$, then $D$ must be very close to a disk, in the sense that their symmetric difference must be small.

\begin{cor}\label{cor_stability1}
Let $D$ be a simply-connected domain with $C^1$ boundary. Assume $D$ is a rotating patch solution with angular velocity $\Omega$, where $\Omega \in (\frac{1}{4}, \frac{1}{2})$. Let $\delta := \frac{1}{2}-\Omega$. Then we have 
\[
|D\triangle B| \leq 2\sqrt{2\delta}|D|,
\]
where $B$ is the disk centered at the origin with the same area as $D$. 
\end{cor}

\begin{proof}
In the proof of Theorem~\ref{thm_simply_connected}, combining the equation $\mathcal{I}=0$ and \eqref{I_eq2} together, we have that 
\[
\frac{1}{4\pi} |D|^2 -\Omega \int_D |x|^2 dx - (1-2\Omega) \int_D p dx=0.
\]
Dividing  both sides by $\Omega$ and rearranging the terms, we obtain
\[
\int_D |x|^2 dx - \frac{1}{2\pi} |D|^2 = \frac{1-2\Omega}{\Omega} \left( \frac{1}{4\pi}|D|^2 - \int_D p dx \right) \leq \frac{2\delta
|D|^2}{\pi},
\]
where in the inequality we used that $2\delta := 1-2\Omega$, $\Omega > \frac{1}{4}$, and $\int_D p dx \geq 0$.

Since $\int_B |x|^2 dx = \frac{1}{2\pi}|D|^2$, the above inequality implies that 
\begin{equation}\label{temp321}
\int_{D \setminus B} |x|^2 dx - \int_{B\setminus D} |x|^2 dx \leq \frac{2\delta |D|^2}{\pi}.
\end{equation}
Since $D$ and $B$ has the same area, let us denote $\beta := |D \setminus B| = |B \setminus D|$. Among all sets $U\subset B^c$ with area $\beta$, $\int_U |x|^2 dx$ is minimized when $U$ is an annulus with area $\beta$ and inner circle coinciding with $\partial B$, thus an elementary computation gives
\[
\int_{D \setminus B} |x|^2 dx \geq \inf_{U\subset B^c, |U|=\beta}  \int_U |x|^2 dx = \frac{\beta(2|B|+\beta)}{2\pi}.
\]
Likewise, among all sets $V\subset B$ with area $\beta$, $\int_V |x|^2 dx$ is maximized when $V$ is an annulus with area $\beta$ and outer circle coinciding with $\partial B$, thus
\[
\int_{B \setminus D} |x|^2 dx \leq \sup_{V\subset B, |V|=\beta}  \int_V |x|^2 dx = \frac{\beta(2|B|-\beta)}{2\pi}.
\]
Subtracting these two inequalities yields
\[
\int_{D \setminus B} |x|^2 dx - \int_{B\setminus D} |x|^2 dx\geq \frac{\beta^2}{\pi},
\]
and combining this with \eqref{temp321} immediately gives
\[
\beta^2 \leq 2\delta |D|^2,
\]
thus $|D\triangle B| = 2\beta \leq 2\sqrt{2\delta} |D|$.
\end{proof}


\subsection{Radial symmetry of non-simply-connected stationary patches}
\label{subsec_non_simply_connected}

In this subsection, we aim to prove radial symmetry of a connected rotating patch $D$ with $\Omega\leq 0$, where $D$ is allowed to be non-simply-connected. Let $D\subset \R^{2}$ be a bounded connected domain with $C^1$ boundary. Assume $D$ has $n$ holes with $n\geq 0$, and then let $h_1, \cdots, h_n \subset \mathbb{R}^2$ denote the $n$ holes of $D$ (each $h_i$ is a bounded open set). Note that $\partial D$ has $n+1$ connected components: they include the outer boundary of $D$, which we denote by $\partial D_{0}$, and the inner boundaries $\partial h_{i}$ for $i=1,...,n$.

To begin with, we point out that even for the steady patch case $\Omega = 0$, the proof of Theorem~\ref{thm_simply_connected} cannot be directly adapted to the non-simply-connected patch. If we define $\vec{v}$ in the same way, then the second way to compute $\mathcal{I}$ still goes through (since Proposition~\ref{talenti} still holds for non-simply-connected $D$), and leads to $\mathcal{I}>0$ if $D$ is not a disk. But the first way to compute $\mathcal{I}$ no longer gives $\mathcal{I}=0$: if $D$ is stationary and not simply-connected, $f(x):= (1_D*\mathcal{N})(x)$ may take different constant values on different connected components of $\partial D$, thus the identity \eqref{i=0} no longer holds.


In order to fix this issue, we still define $\vec v = -\nabla \varphi = -\nabla (\frac{|x|^2}{2} + p)$, but modify the definition of $p$ in the following lemma. Compared to the previous definition \eqref{def_p1}, the difference is that $p$ now takes different values $0,c_1,\dots, c_n$ on each connected component of $\partial D$. The lemma shows that there exist values of $\{c_i\}_{i=1}^n$, such that $\int_{\partial h_{i}} \nabla p\cdot \vec{n} d \sigma=-2|h_{i}|$ along the boundary of each hole. As we will see later, this leads to $\int_{\partial h_{i}} \vec{v}\cdot \vec{n} d \sigma=0 \text{ for } i=1,\dots,n,$ which ensures $\mathcal{I}=0$. (Of course, with $p$ defined in the new way, the second way of computing $\mathcal{I}$ no longer follows from Proposition~\ref{talenti}, and we will take care of this later in Proposition~\ref{talenti2}.)

\begin{lem}\label{p_def}

Let $D$, $h_i$ and $\partial D_0$ be given as in the first paragraph of Section~\ref{subsec_non_simply_connected}. 
Then there exist positive constants $\{c_{i}\}_{i=1}^n$, such that the solution $p:\overline{D}\to\mathbb{R}$ to the Poisson equation
\begin{align}\label{eq1}
\begin{cases}
\lap p=-2 & \text{ in $D$},\\
p= c_{i} &\text{ on $\partial h_{i}$~ for $i=1,\dots,n$},\\
p= 0 &\text{ on $\partial D_{0}$}.
\end{cases}
\end{align}
 satisfies 
\begin{equation}\label{eq_p_int_bdry}
\int_{\partial h_{i}} \nabla p\cdot \vec{n} ~d \sigma=-2|h_{i}| \quad\text{ for } i=1,\dots,n.
\end{equation}
 Here  $|h_i|$ is the area of the domain $h_i\subset \mathbb{R}^2$. 
\end{lem}
\begin{proof} 
Let $u$ satisfy that
\begin{align*}
\begin{cases}
\lap u=-2 &\text{ in $D$}\\
 u=0 &\text{ on $\partial D$}.
\end{cases}
\end{align*}
Furthermore let the function $v_{j}$ for $j=1,...,n$ be the solution to
\begin{align*}
\begin{cases}
\lap v_{j}=0 &\text{ in $D$}\\
v_{j}=0 &\text{ on $\partial D\setminus \partial h_{j}$}\\
v_{j}=1 &\text{ on $\partial h_{j}$}.
\end{cases}
\end{align*}
Now we consider the following linear equation,
\begin{align}
Ax=b, \label{eq2}
\end{align}
where $A_{i,j}=\int_{\partial h_{i}}\nabla v_{j}\cdot \vec{n}\,d\sigma$ and $b_{i}=-2|h_{i}|-\int_{\partial h_{i}}\nabla u\cdot \vec{n}\,d\sigma$. We argue that \eqref{eq2} has a unique solution. Thanks to the divergence theorem, we have  
\[
0=\int_{D}\lap v_{j}dx=\int_{\partial D_{0}}\nabla v_{j}\cdot\vec{n}d\sigma-\sum_{i=1}^{n}\int_{\partial h_{i}}\nabla v_{j}\cdot \vec{n}d\sigma.\]
 Therefore,
\[\sum_{i=1}^{n}A_{i,j}=\int_{\partial D_{0}}\nabla v_{j}\cdot \vec{n}d\sigma<0,
\]
where the last inequality follows from the Hopf Lemma since $v_{j}$ attains its minimum value 0 on $\partial D_{0}$, and $v_j\not\equiv 0$ on $\partial D$. A similar argument gives that $A_{i,j}>0$ for $i \neq j$ and $A_{j,j}<0$. Thus $A$ is invertible by Gershgorin circle theorem \cite{Gerschgorin:eigenvalues-theorem}, leading to a unique solution of \eqref{eq2}. Let us denote the solution by $x=(c_{1},...,c_{n})^{t}$. Then the function $p$ defined by 
\[
p:=u+\sum_{i=1}^{n}c_i v_{i}
\] 
satisfies the desired properties \eqref{eq_p_int_bdry}.

Now we prove that $c_{i}>0$ for $i\ge 1$. Suppose that $c_{i*}:=\min_{i}c_{i}\le 0$. Then by the minimum principle, $p$ attains its minimum on $\partial h_{i^{*}}$. Therefore, 
\[0\le \int_{\partial h_{i^{*}}}\nabla p\cdot \vec{n}d\sigma=-2|h_{i^{*}}|<0,\]
which is a contradiction.
\end{proof}

Next we prove a parallel version of Talenti's theorem for the function $p$ constructed in Lemma~\ref{p_def}. We will use this result throughout Section \ref{sec2}--\ref{sec3}.

\begin{prop}\label{talenti2} Let $D\subset \R^{2}$ be a bounded connected domain with $C^1$ boundary. Assume $D$ has $n$ holes with $n\geq 0$, and denote by $h_1, \cdots, h_n \subset \mathbb{R}^2$ the holes of $D$ (each $h_i$ is a bounded open set). Let $p:\overline{D}\to\mathbb{R}$ be the function constructed in Lemma~\ref{p_def}. Then the following two estimates hold:
\begin{equation}\label{estimate_p_general1}
\sup_{\overline{D}}{p}\le \frac{|D|}{2\pi}
\end{equation}
and
\begin{equation}\label{estimate_p_general}
\quad \int_{D}p(x)dx\le \frac{|D|^{2}}{4\pi}
\end{equation} Furthermore, for each of the two inequalities above, the equality is achieved if and only if $D$ is either a disk or an annulus. 
\end{prop}

\begin{proof} The proof is divided into two parts: In step 1 we prove the two inequalities \eqref{estimate_p_general1} and \eqref{estimate_p_general}, and in step 2 we show that equality can be achieved if and only if $D$ is a disk or an annulus. 
 
\textbf{Step 1.} When $D$ is simply-connected, \eqref{estimate_p_general1} and \eqref{estimate_p_general} directly follow from Talenti's theorem Proposition~\ref{talenti}. Next we consider a non-simply-connected domain $D$, and prove that these inequalities also hold when $p:\overline{D}\to\mathbb{R}$ is defined as in Lemma~\ref{p_def}.

For $k\in\mathbb{R}^+$, let us denote $D_{k}:=\left\{ x\in D: \ p(x)>k\right\}$,  $g(k):=|D_{k}|$ and $\tilde{D}_{k}:=D_{k}\dot{\cup}(\dot{\cup}_{\{i: c_{i}>k\}}\overline{h_{i}})$.  Elliptic regularity theory gives that $p\in C^\infty(D)$, thus by Sard's theorem, $k$ is a regular value for almost every $k\in (0,\sup_{D}p)$, that is, $|\nabla p(x)|>0$ on $\left\{ x\in D: \ p(x)=k\right\}$. Thus $\left\{ x\in D: \ p(x)=k\right\}$ is a union of smooth simple closed curves and equal to $\partial \tilde{D}_{k}$ for almost every $k\in (0,\sup_{D}p)$.

\quad Since $\partial D_{k}=\partial \tilde{D}_{k}  \dot{\cup} (\dot{\cup}_{\{i: c_{i}>k\}}\partial h_{i})$ for $k\notin\left\{c_{1},...,c_{n}\right\}$, we compute
\begin{align*}
g(k)&=-\frac{1}{2}\int_{D_{k}}\lap p(x) dx~=-\frac{1}{2}\int_{\partial D_{k}}\nabla p \cdot \vec{n}d\sigma\\
&=-\frac{1}{2}\int_{\partial \tilde{D}_{k}}\nabla p\cdot \vec{n}d\sigma+\frac{1}{2}\sum_{\{i: c_{i}>k\}}\int_{\partial h_{i}}\nabla p\cdot \vec{n}d\sigma\\
&=-\frac{1}{2}\int_{{\partial \tilde{D}}_{k}} \nabla p\cdot \vec{n}d\sigma-\sum_{\{i: c_{i}>k\}}|h_{i}|,
\end{align*}
where the last identity is due to \eqref{eq_p_int_bdry}. Therefore, it follows that 
\begin{align}
g(k)+\sum_{\{i: c_{i}>k\}}|h_{i}|=-\frac{1}{2}\int_{{\partial \tilde{D}}_{k}}\nabla p\cdot \vec{n}d\sigma=\frac{1}{2}\int_{{\partial \tilde{D}}_{k}}|\nabla p|d\sigma,\label{eq3}
\end{align}
where the last equality follows from the fact that $\nabla p$ is perpendicular to the tangent vector on the level set.\\
 On the other hand, the coarea formula yields that
\begin{align*}
g(k)=\int_{\R}\int_{{\partial \tilde{D}}_{s}}1_{D_{k}}\frac{1}{|\nabla p|}d\sigma ds=\int_{k}^{\infty}\int_{{\partial \tilde{D}}_{s}}\frac{1}{|\nabla p|}d\mathcal{\sigma} ds.
\end{align*}
Therefore, it follows that for almost every $k\in (0,\sup_{D}p)$, 
\begin{align}
g'(k)=-\int_{{\partial \tilde{D}}_{k}}\frac{1}{|\nabla p|}d\sigma.\label{eq4}
\end{align}
Thus it follows from \eqref{eq3} and \eqref{eq4} that
\begin{align}\label{eq28}
g'(k)\left(g(k)+\sum_{\{i: c_{i}>k\}}|h_{i}|\right)=-\frac{1}{2}\left(\int_{{\partial \tilde{D}}_{k}}|\nabla p| d\sigma \right)\left(\int_{{\partial \tilde{D}}_{k}}\frac{1}{|\nabla p|}d\sigma \right)\le -\frac{1}{2}P({\tilde{D}}_{k})^2,
\end{align}
where $P(E)$ denotes the perimeter of a rectifiable curve $\partial E$. Note that the last inequality becomes equality if and only if $|\nabla p|$ is a constant on $\partial \tilde{D}_{k}$. Also, the isoperimetric inequality gives that
\begin{equation}\label{isoperimetric}
P(\tilde{D}_{k})^{2}\ge 4\pi |\tilde{D}_{k}|,
\end{equation}
where equality holds if and only if $\tilde{D}_{k}$ is a disk.  This yields that
\begin{align}
g'(k)\left(g(k)+\sum_{\{i: c_{i}>k\}}|h_{i}|\right)\le -2\pi |\tilde{D}_{k}|=-2\pi \left(g(k)+\sum_{\{i: c_{i}>k\}}|h_{i}|\right). \label{eq5}
\end{align}
Therefore, $g'(k)\le -2\pi$ for almost every $k\in (0,\sup_{D}p)$. Combining it with the fact that $g(0)=|D|$, we have
\[
g(k)\le (g(0)-2\pi k)_{+}=(|D|-2\pi k)_{+} \text{ for almost every }k\geq 0.
\] This proves that $\sup_{\bar{D}}p\le \frac{|D|}{2\pi}$. It follows that
\begin{align*}
\int_{D}p(x)dx=\int_{D}\int_{0}^{\frac{|D|}{2\pi}}1_{\left\{k< p(x)\right\}}dkdx= \int_{0}^{\frac{|D|}{2\pi}}g(k)dk\le\int_{0}^{\frac{|D|}{2\pi}}(|D|-2\pi k)_{+}dx= \frac{|D|^2}{4\pi}.
\end{align*}

\textbf{Step 2.}  Now we show that for the two inequalities \eqref{estimate_p_general1} and \eqref{estimate_p_general}, the equality is achieved if and only if $D$ is either a disk or an annulus. First, if $D$ is either a disk or annulus centered at some $x_0\in\mathbb{R}^2$, then uniqueness of solution to Poisson's equation gives that $p$ is radially symmetric about $x_0$. Since we have $\Delta p = -2$ in $D$ and $p=0$ on the outer boundary of $D$, this gives an explicit formula $p(x) = -\frac{|x-x_0|^2}{2}+\frac{R^2}{2}$ for $x\in D$, where $R$ is the outer radius of $D$. For either a disk or an annulus, one can explicitly compute $\sup_D p$ and $\int_D p dx$ to check that equalities in \eqref{estimate_p_general1} and \eqref{estimate_p_general} are achieved.  

To prove the converse, assume that either \eqref{estimate_p_general1} or \eqref{estimate_p_general} achieves equality, and we aim to show that $D$ is either a disk or an annulus. In order for either equality to be achieved, \eqref{eq5} needs to achieve equality at almost every $k\in (0,\sup_D p)$. In addition, $g(k)$ needs to be continuous in $k$  since $g(k)$ is decreasing. Since \eqref{eq5} follows from a combination of the Cauchy-Schwarz inequality in \eqref{eq28} and the isoperimetric inequality in \eqref{isoperimetric}, we need to have all the three conditions below in order for either \eqref{estimate_p_general1} or \eqref{estimate_p_general} to achieve equality:

(1) $|\nabla p|$ is a constant on each level set $\partial \tilde{D}_{k}$ for almost every $k\in (0,\sup_D p)$; 

(2) $\tilde{D}_{k}$ is a disk for almost every $k\in (0,\sup_D p)$.

(3) $g(k) = |D_k|$ is continuous in $k$. As a result, $|\tilde D_k|$ is continuous in $k$ at all $k\neq c_i$, with $c_i>0$ defined in \eqref{eq1}.

Next we will show that if all these three conditions are satisfied, then $D$ must be an annulus or disk. First, note that by sending $k\searrow 0$ in condition (2), and combining it with the continuity of $|\tilde D_k|$ as $k\searrow 0$, it already gives that the outer boundary of $D$ must be a circle. Therefore if $D$ is simply-connected, it must be a disk.

If $D$ is non-simply-connected, using condition (2) and (3), we claim that $D$ can have only one hole, which must be a disk, and $p$ must achieve its maximum value in $\bar D$ on the boundary of the hole. To see this, let $h_i$ be any hole of $D$, and recall that $p|_{\partial h_i} = c_i$. As we consider the set limit of $\tilde D_k$ as $k$ approaches $c_i$ from below and above, by definition of $\tilde D_k$ we have 
\[
\lim_{k\nearrow c_i} \tilde D_k = \lim_{k\searrow c_i} \tilde D_k \dot{\cup}(\dot{\cup}_{\{j: c_j = c_i\}}\overline{h_{j}}).
\]
By (2) and (3), the left hand side $\lim_{k\nearrow c_i} \tilde D_k $ is a disk, and the set $\lim_{k\searrow c_i} \tilde D_k$ on the right hand side is also a disk (if the limit is non-empty). But after taking union with the holes $\{h_j: c_j = c_i\}$ (each is a simply-connected set), the right hand side will be a disk if and only if $\lim_{k\searrow c_i} \tilde D_k$ is empty, $\dot{\cup}_{\{j: c_j = c_i\}}\overline{h_{j}} = \overline{h_i}$, and $h_i$ is a disk. This implies $c_i= \sup_D p$ and $c_j < c_i$ for all $j\neq i$. But since $h_i$ is chosen to be any hole of $D$, we know $D$ can have only one hole (call it $h$), which is a disk, and $\sup_D p = p|_{\partial h}$. Finally, note that condition (1) gives that all the disks $\{\tilde D_k\}$ are concentric, and as a result we have $D$ is an annulus, finishing the proof. 
\end{proof}

Finally, we are ready to show that every connected stationary patch $D$ with $C^1$ boundary must be either a disk or an annulus. 

\begin{thm}\label{theorem1} Let $D\subset \mathbb{R}^2$ be a bounded domain with $C^1$ boundary. Suppose that $\omega(x):=1_{D}(x)$ is a stationary patch solution to the 2D Euler equation in the sense of \eqref{eq_stat}. Then $D$ is either a disk or an annulus.
\end{thm}
\begin{proof}
 If $D$ has $n$ holes (where $n\geq 0$), denote them by $h_1,\dots,h_n$.  By \eqref{eq_stat},  the function $f:= 1_D * \mathcal{N}$ is constant on each of connected component of $\partial D$, and let us denote 
\begin{equation}\label{eq_f_const}
f(x)=
\begin{cases}
a_{i} & \text{ on $\partial h_{i}$}\\
a_{0} & \text{ on $\partial D_{0}$}.
\end{cases}
\end{equation}
Let $p:\overline{D}\to\mathbb{R}$ be defined as in Lemma~\ref{p_def}, and let $\vphi := \frac{|x|^2}{2} + p$. 
Similar to the proof of Theorem~\ref{thm_simply_connected}, we calculate $\mathcal{I}:=\int_{D}\nabla \vphi \cdot \nabla fdx$ in two different ways. Note that $\nabla f = \nabla (f-a_0)$ in $D$. Applying the divergence theorem to $\mathcal{I}$ and using \eqref{eq_f_const} and $\Delta\varphi=0$ in $D$, it follows that
\begin{equation}\label{temp111}
\begin{split}
\mathcal{I}&=\int_{\partial D}(\nabla \vphi\cdot \vec{n})(f-a_0) d\sigma-\int_{D}\lap \vphi (f-a_0) dx=-\sum_{i=1}^{n}(a_{i} - a_0) \int_{\partial h_{i}}\nabla\vphi \cdot \vec{n}d\sigma.
\end{split}
\end{equation}

By definition of $\varphi$, and combining it with the property of $p$ in \eqref{eq_p_int_bdry}, we have
\begin{equation}\label{temp222}
\begin{split}
\int_{\partial h_{i}}\nabla \vphi\cdot \vec{n}d\sigma&=\int_{\partial h_{i}}\nabla \left(\frac{|x|^{2}}{2}\right)\cdot \vec{n}d\sigma+\int_{\partial h_{i}} \nabla p\cdot \vec{n}d\sigma\\
&=\int_{h_{i}}2dx+\int_{\partial h_{i}} \nabla p\cdot \vec{n}d\sigma =0.
\end{split}
\end{equation}
Plugging this into \eqref{temp111} gives $\mathcal{I}=0$. On the other hand, we also have
\begin{align*}
\mathcal{I}=&\int_{D} x\cdot \nabla fdx+\int_{D} \nabla p\cdot \nabla fdx=:E_{1}+E_{2}.
\end{align*}
 We compute
\begin{align}
E_{1}=\int_{D}x\cdot (1_D\ast \nabla \mathcal{N})dx&=\int_{D}\int_{D} \frac{1}{2\pi}\frac{x\cdot(x-y)}{|x-y|^{2}}dydx =\frac{|D|^{2}}{4\pi},\label{eq9}
\end{align} 
where the last equality is obtained by exchanging $x$ with $y$ and taking the average with the original integral.
For $E_{2}$, the divergence theorem yields that
\begin{align*}
E_{2}&=\int_{\partial D}p \nabla f\cdot \vec{n}d\sigma-\int_{D}p\lap fdx =\int_{\partial D}p \nabla f\cdot \vec{n}d\sigma-\int_{D}pdx.
\end{align*}
Using the property of $p$ in \eqref{eq1} and the fact that $\Delta f = 0$ in $h_i$, the divergence theorem yields
\begin{equation}\label{eq8}
\int_{\partial D}p\nabla f\cdot \vec{n}d\sigma=-\sum_{i=1}^{n}\int_{\partial h_{i}}p\nabla f\cdot \vec{n}d\sigma 
=-\sum_{i=1}^{n}c_{i}\int_{h_{i}}\Delta f dx =0, 
\end{equation}
As a result, we have $E_2 = -\int_D p dx$. If $D$ is neither a disk nor an annulus,  Proposition~\ref{talenti2} gives
\begin{align*}
\mathcal{I}=E_1 + E_2 =\frac{|D|^{2}}{4\pi}-\int_{D}pdx >0,
\end{align*}
contradicting $\mathcal{I}=0$.
\end{proof}

In the next corollary, we generalize the above result to a nonnegative stationary patch with multiple (disjoint) patches. 

\begin{cor}\label{cor_multiple_steady}
Let $\omega(x):=\sum_{i=1}^n \alpha_i1_{D_i}$, where $\alpha_i>0$, each $D_{i}$ is a bounded connected domain with $C^1$ boundary, and $D_{i}\cap D_{j}=\emptyset$ if $i\ne j$. Assume that $\omega$ is a stationary patch solution, that is, the function $f(x) := \omega * \mathcal{N}$ satisfies $\nabla^\perp f \cdot \vec{n} = 0$ on $\partial D_i$ for all $i = 1,\dots,n$. Then $\omega$ is radially symmetric up to a translation. 
\end{cor}

\begin{proof}
Following similar notations as the beginning of Section~\ref{subsec_non_simply_connected}, we denote the outer boundary of $D_i$ by $\partial D_{i0}$, and the holes of each $D_i$ (if any) by $h_{ik}$ for $k=1,\dots, N_i$.  Let $p_i:\overline{D_i} \to \mathbb{R}$ be defined as in Lemma~\ref{p_def}, that is, $p_i$ satisfies
\begin{align*}
\begin{cases}
\lap p_i=-2 &\text{ in $ D_i$}\\
p_i=c_{ik} &\text{ on $\partial h_{ik}$}\\
p_i=0 & \text{ on $\partial D_{i0}$},
\end{cases}
\end{align*}
where $c_{ik}$ is chosen such that $\int_{\partial h_{ik}}\nabla p_i\cdot \vec{n}d\sigma=-2|h_{ik}|$.  We then define $\displaystyle \vphi:\cup_{i=1}^n \overline D_i \to \mathbb{R}$, such that in each $\overline D_i$ we have $\vphi =\vphi_i := \frac{|x|^2}{2}+p_i$. 

 Similar to Theorem~\ref{theorem1}, we compute 
 \[
 \mathcal{I}:=\int_{\mathbb{R}^2} \omega \nabla \vphi \cdot \nabla f dx= \sum_{i=1}^{n}\int_{D_i}\alpha_i\nabla \vphi_i\cdot \nabla fdx
 \]
  in two different ways.  On the one hand, since $f=\omega*\mathcal{N}$ is a constant on each connected component of $\partial D_i$, the same computation of Theorem~\ref{theorem1} yields that $\int_{D_i} \nabla \vphi_i\cdot \nabla fdx=0$, therefore $\mathcal{I}=0$.\\
  On the other hand, since $\nabla \vphi = x + \nabla p_i$ in each $D_i$, we break $\mathcal{I}$ into
\[ \mathcal{I}
 =\sum_{i,j=1}^{n}\alpha_i\alpha_j\int_{D_i}x\cdot \nabla(1_{D_j}*\mathcal{N})dx+\sum_{i,j=1}^{n}\alpha_i\alpha_j\int_{D_i}\nabla p_i\cdot \nabla(1_{D_j}*\mathcal{N})dx=:\mathcal{I}_1+\mathcal{I}_2.
\]
For $\mathcal{I}_1$, we compute
\begin{align}\label{I1_multiple}
\mathcal{I}_1&= \sum_{i,j=1}^{n} \frac{\alpha_i\alpha_j}{2}\left(\int_{D_i}x\cdot \nabla(1_{D_j}*\mathcal{N})dx + \int_{D_j}x\cdot \nabla(1_{D_i}*\mathcal{N})dx\right)\nonumber\\
&=\sum_{i,j=1}^{n} \frac{\alpha_i\alpha_j}{2}\left(\int_{D_i}\int_{D_j}\frac{x\cdot (x-y)}{2\pi|x-y|^2}dydx+\int_{D_j}\int_{D_i}\frac{x\cdot (x-y)}{2\pi|x-y|^2}dydx\right)\nonumber\\
&=\sum_{i,j=1}^{n} \frac{\alpha_i\alpha_j}{4\pi}|D_i||D_j|,
\end{align}  
where we exchanged $i$ with $j$ to get the first equality. For $\mathcal{I}_2$, we have
\[
\mathcal{I}_2=\sum_{i=1}^{n}\alpha_i^2\int_{D_i}\nabla p_i\cdot \nabla(1_{D_i}*\mathcal{N})dx+\sum_{i\ne j}\alpha_i\alpha_j\int_{D_i}\nabla p_i\cdot \nabla(1_{D_j}*\mathcal{N})dx=:I_{21}+I_{22}.
\]
By the same computation for $E_2$ in the proof of Theorem~\ref{theorem1}, we have
\begin{align}\label{I21_multiple}
I_{21}=-\sum_{i=1}^{n}\alpha_i^2\int_{D_i}p_i dx.
\end{align}

For $i\neq j$, we denote $j\prec i$ if $D_j$ is contained in a hole of $D_i$. (And if $D_j$ is not contained in any hole of $D_i$, we say $j\not\prec i$.) Using this notation, the divergence theorem directly yields that 
\begin{equation}\label{int_is_0*}
\int_{\partial D_{i}}p_{i}\nabla (1_{D_{j}}*\mathcal{N}) \cdot \vec n d\sigma = -\sum_{k=1}^{N_i} \int_{\partial h_{ik}}p_{i}\nabla (1_{D_{j}}*\mathcal{N}) \cdot \vec n d\sigma = 0 \quad \text{ if } j \not\prec i.
\end{equation}
And if $j\prec i$, then the divergence theorem and \eqref{estimate_p_general1} in Proposition~\ref{talenti2} yield 
\begin{equation}\label{crude_bd*}
\int_{\partial D_{i}}p_{i}\nabla (1_{D_{j}}*\mathcal{N}) \cdot \vec n d\sigma \ge -\sup_{\partial D_{i}}p_{i}|D_{j}| \geq -\frac{1}{2\pi} |D_i| |D_j|  \quad \text{ if } j \prec i.
\end{equation}
Hence it follows that
\begin{align}\label{I22_multiple}
I_{22}\ge -\sum_{i\neq j}\mathbbm{1}_{j\prec i}\frac{\alpha_i\alpha_j}{2\pi}|D_i||D_j| = -\sum_{i\neq j}(\mathbbm{1}_{j\prec i} + \mathbbm{1}_{i\prec j}) \frac{\alpha_i\alpha_j}{4\pi}|D_i||D_j|,
\end{align}
where the last step is obtained by exchanging $i,j$ and taking average with the original sum. Note that we have $\mathbbm{1}_{j\prec i} + \mathbbm{1}_{i\prec j}\leq 1$ for any $i\neq j$.
From \eqref{I1_multiple}, \eqref{I21_multiple} and \eqref{I22_multiple}, we obtain
\begin{equation}\label{eq53}
\mathcal{I} \ge \sum_{i=1}^{n}\alpha_i^2\left(\frac{|D_i|^2}{4\pi}-\int_{D_i}p_idx\right)+
\sum_{\underset{i\neq j}{j\not\prec i\text{ and }i\not\prec j}}\alpha_i\alpha_j\frac{|D_i||D_j|}{4\pi}.
\end{equation} 
Since we already know that $\mathcal{I}=0$ and all the summands in \eqref{eq53} are nonnegative, it follows that
\begin{equation*}
\frac{|D_i|^2}{4\pi}=\int_{D_i}p_idx \text{ for all $i=1,\ldots,n$} \text{ and }\left\{(i,j) : i\neq j, i\not\prec j \text{ and } j\not\prec i\right\}=\emptyset.
\end{equation*}
Therefore every $D_i$ is either a disk or an annulus by Proposition~\ref{talenti2} and they are nested. By relabeling the indices, we can assume that  $i\prec i+1$ for $i=1,\ldots ,n-1$.


Next we prove that all $D_i$'s are concentric by induction. For $k\geq 1$, suppose $D_1,\dots, D_k$ are known to be concentric about some $o\in\mathbb{R}^2$. To show $D_{k+1}$ is also centered at o, we break $f$ into
 \[
 f = \sum_{i=1}^{k}(\alpha_i1_{D_i})*\mathcal{N} + \sum_{i=k+1}^{n}(\alpha_i1_{D_i})*\mathcal{N}.
  \]
In the first sum, each $D_i$ is centered at $o$ for $i\leq k$, thus Lemma~\ref{lem_newton}(a) (which we prove right after this theorem) yields that $\sum_{i=1}^{k}(\alpha_i1_{D_i})*\mathcal{N} = \frac{C}{2\pi}\ln|x-o|$ on $\inb D_{k+1} $, where $C=\sum_{i=1}^k \alpha_i |D_i|>0$. In the second sum, for each $i\geq k+1$, since each $D_i$ is an annulus with $\inb D_{k+1}$ in its hole, Lemma~\ref{lem_newton}(b) gives that $1_{D_i}*\mathcal{N} \equiv \text{const}$ on  $\inb D_{k+1} $ for all $i\geq k+1$. Thus overall we have
$f = \frac{C}{2\pi}\ln|x-o| + C_2$ on $\inb D_{k+1}$ for $C>0$. Combining it with the assumption that $f$ is a constant on $ \inb D_{k+1}$, we know $D_{k+1}$ must also be centered at $o$, finishing the induction step.
 \end{proof}
 
 Now we state and prove the lemma used in the proof of Corollary~\ref{cor_multiple_steady}, which follows from standard properties of the Newtonian potential.  
 \begin{lem}\label{lem_newton}
 Assume $g\in L^\infty(\mathbb{R}^2)$ is radially symmetric about some $o\in\mathbb{R}^2$, and is compactly supported in $B(o,R)$. Then $\eta := g*\mathcal{N}$ satisfies the following:
 
 \vspace*{-0.2cm}
 \begin{enumerate}
 \item[(a)]
$ \displaystyle \eta(x) = \frac{\int_{\mathbb{R}^2} g dx }{2\pi} \ln|x-o|\,$ for all $x\in B(0,R)^c.
$ 
\smallskip
\item[(b)] If in addition we have $g\equiv 0$ in $B(o,r)$ for some $r\in(0,R)$, then 
$ \eta=\text{const}$ in $B(o,r)$.
 
 \end{enumerate}
 \end{lem}
 
 \begin{proof}
To show (a), we take any $x\in B(o,R)^c$ and consider the circle $\Gamma \ni x$ centered at $o$. By radial symmetry of $\eta$ about $o$ and the divergence theorem, we have
 \[
\nabla \eta \cdot  \frac{x}{|x|}= \frac{1}{|\Gamma|} \int_{\Gamma} \nabla \eta \cdot \vec{n} d\sigma = \frac{1}{|\Gamma|} \int_{\interior(\Gamma)} \Delta \eta dx =\frac{ \int_{\mathbb{R}^2} g(x) dx}{2\pi|x-o|},
 \]
which implies $ \eta(x) =\frac{\int g dx }{2\pi} \ln|x-o|  + C$. To show that $C=0$, for $|x|$ sufficiently large we have 
\[
|C| = \Big|  \int_{B(o,R)} g(x) (\mathcal{N}(x-y)  - \mathcal{N}(x-o)) dy\Big| \leq \|g\|_{L^\infty(\mathbb{R}^2)} \sup_{y\in B(o,R)} |\mathcal{N}(x-o)-\mathcal{N}(x-y)|,
\] and by sending $|x|\to\infty$ we have $C=0$, which gives (a). To show (b), it suffices to prove that $\nabla \eta \equiv 0$ in $B(o,r)$. Take any $x\in B(o,r)$, and consider the circle $\Gamma_2 \ni x$ centered at $o$. Again, symmetry and the divergence theorem yield that \[
|\nabla \eta(x)| = \frac{1}{|\Gamma_2|}\int_{\Gamma_2} \nabla \eta \cdot \vec{n} d\sigma = \frac{1}{|\Gamma_2|} \int_{\interior(\Gamma_2)} \Delta \eta dx =  \frac{\int_{\interior(\Gamma)} g(x) dx}{|\Gamma|}  = 0,\] 
finishing the proof of (b).
 \end{proof}


\subsection{Radial symmetry of non-simply-connected rotating patches with $\Omega<0$}

In this subsection, we show that a nonnegative uniformly rotating patch solution (with multiple disjoint patches) must be radially symmetric if the angular velocity $\Omega<0$. 

\begin{thm}\label{thm_omega<0}
For $i=1,\dots,n$, let $D_i$ be a connected domain with $C^1$ boundary, and assume $D_i\cap D_j=\emptyset$ for $i\ne j$. If $\omega =\sum_{i=1}^n \alpha_i 1_{D_i}$ is a nonnegative rotating patch solution with $\alpha_i > 0$ and angular velocity $\Omega < 0$, then $\omega$ must be radially symmetric.
\end{thm}

\begin{proof}
In this proof, let 
\[
f_\Omega(x):=\omega*\mathcal{N}-\frac{\Omega}{2}|x|^2.
\] In each $D_i$, let us define $p_i$  as in Lemma~\ref{p_def}. Let $\vphi_i:=\frac{|x|^2}{2}+p_i$ in each $D_i$. As in Theorem~\ref{thm_omega<0}, we compute $\mathcal{I}:=\sum_{i=1}^{n}\alpha_i\int_{D_i}\nabla \vphi_i\cdot \nabla f_\Omega dx$ in two different ways. Since $f_\Omega$ is a constant on each connected component of $\partial D_i$ and $\nabla \vphi_i$ is divergence free in $D_i$, we still have $\mathcal{I}=0$ as in the proof of Theorem~\ref{theorem1}.

 On the other hand, we have
 \begin{align*}
 \mathcal{I}&=\sum_{i=1}^n\alpha_i\int_{D_i} (x+\nabla p_i)\cdot \nabla \left(\omega*\mathcal{N}\right)dx + \underbrace{(-\Omega)}_{\geq 0} \sum_{i=1}^n\alpha_i\int_{D_i} (x+\nabla p_i) \cdot x dx\\
&=:\mathcal{I}_1+(-\Omega) \mathcal{I}_2.
 \end{align*}
 As in the proof of Corollary~\ref{cor_multiple_steady}, we have
 \begin{equation}\label{ineq_i1_rotate}
 \mathcal{I}_1=\sum_{i=1}^{n}\alpha_i^2\left(\frac{|D_i|^2}{4\pi}-\int_{D_i}p_idx\right)+ \sum_{\underset{i\neq j}{j\not\prec i\text{ and }i\not\prec j}} \alpha_i\alpha_j\frac{|D_i||D_j|}{4\pi} \geq 0.
 \end{equation}
Note that $\mathcal{I}_1 = 0$ as long as all $D_i$'s are nested annuli/disk, even if they are not concentric. 
 For $\mathcal{I}_{2}$, using Cauchy-Schwarz inequality in the second step, and Lemma~\ref{p_bd} in the third step (which we will prove right after this theorem), we have
\begin{equation}\label{ineq_i2_rotate}
\begin{split}
\mathcal{I}_{2}&= \sum_{i=1}^n \alpha_i \left( \int_{D_i} |x|^2 dx +  \int_{D_i} \nabla p_i \cdot x dx \right)\\
&\geq  \sum_{i=1}^n \alpha_i \left(  \int_{D_i} |x|^2 dx  -\Big(\int_{D_i} |\nabla p|^2 dx \Big)^{1/2} \Big(\int_{D_i} |x|^2 dx \Big)^{1/2}\right) \geq 0.
\end{split}
\end{equation}
Combining \eqref{ineq_i1_rotate} and \eqref{ineq_i2_rotate} gives us $\mathcal{I}\geq 0$. If there is any $D_i$ that is not a disk or annulus centered at the origin, Lemma~\ref{p_bd} would give a strict inequality in the last step of \eqref{ineq_i2_rotate},
which leads to $\mathcal{I}>0$ and thus contradicts with $\mathcal{I}=0$. 
\end{proof}

Now we state and prove the lemma that is used in the proof of Theorem~\ref{thm_omega<0}.

\begin{lem}\label{p_bd}Let $D$ be a connected domain with $C^1$ boundary, and let $p$ be as in Lemma~\ref{p_def}. Then we have
\begin{equation}\label{p_vs_x2}
-\int_{D}\nabla p\cdot xdx = \int_{D}|\nabla p|^{2}dx\le \int_{D} |x|^{2}dx.
\end{equation}
Furthermore, in the inequality, ``='' is achieved if and only if $D$ is a disk or annulus centered at the origin.
\end{lem}
\begin{proof}
We compute
\begin{align*}
\int_{D}|\nabla p|^{2}dx&=\int_{\partial D}p\nabla p\cdot \vec{n}d\sigma+\int_{D}2pdx\\
&=-\int_{\partial D} p x\cdot \vec{n}d\sigma +\int_{D}2pdx,
\end{align*}
where in the last equality we use that $p$ is constant along each $\partial h_i$, as well as the following identity due to \eqref{eq_p_int_bdry} and the divergence theorem (here $\vec n$ is the outer normal of $h_i$):
\[
\int_{\partial h_i} \nabla p\cdot \vec{n}d\sigma=-2|h_i| =-\int_{h_i}\lap \frac{|x|^{2}}{2}dx =-\int_{\partial h_i}x\cdot \vec{n}d\sigma.
\] On the other hand, the divergence theorem yields
\begin{align*}
-\int_{D} \nabla p\cdot x dx=-\int_{\partial D}px\cdot \vec{n}d\sigma+\int_{D}2pdx.
\end{align*} 
Therefore using Young's inequality $-\nabla p\cdot x \leq \frac{1}{2}|\nabla p|^2 +  \frac{1}{2}|x|^2$ (where the equality is achieved if and only if $-\nabla p = x$), we have
\begin{align*}
\int_{D}|\nabla p|^{2}dx&=-\int_{D}\nabla p\cdot xdx\leq \frac{1}{2}\int_{D}|\nabla p|^{2}dx +\frac{1}{2}\int_{D}|x|^{2}dx,
\end{align*}
which proves \eqref{p_vs_x2}. Here the equality is achieved if and only if $-\nabla p = x$ in $D$, which is equivalent with $p + \frac{|x|^2}{2}$ being a constant in $D$, and it can be extended to $\bar D$ due to continuity of $p$. By our construction of $p$ in Lemma~\ref{p_def}, $p$ is already a constant on each connected component of $\partial D$, implying $\frac{|x|^2}{2}$ is constant on each piece of $\partial D$, hence $\partial D$ must be a family of circles centered at the origin. By the assumption that $D$ is connected, it must be either a disk or annulus centered at the origin.\end{proof}

\subsection{Radial symmetry of non-simply-connected rotating patches with $\Omega \geq \frac{1}{2}$}

In this final subsection for patches, we consider a bounded domain $D$ with $C^1$ boundary. $D$ can have multiple connected components, and each connected component can be non-simply-connected. If $1_D$ is a rotating patch solution to the Euler equation with angular velocity $\Omega\ge\frac{1}{2}$, we will show $D$ must be radially symmetric and centered at the origin.

To do this, one might be tempted to proceed as in Theorem~\ref{thm_simply_connected} and replace $p:D\to \mathbb{R}$ by the function defined in Lemma~\ref{p_def}. Here the first way of computing $\mathcal{I} = \int_D (x + \nabla p) \cdot \nabla f_\Omega dx$ still yields $\mathcal{I}=0$, but the second way gives some undesired terms caused by the holes $h_i$:
\[
\mathcal{I} = \frac{1}{4\pi} |D|^2 - \Omega\int_D |x|^2 dx + (2\Omega - 1) \int_D p dx + 2\Omega \sum_{i=1}^n p|_{\partial h_i} |h_i|.
\]
Due to the last term on the right hand side, we are unable to show $\mathcal{I}\leq 0$ when $\Omega \geq \frac{1}{2}$ as we did before in Theorem~\ref{thm_simply_connected}. For this reason, we take a different approach in the next theorem. Instead of defining $p$ as a function in $D$ and $\mathcal{I}$ as an integral in $D$, we want to define them in $D^c$. But since $D^c$ is unbounded, we define $p^R$ and $\mathcal{I}_R$ in a truncated set $B(0,R)\setminus D^c$, and then use two different ways to compute $\mathcal{I}_R$. By sending $R\to\infty$, we will show that the two ways give a contradiction unless $D$ is radially symmetric.

 \begin{thm}\label{clockwise rotating_patch}
  For a bounded domain $D$ with $C^1$ boundary, assume that $1_D$ is a rotating patch solution to the Euler equation with angular velocity $\Omega\ge\frac{1}{2}$. Then $D$ is radially symmetric and centered at the origin. 
 \end{thm}
 \begin{proof}
 Since $D$ is bounded, let us choose $R_0>0$ such that $B_{R_0}\supset D$. For any $R>R_0$, consider the domain $B_R \setminus \overline{D}$, which may have multiple connected components. We call the component touching $\partial B_R$ as  $D_{0,R}$, and name the other connected components by $U_1, \dots, U_n$. Throughout this proof we assume that $n\geq 1$: if not, then each connected component of $D$ is simply connected, which has been already treated in Theorem~\ref{thm_simply_connected} and Remark~\ref{rmk_disconnected}.  We also define $V := B_R \setminus D_{0,R}$, which is the union of $D$ and all its holes. Note that $V$ may have multiple connected components, but each must be simply-connected. See Figure~\ref{fig_inverse} for an illustration of $D_{0,R}$, $\{U_i\}_{i=1}^n$ and $V$.
\begin{figure}[h!]
\begin{center}
\includegraphics[scale=0.9]{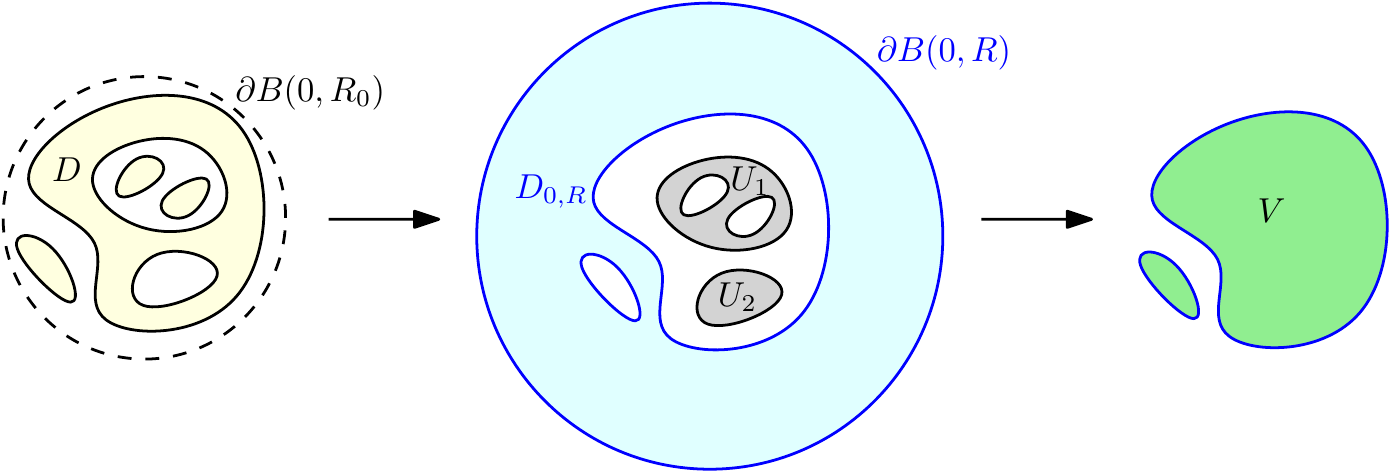}
\caption{For a set $D \subset B(0,R_0)$ (the whole yellow region on the left), the middle figure illustrates the definition of $D_{0,R}$ (the blue region), $\{U_i\}$ (the gray regions), and right right figure illustrate $V = B_R \setminus D_{0,R}$ (the green region). \label{fig_inverse}}
\end{center}
\end{figure}

To prove the theorem, the key idea is to define $p^R$ and $\mathcal{I}_R$ in $B_R \setminus D$, instead of in $D$.  Let $p_{0,R}$ and $p_i$ be defined as in Lemma~\ref{p_def} in $D_{0,R}$ and $U_i$ respectively, then set $\vphi_{0,R}:=p_{0,R}+\frac{|x|^2}{2}$ in $D_{0,R}$, and $\vphi_i:=p_i +\frac{|x|^2}{2}$ in $U_i$ for $i=1,\dots, n$. Finally, define $p^R$ and $\vphi^R: \mathbb{R}^2\to\mathbb{R}$ as
\[
p^R:=p_{0,R}1_{D_{0,R}}+\sum_{i=1}^{n}p_i1_{U_i},\quad\quad \vphi^R:=\vphi_{0,R}1_{D_{0,R}}+\sum_{i=1}^{n}\vphi_{i}1_{U_i}.
\]

Since $1_D$ rotates with angular velocity $\Omega\geq \frac{1}{2}$, we know $f_\Omega:=1_D*\mathcal{N}-\frac{\Omega}{2}|x|^2$ is constant on each connected component of $\partial D$.  Next we will compute 
\begin{align}
\mathcal{I}_{R}:=\int_{B_R\setminus \overline{D}}\nabla f_\Omega \cdot \nabla \vphi^R dx
\end{align}
in two different ways. If some connected component of $\partial D$ is not a circle, we will derive a contradiction by sending $R\to \infty$. We point out that as we increase $R$, the domain $D_{0,R}$ will change, but the domains $\{U_i\}_{i=1}^n$ and $V$ all remain unchanged. 

 On the one hand, we break $\mathcal{I}_R$ into
 \[
\mathcal{I}_R =\int_{D_{0,R}}\nabla f_\Omega \cdot  \nabla \vphi_{0,R}dx +\sum_{i=1}^n\int_{U_i}\nabla f_\Omega \cdot \nabla \vphi_idx =: \mathcal{I}_R^1 + \mathcal{I}_R^2.
 \]
 Since $f_\Omega$ is constant on each connected component of $\partial U_i$, the same computation as \eqref{temp111}--\eqref{temp222} gives $\mathcal{I}_R^2 = 0$.
For $\mathcal{I}_{R}^1$, note that although $f_\Omega $ is a constant along the boundary of each hole of $D_{0,R}$, it is \emph{not} a constant along $\outb D_{0,R} = \partial B_R$. Thus similar computations as \eqref{temp111}--\eqref{temp222} now give
\begin{equation}\label{int_d0r}
\begin{split}
\mathcal{I}_R^1 
&=\int_{\partial B_R}\big(1_D * \mathcal{N} - \frac{\Omega}{2} R^2\big) \nabla \vphi_{0,R}\cdot \vec{n}d\sigma\\
&= \int_{\partial B_R}\left((1_{D}*\mathcal{N})(x) - |D|\mathcal{N}(x)\right)\nabla \vphi_{0,R}\cdot \vec{n}d\sigma(x),
\end{split}
\end{equation}
where in the second equality we used $\int_{\partial B_R}\nabla \vphi_{0,R}\cdot\vec{n}d\sigma=0$ and the fact that $\mathcal{N}(x)$ is constant on $\partial B_R$. For any $x\in \partial B_R$, since $D \subset B_{R_0}$ and $R>R_0$, we can control $(1_{D}*\mathcal{N})(x) - |D|\mathcal{N}(x)$ as
\begin{equation}\label{temp_diff1}
\begin{split}
\big|(1_{D}*\mathcal{N})(x) - |D|\mathcal{N}(x)\Big| \leq \frac{1}{2\pi}  \int_D\Big| \log|x-y| - \log|x| \Big|  dy \leq \frac{|D|}{2\pi} \left|\log\left(1-\frac{R_0}{R}\right)\right| \text{ on }\partial B_R.
\end{split}
\end{equation}
We introduce the following lemma to control $|\nabla\vphi_{0,R}\cdot \vec{n}|$ on $\partial B_R$, whose proof is postponed to the end of this subsection.

\begin{lem}\label{phi_convergence} Let $D \subset B_{R_0}$ be a domain with $C^1$ boundary. For any $R>R_0$, let $D_{0,R}$, $V$, $p_{0,R}$ and $\vphi_{0,R}$ be defined as in the proof of Theorem~\ref{clockwise rotating_patch}. Then we have 
\begin{equation}\label{phi_conv_2}
\displaystyle|\nabla \vphi_{0,R} \cdot \vec{n}| \leq \frac{N R_0^2}{2R \log(R/R_0)}\quad\text{ on }\partial B_R,
\end{equation}
where $N>0$ is the number of connected components of $V$ (and is independent of $R$).
\end{lem}

Once we have this lemma, plugging \eqref{phi_conv_2} and \eqref{temp_diff1} into \eqref{int_d0r} yields
\[
|\mathcal{I}_R^1| \leq    \frac{ N |D| R_0^2}{2 } \left| \log\left(1-\frac{R_0}{R}\right)\right| (\log(R/R_0))^{-1} \leq|D|\frac{ C(D, R_0)}{R \log R} \to 0 \quad \text{ as }R\to \infty.
\]
Combining this with $\mathcal{I}_R^2 = 0$ gives
\begin{equation}\label{limit I_R}
\lim_{R\to\infty}\mathcal{I}_{R}=0.
\end{equation}
 Next we compute $\mathcal{I}_{R}$ in another way. Note that  $1_{B_R}*\mathcal{N}-\frac{|x|^2}{4}$ is a radial harmonic function in $B_R$, thus is equal to some constant $C_R$ in $B_R$. Using this fact, we can rewrite $f_\Omega$ as 
 \[
 f_\Omega = 1_D * \mathcal{N} - \frac{\Omega}{2}|x|^2 = (1_D - 1_{B_R})* \mathcal{N}  - \Big(\frac{\Omega}{2}-\frac{1}{4}\Big) |x|^2 + C_R.
 \]
 As a result, $\mathcal{I}_{R}$ can be rewritten as
 \begin{equation}
 \begin{split}\label{redef I}
 \mathcal{I}_{R}&=-\int_{{B_{R}\setminus \overline{D}}}\nabla\left(1_{B_{R}\setminus \overline{D}}*\mathcal{N}\right)\cdot \nabla \vphi^Rdx-\frac{(2\Omega-1)}{2}\int_{{B_{R}\setminus \overline{D}}} x\cdot\nabla \vphi^{R}dx\\
 &=:-\mathcal{J}_R^1-\frac{(2\Omega-1)}{2}\mathcal{J}_R^2.
 \end{split}
 \end{equation}
 Next we will show $\mathcal{J}_R^1, \mathcal{J}_R^2 \geq 0$, leading to $\mathcal{I}_R\leq 0$. Let us start with $\mathcal{J}_R^2$. Applying Lemma~\ref{p_bd} to each of $D_{0,R}$ and $\{U_i\}_{i=1}^n$ immediately gives 
\begin{equation}\label{I_R2}
 \mathcal{J}_R^2 \geq \int_{D_{0,R}}  |x|^2 + \nabla p_{0,R} \cdot x dx + \sum_{i=1}^n \int_{U_i}  |x|^2 + \nabla p_i \cdot x dx =: T_{0,R} + \sum_{i=1}^n T_i \geq 0.
 \end{equation}
 Note that the $T_i$'s are independent of $R$ for $i=1,\dots,n$, and we know $T_i\geq 0$ with equality achieved if and only if $U_i$ is an annulus or a disk centered at the origin. This will be used later to show all $\{U_i\}_{i=1}^n$ are centered at the origin in the $\Omega>\frac{1}{2}$ case. (When $\Omega=\frac{1}{2}$, the coefficient of $\mathcal{J}_R^2$ becomes 0 in \eqref{redef I}, thus a different argument is needed in this case.)

We now move on to $\mathcal{J}_R^1$. We first break it as
\[
\begin{split} \mathcal{J}_R^1&=\int_{B_R\setminus \overline{D}}\nabla (1_{B_R\setminus \overline{D}}*\mathcal{N})\cdot xdx+\int_{B_R\setminus \overline D}\nabla(1_{B_R\setminus \overline{D}}*\mathcal{N})\cdot \nabla p^Rdx=:J_{11}+J_{12}.
\end{split}
\]
 An identical computation as \eqref{eq9} gives 
$ \displaystyle J_{11}
 = \frac{1}{4\pi}\Big(|D_{0,R}| + \sum_{i=1}^n|U_i|\Big)^2.$
  For $J_{12}$, the same computation as \eqref{I21_multiple}--\eqref{crude_bd*} gives the following (where we used that each $U_i$ lies in a hole of $D_{0,R}$ for $i=1,\dots,n$): 
 \[
 J_{12} \geq -\int_{D_{0,R}} p_{0,R} dx - \sum_{i=1}^n \int_{U_i} p_i dx -\sum_{1\leq i\leq n}|U_i| \sup_{D_{0,R}} p_{0,R}-   \sum_{\substack{1\leq i,j\leq n\\ j\prec i}} \frac{|U_i| |U_j|}{2\pi}.
  \]
 Adding up the estimates for $J_{11}$ and $J_{12}$, we get
 \begin{equation}\label{jr1}
 \begin{split}
 \mathcal{J}_R^1 \,\geq& \, \left(\frac{1}{4\pi} |D_{0,R}|^2 - \int_{D_{0,R}} p_{0,R} dx \right) + \left(\sum_{1\leq i\leq n}|U_i|\right) \left(\frac{1}{2\pi}|D_{0,R}| -\sup_{D_{0,R}} p_{0,R}\right)  \\
 &+ \sum_{i=1}^n \left(\frac{1}{4\pi} |U_i|^2 - \int_{U_i} p_i dx \right) +  \sum_{\substack{j\not\prec i \text{ and } i\not\prec j\\ i\neq j}} \frac{1}{4\pi} |U_i| |U_j|.
 \end{split}
 \end{equation}
 By Proposition~\ref{talenti2}, all terms on the right hand side are nonnegative. But note that only the two terms in the second line are independent of $R$.  Plugging \eqref{jr1} and \eqref{I_R2} into \eqref{redef I} gives the following (where we only keep the terms independent of $R$ on the right hand side):
 \[
 \liminf_{R\to\infty} (-\mathcal{I}_R) \geq \sum_{i=1}^n \left(\frac{1}{4\pi} |U_i|^2 - \int_{U_i} p_i dx \right) +  \sum_{\substack{j\not\prec i \text{ and } i\not\prec j\\ i\neq j}} \frac{1}{4\pi} |U_i| |U_j| + \frac{2\Omega-1}{2} \sum_{i=1}^n T_i  \geq 0.
 \]
Combining this with the previous limit \eqref{limit I_R}, we know $U_i$ must be an annulus or a disk for $i=1,\dots,n$, and  they must be nested in each other. In addition, if $\Omega>\frac{1}{2}$, we have $T_i=0$ for $i=1,\dots,n$, implying that each $U_i$ is centered at the origin.

The radial symmetry of $D_{0,R}$ is more difficult to obtain. Although the first two terms on the right hand side of \eqref{jr1} are both strictly positive if $D_{0,R}$ is not an annulus, we need some  uniform-in-$R$ lower bound to get a contradiction in the $R\to\infty$ limit. Between these two terms, it turns out the second term is easier to control. This is done in the next lemma, whose proof we postpone to the end of this subsection.

\begin{lem}
\label{lem2} Let $D \subset B_{R_0}$ be a domain with $C^1$ boundary. For any $R>R_0$, let $D_{0,R}$,  $V$ and $p_{0,R}$ be given as in the proof of Theorem~\ref{clockwise rotating_patch}. 
If $V$ is not a single disk, there exists some constant $C(V)>0$ only depending on $V$, such that
\[
\liminf_{R\to\infty} \left(\frac{1}{2\pi}|D_{0,R}| -\sup_{D_{0,R}} p_{0,R}\right) \geq C(V) > 0.
\]
\end{lem}

If $V$ is not a disk,  Lemma~\ref{lem2} gives $\liminf_{R\to\infty} \mathcal{J}_R^1>\left(\sum_{1\leq i\leq n}|U_i|\right)  C(V)>0$. (Recall that in the beginning of this proof we assume $\sum_{1\leq i\leq n}|U_i|>0$, and it is independent of $R$.) This  implies $\liminf_{R\to\infty}(- \mathcal{I}_R)\geq C(V)>0$, contradicting \eqref{limit I_R}.

So far we have shown that  $\partial D$ is a union of nested circles, and it remains to show that they are all centered at 0. For the $\Omega > \frac{1}{2}$ case, we already showed all $\{U_i\}_{i=1}^n$ are centered at 0, so it suffices to show the outmost circle $\partial V$ (denote by $B(\tilde o, \tilde r)$) is also centered at 0. By definition of $\{U_i\}_{i=1}^n$, we have $D = B(\tilde o, \tilde r) \setminus \left(\dot{\cup}_{i=1}^n U_i\right)$. Note that $1_{B(\tilde o, \tilde r)}*\mathcal{N} = \frac{|x-\tilde o|^2}{4}+C$ for some constant $C$, and $1_{\dot{\cup}_{i=1}^n U_i} * \mathcal{N}$ is radially increasing. Therefore $f_\Omega$ can be written as
\[
f_\Omega = 1_{B(\tilde o, \tilde r)}*\mathcal{N} - 1_{\dot{\cup}_{i=1}^n U_i} * \mathcal{N} - \frac{\Omega}{2}|x|^2 = \frac{|x-\tilde o|^2}{4} - g(x),
\]
where $g$ is radially symmetric, and strictly increasing in the radial variable. Since both $f_\Omega$ and $\frac{|x-\tilde o|^2}{4}$ are known to take constant values on $\partial B(\tilde o, \tilde r)$, it implies $g$ must be constant on $\partial B(\tilde o, \tilde r)$ too, and the fact that $g$ is a radially increasing function gives that $\tilde o=0$. This finishes the proof for $\Omega >\frac{1}{2}$.

For $\Omega=\frac{1}{2}$, we do not know whether $\{U_i\}_{i=1}^n$ are centered at 0 yet. Denote by $U_1$ be the innermost one. Then we have 
\begin{equation}\label{tempf0}
f_\Omega(x) =  \frac{|x-\tilde o|^2}{4} - 1_{\dot{\cup}_{i=1}^n U_i} * \mathcal{N} - \frac{|x|^2}{4} =  \frac{\tilde o \cdot x}{2}+ \text{const}\quad\text{for }x\in \outb U_1,
\end{equation}
where the second equality follows from Lemma~\ref{lem_newton}(b), where we used that $1\prec j$ for all $2\leq j\leq n$. Combining \eqref{tempf0} with the fact that $f_\Omega= \text{const}$ on  $\outb U_1$ gives $\tilde o=0$, that is, the outmost circle must be centered at 0. This leads to $f_\Omega = - \sum_{i=1}^n 1_{U_i} * \mathcal{N}$. Since $f_\Omega=\text{const}$ on each connected component of $\partial U_i$,  we can apply the last part in the proof of Corollary~\ref{cor_multiple_steady} to show that all $\{U_i\}_{i=1}^n$ are all concentric. Denoting their center by $o_1$, we can show that $o_1 = 0$: Lemma~\ref{lem_newton}(a)  gives $f_\Omega(x) = C \ln |x- o_1|$ for some $C<0$ on $\partial B(\tilde o, \tilde r)$, and since we have $f=\text{const}$ on $\partial B(\tilde o, \tilde r)$ and $\tilde o=0$, it implies that $o_1=0$, finishing the proof.
\end{proof}

\begin{proof}[\textup{\textbf{Proof of Lemma~\ref{phi_convergence}}}]   For notational simplicity, we shorten $p_{0,R}$, $D_{0,R}$ and $\vphi_{0,R}$ into $p_R$, $D_R$ and $\vphi_R$ thoughout this proof. Recall that $\partial D_R = \partial B_R \cup \partial V$. Clearly we have $\vphi_R=\frac{R^2}{2}$ on $\partial B_R$, due to $p_R = 0$ on $\outb D_R = \partial B_R$. We claim that 
 \begin{equation}\label{vphi_bound}
-\frac{N R_0^2}{2} \leq \vphi_R-\frac{R^2}{2}\le  \frac{R_0^2}{2}\quad\text{ on }\partial V,
 \end{equation}
 where $N\geq 1$ is the number of connected components of $V$.
 Once it is proved, we apply the comparison principle to the functions $\varphi_{R}-\frac{R^2}{2}$ and $\pm g$, where 
 \[g(x) := \frac{N R_0^2}{2\log(R/R_0)} \log\frac{R}{|x|}.
 \] Note that $g\equiv 0$ on $\partial B_R$, and  $g \geq  \frac{N R_0^2}{2}$ on $\partial V$ since $\partial  V \subset B_{R_0}$. If $0 \not\in \overline{D_{R}}$, then the functions $\vphi_R-\frac{R^2}{2}$ and $\pm g$ are all harmonic in $D_{R}$, their values on $\partial B_R$ are all 0, and their boundary values on $\partial  V$ are ordered due to \eqref{vphi_bound}. The comparison principle in $D_R$ then yields 
 \begin{equation}\label{phi_conv_1}
-g(x)\leq \vphi_R (x)-\frac{R^2}{2} \leq g(x) \quad \text{ in $D_{R}$}.
\end{equation}
Since $\vphi_R-\frac{R^2}{2} \equiv g \equiv 0$  on $\partial B_R$, \eqref{phi_conv_1} gives $|\nabla \varphi_R \cdot \vec{n}| \leq |\nabla g\cdot \vec{n}| = \frac{N R_0^2}{2R \log(R/R_0)}$ on $\partial B_R$, which is the desired estimate \eqref{phi_conv_2}.
And if $0 \in \overline{D_R}$, then \eqref{phi_conv_1} still holds in $D_R \setminus B_{\epsilon}$ for all sufficiently small $\epsilon>0$ by applying the comparison principle in this set, and \eqref{phi_conv_2} again follows as a consequence.

In the rest of the proof we will show \eqref{vphi_bound}. Its second inequality is straightforward: 
\[
 \vphi_R-\frac{R^2}{2}\leq \frac{R_0^2}{2}+\sup_{\overline{D_R}} p_R-\frac{R^2}{2}\leq \frac{R_0^2}{2} \quad\text{ on }\partial V,
\]
here the first inequality follows from the definition of $\vphi_R$ and the fact that $ V \subset B_{R_0}$, and the second inequality is due to $\sup_{\overline{D_R}} p_R \le \frac{|D_R|}{2\pi}\le \frac{R^2}{2}$ in Proposition~\ref{talenti2}. 
 
It remains to prove the first inequality of \eqref{vphi_bound}. Let us fix any $R>R_0$. Denote the $N$ connected components of $ \partial  V$ by  $\{\Gamma_i\}_{i=1}^N$, and let $\Gamma_0:=\partial B_R$. These notations lead to $\partial D_R = \cup_{i=0}^N \Gamma_i$. For $i=0,\dots,N$, let $L_i\subset \mathbb{R}$ be the range of $\vphi_R-\frac{R^2}{2}$ on $\Gamma_i$. By continuity of $\vphi_R$,  each $L_i$ is a closed bounded interval, which can be a single point. Clearly,  $L_0 = \{0\}$ due to $\varphi_R|_{\partial B_R}=\frac{R^2}{2}$.
Towards a contradiction, suppose
\begin{equation}\label{def_ar} v_{\min} := \min_{1\leq i\leq N} \inf L_i = \inf_{\partial  V}\Big(\vphi_R-\frac{R^2}{2}\Big)=: -\frac{N |R_0|^{2}}{2} -\delta \text{\quad for some }\delta>0.
\end{equation}
As for the maximum value, since $L_0 = \{0\}$ we have 
\begin{equation}\label{max_vi}
v_{\max} := \max_{0\leq i\leq N} \sup L_i \geq 0.
\end{equation}
For $i=1,\dots, N$, using $p_R|_{\Gamma_i}=\text{const}$, $\vphi_R = p_R + \frac{|x|^2}{2}$ and $\Gamma_i \subset B_{R_0}$, the length of each interval $L_i$ satisfies
\begin{equation}\label{tempvi}
\big|L_i\big| = \text{osc}_{\Gamma_i}\frac{|x|^2}{2} \leq \frac{R_0^2}{2} \quad\text{ for }i = 1,\dots, N.
\end{equation}

Comparing \eqref{tempvi} with \eqref{def_ar}--\eqref{max_vi}, we know the union of $\{L_i\}_{i=0}^N$ cannot fully cover the interval $[v_{\min}, v_{\max}]$, thus they can be separated in the following sense: there exists a nonempty proper subset $S \subset \{0,\dots,N\}$, such that the range of $L_i$ for indices in $S$ and $S^c:=\{0,\dots,N\}\setminus S$ are strictly separated by at least $\delta$, i.e.
$
\min_{i\in S} \inf L_i \geq \max_{i\in S^c}\sup L_i + \delta.
$
In terms of $\varphi_R$, we have
\begin{equation}\label{gamma_2}
\min_{i\in S} \inf_{\Gamma_i} \vphi_R \geq \max_{i\in S^c} \max_{\Gamma_i} \vphi_R + \delta.
\end{equation}
Since $\vphi_R$ is harmonic in $D_R$, whose boundary is $\cup_{i=0}^N \Gamma_i$, it is a standard comparison principle exercise to show that \eqref{gamma_2} implies
\begin{equation}\label{gamma_1}
\sum_{i\in S} \int_{\Gamma_i} \nabla\vphi_R\cdot \vec{n} d\sigma > 0,
\end{equation}
where $\vec n$ denotes the outer normal of $D_R$. But on the other hand, we have
\begin{equation}\label{gamma_0}
\int_{\Gamma_i} \nabla\vphi_R\cdot \vec{n} d\sigma=0\quad\text{ for }i=0,\dots,N.
\end{equation}
To see this, the cases $i=1,\dots,N$ can be done by an identical computation as \eqref{temp222}, and the $i=0$ case follows from $\int_{\partial D_R}\vphi_R\cdot \vec{n} d\sigma=\int_{D_R} \Delta \varphi_R dx = 0$ and the fact that $\partial D_R = \cup_{i=0}^N \Gamma_i$. Thus we have obtained a contradiction between \eqref{gamma_1} and \eqref{gamma_0}, completing the proof.
 \end{proof}

\begin{proof}[\textbf{\textup{Proof of Lemma \ref{lem2}}}] Assume $V$ has $N$ connected components $\{V_j\}_{j=1}^N$ for $N\geq 1$. For notational simplicity, we shorten $D_{0,R}$, $p_{0,R}$ and $\vphi_{0,R}$ into $D_R, p_R$ and $\vphi_R$ in this proof. Let $\epsilon_R :=  \frac{1}{2\pi}|D_{R}| - \sup_{D_R}p_R$, which is nonnegative by Proposition~\ref{talenti2}.  
Towards a contradiction, assume there exists a diverging subsequence $\{R_i\}_{i=1}^\infty$ such that $\lim_{i\to 0}\epsilon_{R_i}=0$. 

Define $\tilde \vphi_{R_i} := \vphi_{R_i} - \frac{R_i^2}{2}$. We claim that $\{\tilde \vphi_{R_i}\}_{i=1}^\infty$ has a subsequence that converges locally uniformly to some bounded harmonic function $\vphi_\infty$ in $\mathbb{R}^2 \setminus V$. 

To show this, we will first obtain a uniform bound of $\{\tilde \vphi_{R_i}\}_{i=1}^\infty$. Note that \eqref{vphi_bound} gives that $\sup_{\partial V}|\tilde\vphi_{R_i}| \leq \frac{N R_0^2}{2}$ for all $i\in\mathbb{N}^+$. 
Since $\tilde \vphi_{R_i}\equiv 0$ on $\partial B_{R_i}$ for all $i\in\mathbb{N}^+$, the maximum principle for harmonic function gives $
\sup_{D_{R_i}}|\tilde\vphi_{R_i}| \leq \frac{N R_0^2}{2}$ for all $i\in\mathbb{N}^+$. 


For any $R>2R_0$, we will obtain a uniform gradient estimate for $\{\tilde\vphi_{R_i}\}$  in $D_R$ for all  $R_i>2R$. First note that since $\partial B_R$ is in the interior of $D_{R_i}$ (due to $R_i>2R$), interior estimate for harmonic function (together with the above uniform bound) gives that $\|\tilde\vphi_{R_i}\|_{C^2(\partial B_R)}\leq C(N,R_0)$. On the other boundary $\partial V$, recall that $\tilde \vphi_{R_i}|_{\partial V_j}= \frac{|x|^2}{2}+c_{i,j}$, with $|c_{i,j}|\leq \frac{(N+1) R_0^2}{2}$. Thus $\|\tilde\vphi_{R_i}\|_{C^2(\partial D_R)}\leq C(N,R_0)$ for all $R_i>2R$, and the standard elliptic regularity theory gives the uniform gradient estimate 
$\sup_{D_R}|\nabla \tilde\vphi_{R_i}| \leq C(V)$. 
This allows us to take a further subsequence (which we still denote by $\{\tilde \varphi_{R_i}\}$) that converges uniformly in $\bar{D_R}$ to some harmonic function $\tilde \vphi_\infty \in C(\bar{D_R})$. Since $R>2R_0$ is arbitrary, we can repeat this procedure (for countably many times) to obtain a subsequence that converges locally uniformly to a harmonic function $\tilde \vphi_\infty$ in $\mathbb{R}^2 \setminus V$, where $\tilde \vphi_\infty|_{\partial V_j}= \frac{|x|^2}{2}+c_{j}$ with $|c_{j}|\leq \frac{(N+1) R_0^2}{2}$. This finishes the proof of the claim.

Now define 
\[
\tilde p_{R_i} := p_{R_i} -\frac{R_i^2}{2}= \tilde \vphi_{R_i}-\frac{|x|^2}{2},
\] which is known to converge locally uniformly to $\tilde p_\infty := \tilde \vphi_\infty - \frac{|x|^2}{2}$ in $\mathbb{R}^2\setminus V$. Note that $\tilde p_\infty$ is \emph{not} radially symmetric up to any translation: To see this, recall that $\tilde p_\infty|_{\partial V_j}\equiv c_{j}$. If $\tilde p_\infty$ is radial about some $x_0$,  it must be of the form $-\frac{|x-x_0|^2}{2}+c$ due to $\Delta \tilde p_\infty=-2$. As a result, the level sets of $\tilde p_\infty$ are all nested circles, thus $V$ must be a single disk (where we used that each connected component of $V$ is simply-connected).

Next we will show that $\lim_{i\to 0}\epsilon_{R_i}=0$ implies $\tilde p_\infty$ is radial up to a translation, leading to a contradiction. For $k\in\mathbb{R}$, let $g_i(k) := |\{x\in D_{R_i}: p_{R_i}(x)>k\}|$. In the proof of Proposition~\ref{talenti2}, we have shown that $g_i(0) = |D_{R_i}|$, $g_i$ is decreasing in $k$, with $g_i'(k) \leq -2\pi$ for almost every $k \in (0,\sup_{D_{R_i}} p_{R_i})$. Since $\sup_{D_{R_i}}p_{R_i}=  \frac{1}{2\pi}|D_{R_i}| - \epsilon_{R_i} $, we can control $g_i(k)$ from below and above as follows:
\begin{equation}\label{squeeze}
 (|D_{R_i}|-2\pi k -2\pi\epsilon_{R_i} )_+ \leq g_i(k) \leq  (|D_{R_i}|-2\pi k )_+  \quad\text{ for all }k\geq 0.
\end{equation}
Likewise, define $\tilde g_i(k) := |\{x\in D_{R_i}: \tilde p_{R_i}(x)>k\}|$, and $\tilde g_{\infty}(k) := |\{x\in D_{R_i}: \tilde p_{\infty}(x)>k\}|$. Since $\tilde p_{R_i} = p_{R_i} - \frac{R_i^2}{2}$, we have $\tilde g_i(k) = g_i(k+\frac{R_i^2}{2}) $ for all $k\geq -\frac{R_i^2}{2}$, thus \eqref{squeeze} is equivalent to
\[
 (-|V|-2\pi k - 2\pi \epsilon_{R_i})_+ \leq \tilde g_i(k)\leq (-|V| -2\pi k  )_+ \quad \text{ for all }k\geq -\frac{R_i^2}{2}.
\]
The locally uniform convergence of $p_{R_i}$ gives $\lim_{i\to\infty} \tilde g_i = \tilde g_0$, and since we assume $\lim_{i\to\infty}\epsilon_{R_i}=0$, we take the $i\to\infty$ limit in the above inequality and obtain
\[
\tilde g_{\infty}(k) = (-2\pi k - |V|)_+\quad \text{ for all }k\in \mathbb{R},
\]which implies 
\begin{equation}
\label{g0}
\tilde g'_{\infty}(k)=-2\pi\quad\text{ for all }k\in(-\infty,\sup_{\mathbb{R}^2\setminus V} \tilde p_{\infty}).\end{equation}
Applying the proof of Proposition~\ref{talenti2} to $\tilde p_\infty$ (note that the proof still goes through even though $\tilde p_\infty$ takes negative values, and is defined in an unbounded domain), we have that \eqref{g0} can happen only if $\tilde D_k := \{\tilde p_\infty>k\}\dot{\cup}( \dot{\cup}_{\{j: c_j > k\}} \overline{V_j})$ is a disk for almost every $k\in(-\infty,\sup_{\mathbb{R}^2\setminus V} \tilde p_{\infty})$, and $|\nabla \tilde p_\infty|$ is a constant on almost every $\partial \tilde D_k$. These two conditions imply that all regular sets of $\tilde p_\infty$ are concentric circles, thus $\tilde p_\infty$ is radial up to a translation, and we have obtained a contradiction.\end{proof}

\color{black}

\section{Radial symmetry of nonnegative smooth stationary/rotating solutions\\ for 2D Euler with $\Omega\leq 0$}\label{sec3}

 
 
Let $\omega(x,t)=\omega_0(R_{\Omega t}x)$ be a nonnegative compactly-supported stationary/rotating  solution of  2D Euler with angular velocity $\Omega\leq 0$. Recall that by \eqref{eq_smooth}, $f_\Omega := \omega_0 * \mathcal{N}-\frac{\Omega}{2}|x|^2$ is a constant along each connected component of a regular level set of $\omega_0$. In this section, we prove that $\omega_0$ is radial up to a translation for $\Omega=0$, and radial for $\Omega<0$. As we discussed in the introduction, the $\omega\geq 0$ condition is necessary:  in a forthcoming work \cite{GomezSerrano-Park-Shi-Yao:nonradial-stationary-solutions} we will show that there exists a compactly-supported, sign-changing smooth stationary vorticity $\omega_0$ that is non-radial, and the construction also works for $\Omega<0$ that is close to 0.

Most of this section is devoted to the proof of Theorem~\ref{theorem3} in the $\Omega=0$ case (the $\Omega<0$ case is done in Corollary~\ref{cor_conti_rotating} as a simple extension). In the proof, the two key steps are to show that every connected component of a regular level set of $\omega$ is a circle, and these circles are concentric. These are done by approximating $\omega$ by a step function $\omega_n = \sum_{i=1}^{M_n} \alpha_i 1_{D_i}$ such that the sets $\{D_i\}$ are disjoint, and $\|\omega-\omega_n\|_{L^{\infty}} = O(1/n)$. We then define $\varphi^n=\frac{|x|^2}{2}+\sum_{i=1}^{M_n} 1_{D_i} p_i$ corresponding to this step function $\omega_n$, and compute $\mathcal{I}_n = \int \omega_n \nabla \varphi^n \cdot \nabla f dx$ in two ways. 


Due to the $O(1/n)$ error in the approximation, the qualitative statement in Proposition~\ref{talenti2} that ``the equality is achieved if and only if $D$ is a disk or annulus'' is no longer good enough for us. We need to obtain various quantitative versions of \eqref{estimate_p_general1} for doubly-connected domains, and three such versions are stated below.

In Lemma~\ref{talenti_2}, the quantitative constant $c_0>0$ depends on the Fraenkel asymmetry of the outer boundary defined in Definition~\ref{Fraenkel}. 
 
 \begin{definition}[{c.f.~\cite[Section 1.2]{Fusco-Maggi-Pratelli:stability-faber-krahn}}] \label{Fraenkel}
 For a bounded domain $E\subseteq \R^{2}$, we define the \emph{Fraenkel asymmetry} $\mathcal{A}(E)\in[0,1)$ as
\begin{align*}
\mathcal{A}(E):=\inf_{x_0}\left\{ \frac{|E\lap (x_{0}+rB)|}{|E|}: x_{0}\in \R^{2},\pi r^{2}=|E|\right\},
\end{align*}
where $B$ is a unit disk in $\R^{2}$.
\end{definition}

\begin{lem}\label{talenti_2}
Let $D$ be a doubly connected set. Let us denote the hole of $D$ by an open set $h$, and let $\tilde D := D\cup \bar h$. 
We define $p$ in $D$ as in Lemma~\ref{p_def}. Then if $\mathcal{A}(\tilde D)>0$, there is a constant $c_{0}>0$ that only depends on $\mathcal{A}(\tilde D)$, such that
\begin{align*}
p|_{\partial h}\le \frac{|D|}{2\pi}(1-c_{0}).
\end{align*}
\end{lem}

Lemma~\ref{talenti_2} will be used in the proof of Theorem~\ref{theorem3} to show that all level sets of $\omega$ are circles. To obtain radial symmetry of $\omega$, we also need to show all these level sets are concentric. To do this, we need to obtain some quantitative lemmas for a region between two non-concentric disks.  In Lemma~\ref{talenti_3} we consider a thin tubular region between two non-concentric disks whose radii are  close to each other, and obtain a quantitative version of \eqref{estimate_p_general1} for such domain. 

\begin{lem}\label{talenti_3}
For $\epsilon>0$, consider two open disks $B_{1}:=B(o_{1},1)$ and $B_{2}=B(o_{2}, 1+\epsilon)$ such that $B_{1}\subset B_{2}$.  Suppose $|o_{1}-o_{2}|=a\epsilon$ with $a\in (0,1)$, and let $p$ be defined as in Lemma~\ref{p_def} in $D:=\overline{B_{2}}\setminus B_{1}$. Then if $\epsilon$ and $a$ satisfy that $0<\epsilon < \frac{a^2}{64}$, we have
\begin{align}
p\big|_{\partial B_{1}}\le \frac{|D|}{2\pi}\left(1-\frac{a^2}{16}\right).
\end{align}
\end{lem}

In Lemma~\ref{talenti_4} we consider a region between two non-concentric disks (that is not necessarily a thin tubular region),
 and obtain a quantitative version of \eqref{estimate_p_general1} for such domain. 
 

\begin{lem}\label{talenti_4}
Consider two open disks $B_{r}:=B(o_{1},r)$ and $B_{R}=B(o_{2}, R)$ such that $B_{r}\subset B_{R}$ . Let $p$ be defined as in \eqref{p_def} in $D:=B_{R}\backslash \overline{B_{r}}$. Suppose $l:=|o_{1}-o_{2}|>0$ and there exist $\delta_1>0$, and $\delta_2>0$ such that $\delta_1<r<R<\delta_2$. Then there exist a constant $c_{0}$ that only depends on $\delta_1,$ $\delta_2$ and $l$ such that
\begin{align*}
p|_{\partial B_{r}}\le \frac{|D|}{2\pi}(1-c_{0}).
\end{align*}
\end{lem}

The proofs of the above quantitative lemmas will be postponed to Section~\ref{sec21}. Now we are ready to prove the main theorem.

\color{black}

\begin{thm}\label{theorem3}
Let $\omega$ be a compactly supported smooth nonnegative stationary solution to the 2D Euler equation. Then $\omega$ is radially symmetric up to a translation.

\end{thm}
\begin{proof}
Note that as mentioned in step 1 of Proposition~\ref{talenti2}, we have that for almost every $k\in (0,\|\omega\|_{L^\infty})$, $\omega^{-1}(\left\{ k\right\})$ is a smooth 1-manifold. Furthermore, since $\omega$ is compactly supported, each such level set is a disjoint union of finite number of simple closed curves. For any such closed curve, we call it a ``\emph{level set component}'' in this proof.

\indent We split the proof into several steps. Throughout step 1, 2 and 3, we prove that all level set components of $\omega$ must be circles. In step 4, we will prove that any two level set components are nested, i.e. one is contained in the other. Lastly we present the proof that all level set components are concentric in step 5 and 6. 

\textbf{Step 1.} Towards a contradiction, suppose there is $k>0$ that is a regular value of $\omega$, and $\omega^{-1}(\left\{ k\right\})$ has a connected component $\Gamma$ that differs from a circle.   Recall that $\interior(\Gamma)$ denotes the interior of $\Gamma$, which is open and simply connected. 
Since $\Gamma$ is not a circle, we have $\mathcal{A}(\interior(\Gamma))>0$, with $\mathcal{A}$ as in Definition~\ref{Fraenkel}.

In this step, we investigate level set components near  $\Gamma$. Since $k$ is a regular value, we can find an open neighborhood $U$ of $\Gamma$ and a constant $\eta>0$ such that $|\nabla \omega|>\eta$ in $U$. For any $x\in\Gamma$, consider the flow map $\Phi_t(x)$ originating from $x$, given by
\[
\frac{d}{dt} \Phi_t(x) = \frac{\nabla\omega(\Phi_t(x))}{|\nabla \omega(\Phi_t(x))|^{2}}
\]
with initial condition $\Phi_0(x) = x$. Since $\frac{\nabla \omega}{|\nabla\omega|^2}$ is smooth and bounded in $U$, we can choose $\delta_1>0$ so that $\Phi_{t}(\Gamma):=\{\Phi_t(x): x\in \Gamma\}$ lies in $U$ for any $t\in (-\delta_1,\delta_1)$. Note that $\Phi_{t}$'s are diffeomorphisms, thus $\Phi_{t}(\Gamma)$ is also a smooth simple closed curve for $t\in (-\delta_1,\delta_1)$.
Then we observe that
\begin{align}\label{eq27}
\frac{d}{dt}\omega(\Phi_{t}(x))=\nabla \omega(\Phi_{t}(x))\cdot \frac{\nabla \omega(\Phi_{t}(x))}{|\nabla \omega(\Phi_{t}(x))|^{2}}=1 \quad\text{ for $(t,x)\in (-\delta_1,\delta_1)\times \Gamma$.}
\end{align}
Hence for each $t\in (-\delta_1,\delta_1)$, $\Phi_{t}(\Gamma)$ is a level set component and 
\begin{align}\label{eq14}
\omega(\Phi_{t_{1}}(\Gamma))\ne \omega(\Phi_{t_{2}}(\Gamma)) \text{ if $t_{1}\ne t_{2}$.}
\end{align}
 By continuity of the map $(t,x)\mapsto \Phi_{t}(x)$, we can find $\delta_2\in(0,\delta_1)$ such that 
\begin{align}\label{eq15}
\mathcal{A}(\interior(\Phi_{t}(\Gamma)))>\frac{1}{2}\mathcal{A}(\interior(\Gamma)) \text{ for any $t\in(-\delta_2,\delta_2)$}.
\end{align}
\indent Since two different level sets cannot intersect,  we can assume without loss of generality that $\interior(\Phi_{-\delta_2}(\Gamma))\subset \interior(\Phi_{\delta_2}(\Gamma))$. Then it follows from the intermediate value theorem  and \eqref{eq27} that
\begin{align}
\interior(\Phi_{-\delta_2}(\Gamma))\subset \Phi_{t}(\Gamma)\subset  \interior(\Phi_{\delta_2}(\Gamma)), \text{ for any $t\in (-\delta_2,\delta_2)$}.\label{eq16}
\end{align} 
Lastly we denote $V:=\interior(\Phi_{\delta_2}(\Gamma))\backslash \overline{\interior(\Phi_{-\delta_2}(\Gamma))}$ which is a nonempty open doubly-connected set, therefore $|V|>0$.\\

\indent \textbf{Step 2.} For any integer $n> 1$, we claim that we can approximate $\omega$ by a step function $\omega_n$ of the form $\omega_n(x) = \sum_{i=1}^{M_n} \alpha_{i}1_{D_{i}}(x)$, which satisfies all the following properties. 

\begin{itemize}
\item[(a)] Each $D_{i}$ is a connected open domain with smooth boundary and possibly has a finite number of holes.
\item[(b)] Each connected component of $\partial D_{i}$ is a level set component of $\omega$.
\item[(c)] $D_{i}\cap D_{j}=\emptyset$ if $i\ne j$. 
\item[(d)] $\|\omega_n - \omega\|_{L^\infty(\mathbb{R}^2)} \le \frac{2}{n} \|\omega\|_{L^\infty(\mathbb{R}^n)}.$
\end{itemize}

To construct such $\omega_n$ for a fixed $n>1$, let $r_{0}=0$ and $r_{n+1}=\| \omega\|_{L^{\infty}}$. We pick $r_1,\dots,r_n$ to be regular values of $\omega$, such that $0< r_{1}<\dots<r_{n}<\| \omega\|_{L^{\infty}}$, and $r_{i+1}-r_{i}<\frac{2}{n} \| \omega\|_{L^{\infty}}$ for $i=0,\dots,n$. We denote $D_{i}:=\left\{ x\in \R^{2}: r_{i}<\omega(x)<r_{i+1}\right\}$ for $i = 1,\dots, n-1$, and let $D_n := \left\{ x\in \R^{2}: \omega(x)> r_n\right\}$. Thus for each $i = 1,\cdots,n$, $D_{i}$ is a bounded domain with smooth boundary. We can then write it as $D_{i}=\dot{\cup}_{l=1}^{m_{i}}D_{i}^{l}$ for some $m_{i}\in \N$ where $D_{i}^{l}$'s are connected components of $D_{i}$. Then let $\omega_{n}(x):=\sum_{i=1}^{n}r_{i}\sum_{l=1}^{m_{i}}1_{D_{i}^{l}}$. By relabeling the indices, we rewrite $\omega_n(x)=\sum_{i=1}^{M_n} \alpha_i 1_{D_i}$, where $M_n=\sum_{i=1}^n m_i$, and each $\alpha_i \in \{r_1,\dots, r_n\}$. One can easily check that such $\omega_n$ satisfies properties (a)--(d). 

Of course, there are many ways to choose the values $r_1,\cdots, r_n$, with each choice leading to a different $\omega_n$. From now on, for any $n>1$, we fix $\omega_n(x) := \sum_{i=1}^{M_n} \alpha_{i}1_{D_{i}}(x)$ as any function constructed in the above way. (Note that $\alpha_i$ and $D_i$ all depend on $n$, but we omit their $n$ dependence for notational simplicity.)

\indent Finally, let us point out that for $\omega_n(x)$ constructed above, if $D_{i}\subset V$ for some $i$, then  $D_{i}$ must be doubly connected, since step 1 shows that all level set components in $V$ are nested curves. We will use this  in step 3 and 5. \\

\textbf{Step 3.} For any $n>1$, let $\omega_n(x) = \sum_{i=1}^{M_n} \alpha_{i}1_{D_{i}}(x)$ be constructed in Step 2.  For each $D_i$, let we define $p_{i}^n$ in $D_{i}$ as in Lemma~\ref{p_def}. We set 
\begin{align}\label{eq42}
\begin{cases}
&p^n:=\sum_{i=1}^{M_n}p_{i}^n1_{D_{i}}\\
&\vphi^n_{i}:=p_{i}+\frac{|x|^{2}}{2} \text{ in }D_{i}\\
&\vphi^n:=\sum_{i=1}^{M_n}\vphi_{i}^{n}1_{D_{i}}.
\end{cases}
\end{align}
As in Theorem~\ref{theorem1}, let $f:=\omega * \mathcal{N}$, and we will compute 
\begin{equation}\label{def_i_n}
\mathcal{I}^n := \int_{\mathbb{R}^2} \omega_n(x) \nabla \varphi^n(x)\cdot \nabla f(x) dx
\end{equation}
 in two different ways and derive a contradiction by taking the $n\to\infty$ limit.\\
\indent On the one hand, the same computation as in \eqref{temp111}--\eqref{temp222} yields that
\begin{align}\label{eq20}
\mathcal{I}^n=\sum_{i=1}^{M_n}\alpha_{i}\left(\int_{\partial D_{i}} f(x)\nabla \vphi_{i}^n(x)\cdot \vec{n}d\sigma-\int_{D_{i}}f(x)\lap \vphi_{i}^n(x)dx\right)=0.
\end{align}
On the other hand, 
\begin{align*}
\mathcal{I}^n&=\int_{\R^{2}}\omega_{n}(x)x\cdot \nabla f(x)dx+\int_{\R^{2}}\omega_{n}(x) \nabla p^n(x) \cdot \nabla f(x)dx\\
&=:\mathcal{I}^n_1+\mathcal{I}^n_2.
\end{align*}
Since $\omega_n$ satisfies property (d) in step 2, it follows that
\begin{align*}
\lim_{n\to\infty}\mathcal{I}^n_1=\int_{\R^{2}}\omega(x) x\cdot \nabla f(x)dx.
\end{align*}
A similar computation as in \eqref{eq9} yields that
\begin{align}
 \int_{\R^{2}}\omega(x)x\cdot \nabla f(x)dx&=\frac{1}{2\pi}\int_{\R^{2}}\int_{\R^{2}}\omega(x)\omega(y) \frac{x\cdot(x-y)}{|x-y|^{2}}dxdy\nonumber\\
&=\frac{1}{4\pi}\int_{\R^{2}}\int_{\R^{2}}\omega(x)\omega(y)dxdy\nonumber\\
&=\frac{1}{4\pi}\left(\int_{\R^{2}}\omega(x)dx\right)^{2},\label{eq21}
\end{align}
where we used the symmetry of the integration domain to get the second equality. \\
\indent Now we estimate the limit of $\mathcal{I}_2^n$. By Lemma~\ref{p_bd}, we have $\int_{D_i} |\nabla p^n_i|^2 dx \leq \int_{D_i} |x|^2 dx$, hence $\|\omega_{n}\nabla p\|_{L^{2}(\R^{2})}$ is uniformly bounded. Since $\omega_{n}\to\omega$ in $L^{\infty}$, the bounded convergence theorem yields that
\begin{align*}
\lim_{n\to\infty}\int_{\R^{2}}\omega_{n}\nabla p^n\cdot \nabla \left((\omega_{n}-\omega)*\mathcal{N}\right)(x)dx=0,
\end{align*}
therefore
\begin{align*}
\liminf_{n\to\infty}\mathcal{I}^n_2&=\liminf_{n\to\infty}\int_{\R^{2}}\omega_{n}(x)\nabla p^n(x)\cdot \nabla (\omega_{n}*\mathcal{N})dx.
\end{align*}

From now on, we will omit the $n$ dependence in $p_i^n$ for notational simplicity. Let us break the integral in the right hand side as
\begin{equation}
\label{def_f1f2}
\begin{split}
\int_{\R^{2}}\omega_{n}(x)\nabla p^n(x)\cdot \nabla (\omega_{n}*\mathcal{N})dx
&=\sum_{i,j=1}^{M_n}\alpha_{i}\alpha_{j}\int_{D_{i}}\nabla p_{i}\cdot \nabla(1_{D_{j}}*\mathcal{N})dx\\
&=\sum_{i=1}^{M_n}\alpha_{i}^{2}\int_{D_{i}}\nabla p_{i}\cdot \nabla(1_{D_{i}}*\mathcal{N})dx +\sum_{i\ne j}\alpha_{i}\alpha_{j}\int_{D_{i}}\nabla p_{i}\cdot \nabla(1_{D_{j}}*\mathcal{N})dx\\
&=:F_{1}+F_{2}.
\end{split}
\end{equation}
For $F_{1}$, the divergence theorem yields
\begin{align}\label{eq23}
F_{1}&=\sum_{i=1}^{M_n}\alpha_{i}^{2}\left( \int_{\partial D_{i}}p_{i}\nabla(1_{D_{i}}*\mathcal{N})\cdot \vec{n}d\sigma-\int_{D_{i}}p_{i}dx\right) = -\sum_{i=1}^{M_n}\alpha_{i}^{2}\int_{D_{i}}p_{i}dx,
\end{align}
where the second equality follows from an identical computation as in \eqref{eq8}.
 Then by Proposition~\ref{talenti2}, we have
\begin{align}\label{eq17}
F_{1}&\ge-\frac{1}{4\pi}\sum_{i=1}^{M_n}\alpha_{i}^{2}|D_{i}|^{2}.
\end{align}

For $F_{2}$, the divergence theorem yields
\[
F_{2}=\sum_{i\ne j}\alpha_{i}\alpha_{j}\left(   \int_{\partial D_{i}}p_{i}\nabla (1_{D_{j}}*\mathcal{N})\cdot\vec{n}d\sigma-\int_{D_{i}}p_{i}1_{D_{j}}dx\right)=\sum_{i\ne j}\alpha_{i}\alpha_{j}\int_{\partial D_{i}}p_{i}\nabla(1_{D_{j}}*\mathcal{N})\cdot\vec{n}d\sigma,
\]where we use property (c) in step 2 to get the last equality. 

For $i\neq j$, recall that as in the proof of Corollary~\ref{cor_multiple_steady}, we denote $j\prec i$ if $D_j$ is contained in a hole of $D_i$. Then divergence theorem gives 
\begin{equation}\label{int_crude}
\int_{\partial D_{i}}p_{i}\nabla (1_{D_{j}}*\mathcal{N}) \cdot \vec n d\sigma \begin{cases}= 0 \quad &\text{ if } j \not\prec i,\\
\ge -\sup_{\partial D_{i}}p_{i}|D_{j}|  & \text{ if } j\prec i.
\end{cases}
\end{equation}

Next we will improve this inequality for $j\prec i$ and $i\in L$, where $L:=\left\{ i : D_{i}\subset V\right\}$. (Note that $L$ depends on $\omega_n$, where we omit this dependence for notational simplicity.) From the discussion at the end of step 2, we know that $D_i$ has exactly one hole for all $i \in L$. Using the divergence theorem together with this observation,  \eqref{int_crude} becomes
\begin{equation}\label{eq29}
\int_{\partial D_{i}}p_{i}\nabla (1_{D_{j}}*\mathcal{N}) \cdot \vec n d\sigma \begin{cases}= 0 \quad &\text{ if } j \not\prec i,\\
\ge -\sup_{D_{i}}p_{i}|D_{j}|  & \text{ if } j\prec i \text{ and } i \not\in L,\\
= -p_{i}|_{\partial_{\text{in}}D_{i}}|D_{j}|  & \text{ if } j\prec i \text{ and } i \in L.
\end{cases}
\end{equation}

For the second case on the right hand side, we simply use the crude bound $\sup_{D_i} p_i \leq \frac{|D_i|}{2\pi}$ from Proposition~\ref{talenti2}. For the third case we can have a better bound:  for any $i\in L$, by Lemma~\ref{talenti_2} and \eqref{eq15}, there exists an $\epsilon>0$ that only depends on $\mathcal{A}(\interior(\Gamma))$ (and in particular is independent of $n$), such that $p_{i}|_{\inb D_{i}}\le(\frac{1}{2\pi}-\epsilon)|D_{i}|.$ Thus \eqref{eq29} now becomes
\begin{align}\label{eq29_1}
\int_{\partial D_{i}}p_{i}\nabla (1_{D_{j}}*\mathcal{N}) \cdot \vec n d\sigma \begin{cases}= 0 \quad &\text{ if } j \not\prec i,\\
\ge -\frac{1}{2\pi} |D_i| |D_j|   & \text{ if } j\prec i \text{ and } i \not\in L,\\
\geq -(\frac{1}{2\pi}-\epsilon) |D_i| |D_j| & \text{ if } j\prec i \text{ and } i \in L.
\end{cases}
\end{align}

Now we are ready to estimate $F_2$. Let us break it into
\[
\begin{split}
F_{2}&=\sum_{\underset{(i,j)\not\in L\times L}{j\prec i}}  \alpha_{i}\alpha_{j}\int_{\partial D_{i}}p_{i}\nabla(1_{D_{j}}*\mathcal{N})\cdot\vec{n}d\sigma + \sum_{\underset{(i,j)\in L\times L}{j\prec i}}  \alpha_{i}\alpha_{j}\int_{\partial D_{i}}p_{i}\nabla(1_{D_{j}}*\mathcal{N})\cdot\vec{n}d\sigma\\
&\geq -\sum_{\underset{(i,j)\not\in L\times L}{j\prec i}}  \alpha_i \alpha_j \frac{1 }{2\pi} |D_i| |D_j| - \sum_{\underset{(i,j)\in L\times L}{j\prec i}} \alpha_i \alpha_j \Big( \frac{1 }{2\pi}-\epsilon\Big) |D_i| |D_j| \\
\end{split}
\]where the first equality follows from case 1 of \eqref{eq29_1}, and the second inequality follows from case 2,3 of \eqref{eq29_1}. Finally, by exchanging  $i$ with $j$ and taking average with the original inequality, we have
\begin{equation}\label{eq18}
\begin{split}
F_{2}&\geq -\frac{1}{4\pi}\sum_{\underset{(i,j)\not\in L\times L}{i \neq j}}(\mathbbm{1}_{i\prec j} + \mathbbm{1}_{j\prec i})  \alpha_i \alpha_j |D_i| |D_j| - \frac12 \sum_{\underset{(i,j)\in L\times L}{i \neq j}}(\mathbbm{1}_{i\prec j} + \mathbbm{1}_{j\prec i}) \alpha_i \alpha_j \Big( \frac{1 }{2\pi}-\epsilon\Big) |D_i| |D_j| \\
&\geq -\frac{1}{4\pi}\sum_{\underset{(i,j)\not\in L\times L}{i \neq j}}  \alpha_i \alpha_j |D_i| |D_j| - \frac12 \sum_{\underset{(i,j)\in L\times L}{i \neq j}} \alpha_i \alpha_j \Big( \frac{1 }{2\pi}-\epsilon\Big) |D_i| |D_j| \\
& = - \frac{1 }{4\pi} \sum_{i\neq j} \alpha_i \alpha_j  |D_i| |D_j| + \frac{\epsilon}{2} \sum_{\underset{(i,j)\in L\times L}{i\neq j}}  \alpha_i \alpha_j  |D_i| |D_j|,
\end{split}
\end{equation}
where the second inequality is due to the fact that for any $i\neq j$, at most one of $i \prec j$ and $j\prec i$ can be true, thus we always have $\mathbbm{1}_{i\prec j} + \mathbbm{1}_{j\prec i} \leq 1$.

Therefore, from \eqref{eq17} and \eqref{eq18} it follows that
\begin{align}
F_{1}+F_{2}&\ge -\frac{1}{4\pi}\sum_{i,j=1}^{M_n}\alpha_{i}\alpha_{j}|D_{i}||D_{j}|+\frac{\epsilon}{2}\sum_{(i,j)\in L\times L}\alpha_{i}\alpha_{j}|D_{i}||D_{j}|-\frac{\epsilon}{2}\sum_{i\in L}\alpha_{i}^{2}|D_{i}|^{2}\nonumber\\
&=-\frac{1}{4\pi}\left(\sum_{i=1}^{M_n}\alpha_{i}|D_{i}|\right)^{2}+\frac{\epsilon}{2}\left(\sum_{i\in L}\alpha_{i}|D_{i}|\right)^{2}-\frac{\epsilon}{2}\sum_{i\in L}\alpha_{i}^{2}|D_{i}|^{2}.
\end{align}
Since we will send $n\to\infty$, in the rest of step 3 we will denote $L$ by $L^n$ to emphasize that $L$ depends on $\omega_n$. (In fact $\alpha_i$ and $D_i$ depend on $n$ as well, and we omit the $n$ dependence for them to avoid overcomplicating the notations.)

Note that $\sum_{i\in L^n}\alpha_{i}1_{D_{i}}$ converges to $\omega1_{V}$ in $L^1(\R^{2})$. Also if $i\in L^n$, then the nondegeneracy of $|\nabla \omega|$ on $V$ yields that $\lim_{n\to\infty}\sup_{i\in L^n}|D_{i}|=0$, consequently
\begin{align*}
\lim_{n\to\infty}\sum_{i\in L^n}\alpha_{i}^{2}|D_{i}|^{2}\le \| \omega\|_{L^{\infty}}\lim_{n\to\infty}\sup_{i\in L^n}|D_{i}|\int_{\R^{2}}\omega dx=0.
\end{align*}
Therefore it follows that
\begin{align}
\liminf_{n\to\infty}\mathcal{I}^n_2&=\liminf_{n\to\infty}\left(F_{1}+F_{2}\right)\nonumber\\
&\ge-\lim_{n\to\infty}\frac{1}{4\pi}\left(\sum_{i=1}^{M_n}\alpha_{i}|D_{i}|\right)^{2}+\lim_{n\to\infty}\frac{\epsilon}{2}\left(\sum_{i\in L^n}\alpha_{i}|D_{i}|\right)^{2}\nonumber\\
&= -\frac{1}{4\pi}\left(\int_{\R^{2}}\omega(x)dx\right)^{2}+\frac{\epsilon}{2}\left(\int_{V}\omega(x)dx\right)^{2}.\label{eq19}
\end{align}
Note that $\omega$ is strictly positive in $V$, due to $|\nabla \omega|>0$ in $V$ and $\omega\geq 0$ in $\mathbb{R}^2$. Thus from \eqref{eq20}, \eqref{eq21} and \eqref{eq19}, it follows that
\begin{equation}\label{final_contra}
0=\lim_{n\to\infty}\mathcal{I}^n\ge\lim_{n\to\infty}\mathcal{I}^n_1+\liminf_{n\to\infty}\mathcal{I}^n_2\ge \frac{\epsilon}{2}\left(\int_{V}\omega(x)dx\right)^{2}>0,
\end{equation}
which is a contradiction and we have proved that any level set component is a circle.\\

\textbf{Step 4.} In this step we show that every pair of level set components are nested. Towards a contradiction, assume that there exist two level set components $\Gamma_{1}$ and $\Gamma_{2}$ that are not nested.

From step 3, we know that $\Gamma_{1}$ and $\Gamma_{2}$ are circles.  Then the disks $\interior(\Gamma_1)$ and $\interior(\Gamma_2)$ are disjoint, and they must be separated by a positive distance since $\Gamma_{1}$ and $\Gamma_{2}$ are level sets of regular values of $\omega$. As in step 1, using the flow map $\Phi_t$ originating from the two circles, we can find disjoint open annuli $V_{1}$ and $V_{2}$ such that $\Gamma_{i}\subset V_{i}$ for $i=1,2$, and both $\outb V_i$ and $\inb V_i$ are level set components of $\omega$.
 
 \indent For any $n>1$, let $\omega_n(x) = \sum_{i=1}^{M_n} \alpha_{i}1_{D_{i}}(x)$ be constructed in step 2, and let 
\[
L_{1}^n:=\left\{i:  D_i \subset V_{1}\right\} \quad \text{ and } \quad L_{2}^n:=\left\{i:  D_i \subset V_{2}\right\}.
\]
 Let $p_i$ be defined in \eqref{eq42} of step 3, and $\mathcal{I}^n$ defined in \eqref{def_i_n}. Then on the one hand, the same computations in step 3 give 
\begin{align}\label{eq25}
\lim_{n\to\infty}\mathcal{I}^n = 0 \quad\text{ and  }\quad \lim_{n\to\infty}\mathcal{I}^n_1=\frac{1}{4\pi}\left(\int_{\R^{2}}\omega(x)dx\right)^{2}.
\end{align}
 
Let $F_1$ and $F_2$ be given by \eqref{def_f1f2}. For $F_1$, the estimate \eqref{eq17} still holds. For $F_2$, using \eqref{int_crude} and Proposition~\ref{talenti2}, we have
\[
F_{2}\ge-\frac{1}{4\pi}\sum_{i\prec j \text{ or }j\prec i}\alpha_{i}\alpha_{j}|D_{i}||D_{j}|.
\]
Since $V_1$ and $V_2$ are assumed to be not nested, if $(i,j)\in L_{1}^n\times L_{2}^n$ then neither $i\prec j$ nor $j\prec i$. Therefore it follows that
\begin{align*}
F_{2}\ge -\frac{1}{4\pi}\sum_{i\ne j}\alpha_{i}\alpha_{j}|D_{i}||D_{j}|+\frac{1}{4\pi}\sum_{(i,j)\in L_{1}\times L_{2}}\alpha_{i}\alpha_{j}|D_{i}||D_{j}|+\frac{1}{4\pi}\sum_{(j,i)\in L_{1}\times L_{2}}\alpha_{i}\alpha_{j}|D_{i}||D_{j}|.
\end{align*}
Combining the estimates for $F_1$ and $F_2$ yields
\begin{align*}
F_{1}+F_{2}\ge -\frac{1}{4\pi}\sum_{i,j=1}^{M_n}\alpha_{i}\alpha_{j}|D_{i}||D_{i}|+\frac{1}{2\pi}\left(\sum_{i\in L_{1}^n}\alpha_{i}|D_{i}|\right)\left(\sum_{i\in L_{2}^n}\alpha_{i}|D_{i}|\right).
\end{align*}
As $n\to\infty$, since $\sum_{i\in L_{1}^n}\alpha_{i}1_{D_{i}}$ and $\sum_{i\in L_{2}^n}\alpha_{i}1_{D_{i}}$ converge to $\omega1_{V_{1}}$ and $\omega1_{V_{2}}$ respectively in $L^1(\mathbb{R}^2)$, we have\begin{align}
\liminf_{n\to\infty}\mathcal{I}^n_2&\ge \lim_{n\to\infty}\left(F_{1}+F_{2}\right)=-\frac{1}{4\pi}\left(\int_{\R^{2}}\omega(x)dx\right)^{2}+\frac{1}{2\pi}\left(\int_{V_{1}}\omega(x)dx\right)\left(\int_{V_{2}}\omega(x)dx\right).\label{eq26}
\end{align} 
Combining \eqref{eq25} and \eqref{eq26} gives us a similar contradiction as in \eqref{final_contra}, except that $ \frac{\epsilon}{2}\left(\int_{V}\omega(x)dx\right)^{2}$ is now replaced by $\frac{1}{2\pi}\left(\int_{V_{1}}\omega(x)dx\right)\left(\int_{V_{2}}\omega(x)dx\right)$. Thus we complete the proof that level sets are nested.

\indent \textbf{Step 5.}   In this step, we aim to show that all level set components are concentric within the \emph{same connected component} of $\supp\,\omega$. This immediately implies that each connected component of $\supp\,\omega$ is an annulus or a disk, and $\omega$ is radially symmetric about its center.

Towards a contradiction, suppose that there are two level set components $\Gamma_{\text{in}}$ and $\Gamma_{\text{out}}$ in the same connected component of $\supp\,\omega$, such that they are nested circles, but their centers $O_{\text{in}}$ and $O_{\text{out}}$ do not coincide. We denote their radii by $r_{\text{in}}$ and $r_{\text{out}}$, and define
\[U := \overline{\interior(\Gamma_{\text{out}})}\setminus \interior(\Gamma_{\text{in}}).
\] 
For an illustration of $\Gamma_{\text{in}}$ and $\Gamma_{\text{out}}$ and $U$, see Figure~\ref{fig_dn}(a).

We claim that $\omega$ is uniformly positive in $U$. Recall that all level set components of $\omega$ are nested by step 4. Thus if $\omega$ achieves zero in $U$, the zero level set must be also nested between $\Gamma_{\text{in}}$ and $\Gamma_{\text{out}}$, since it can be taken as a limit of level set components whose value approaches 0; but this contradicts with the assumption that $\Gamma_{\text{in}}$ and $\Gamma_{\text{out}}$ lie in the same connected component of $\supp\,\omega$. As a result, we have $\omega_{\min} := \inf_U \omega>0$.

\begin{figure}[h!]
\begin{center}
\includegraphics[scale=1]{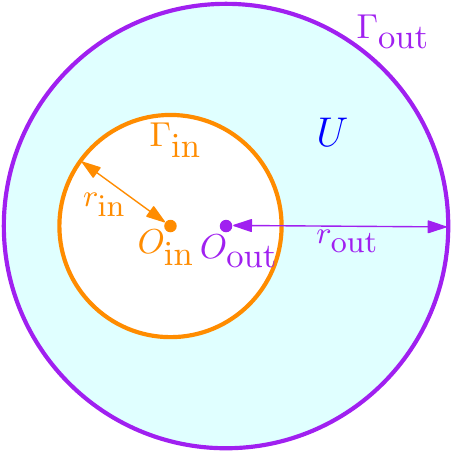}\qquad\qquad\quad\quad \includegraphics[scale=1]{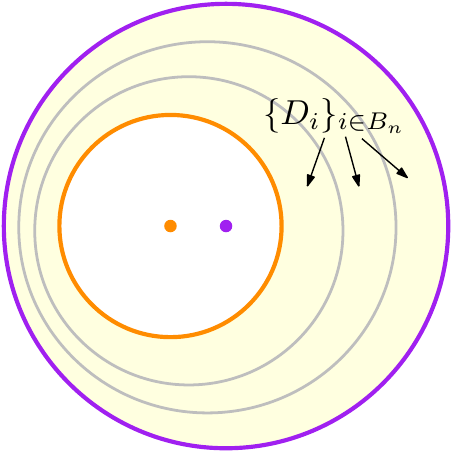}\\
(a)\hspace*{6.7cm}(b)
\end{center}

\vspace*{-0.4cm}
\caption{(a) Illustration of the circles $\Gamma_{\text{in}}$ and $\Gamma_{\text{out}}$, whose centers are $O_{\text{in}}$ and $O_{\text{out}}$. The set $U$ is colored in blue. (b) For a fixed $n$, each open set $\{D_i\}_{i\in B_n}$ is colored in yellow. Note that their union gives exactly the set $U$.\label{fig_dn}} 
\end{figure}

For a sufficiently large $n$, let $\omega_n=\sum_{i=1}^{M_n} \alpha_i 1_{D_i}(x)$ be given in step 2, where we further require both $\Gamma_{\text{in}}$ and $\Gamma_{\text{out}}$ coincide with some boundary of $D_i$. (This is allowed in our construction of $\omega_n$ in step 2, since $\omega$ is regular along both $\Gamma_{\text{in}}$ and $\Gamma_{\text{out}}$.)
Let us denote 
\[
B_n := \{1\leq i\leq M_n: D_i\subset U\},
\]
and note that $
 U := \cup_{i \in B_n} \overline D_i.$  
See Figure~\ref{fig_dn}(b) for an illustration of $\{D_i\}_{i\in B_n}$. 

As before, we denote $i\prec j$ if $D_i$ is nested in $D_j$. 
For the integral $\mathcal{I}^n$ in \eqref{def_i_n}, on the one hand, we have $\mathcal{I}^n=0$ for all $n>1$ by \eqref{eq20}. On the other hand, following the same argument as in step 3 up to \eqref{int_crude} (where we also use that each $D_i$ is already known to be doubly-connected, thus $\int_{\partial D_i}p_i \nabla (1_{D_j} * \mathcal{N})\cdot \vec{n} d\sigma = -p_i|_{\inb D_i} |D_j|$ if $j \prec i$), we have
\[
\begin{split}
\liminf_{n\to\infty} \mathcal{I}^n &= \liminf_{n\to\infty}\left( \frac{1}{4\pi} \Big(\sum_{i=1}^{M_n} \alpha_i |D_i|\Big)^2 - \sum_{i=1}^{M_n} \alpha_i^2 \int_{D_i} p_i dx - \sum_{1\leq i,j \leq M_n, j\prec i} \alpha_i \alpha_j p_i\big|_{\inb D_i} |D_j| \right)\\
&\geq  \liminf_{n\to\infty}\bigg(  \underbrace{\sum_{1\leq i,j\leq M_n,\, j\prec i}  \alpha_i \alpha_j\Big( \frac{1}{2\pi} |D_i| - p_i\big|_{\inb D_i}\Big)|D_j|  }_{=:T_n} \bigg),
\end{split}
\]
where in the last step we used  Proposition~\ref{talenti2}.

Note that Proposition~\ref{talenti2} gives  $T_n> 0$, where we have strict positivity, since $O_{\text{in}} \neq O_{\text{out}}$ implies that some $\{D_i\}_{i\in B_n}$ must be non-radial. But since the area of these $D_i$'s may approach 0 as $n\to\infty$, in order to derive a contradiction after taking $\liminf_{n\to\infty}$, we need to obtain a quantitative estimate for Proposition~\ref{talenti2} for a thin tubular region $D_i$ between two circles, which is done in Lemma~\ref{talenti_3}. 

Next we show that the sets $\{D_i\}_{i\in B_n}$ that are ``non-radial to some extent'' must occupy a certain portion of $ U$. For $i\in B_n$, denote by $o^i_{\text{in}}$ and $r^i_{\text{in}}$ the center and radius of $\inb D_i$, and likewise $o^i_{\text{out}}$ and $r^i_{\text{out}}$ the center and radius of $\outb D_i$. Note that if $D_i$ is the inner-most set in $\{D_i\}_{i\in B_n}$, then we have $o^i_{\text{in}} = O_{\text{in}}$, and the outmost $D_i$ satisfies $o^i_{\text{out}} =  O_{\text{out}}$. In addition, if  $\outb D_i=\inb D_j$ for some $i,j \in B_n$, then $o^i_{\text{out}}$ = $o^j_{\text{in}}$. Thus triangle inequality gives
\begin{equation}\label{eq_sum_o}
  \sum_{i\in B_n} |o^i_{\text{in}} - o^i_{\text{out}}| \geq | O_{\text{in}} - O_{\text{out}}| =: c_0>0 \quad \text{ for all }n\geq n_0.
\end{equation}

In order to apply Lemma~\ref{talenti_3} (which requires the region to have inner radius 1), for each $i \in B_n$, consider the scaling 
\[\tilde p_i(x) := (r_{\text{in}}^i)^{-2} p_i ( r_{\text{in}}^i x).
\] Then $\tilde p_i$ is defined in $\tilde D_i :=  (r_{\text{in}}^i)^{-1}D_i$. 
Due to the scaling, $\tilde D_i$ has inner radius 1 (denote the hole by $\tilde h_i$), and outer radius $1+\epsilon_i$, where $ \epsilon_i := \frac{r_{\text{out}}^i- r_{\text{in}}^i}{r_{\text{in}}^i}>0$. In addition, the distance between the centers of $\inb\tilde D_i$ and $\outb\tilde D_i$ is $ a_i \epsilon_i$,
where \[
 a_i := \frac{|o^i_{\text{in}} - o^i_{\text{out}}|}{|r^i_{\text{in}}-r^i_{\text{out}}|}.
 \] One can also easily check that $\tilde p_i$ satisfies $\Delta \tilde p_i = -2$ in $\tilde D_i$, and $\int_{\partial \tilde h_i} \nabla \tilde p \cdot \vec{n} d\sigma = -2|\tilde h_i| = -2\pi$.
By Lemma~\ref{talenti_3}, if $0<\epsilon_i < \frac{(a_i)^2}{64}$, then $\tilde p_i |_{\partial \tilde h_i} \leq \frac{|\tilde D_i|}{2\pi} (1- \frac{a_i^2}{16})$. Thus in terms of $p_i$, we have
\begin{equation}\label{p_eq000}
 p_i \big|_{\inb D_i} \leq \frac{1}{2\pi} |D_i| - c_1 a_i^2 |D_i| \quad\text{ if $i\in B_n$ satisfies ~} r_{\text{out}}^i- r_{\text{in}}^i \leq c_2 a_i^2,
\end{equation}
where $c_1 := \frac{1}{32\pi}$ and $c_2 := \frac{r_{\text{in}}}{64}$ are independent of $n$ and $i$, due to the fact that $r_{\text{in}}^i \geq r_{\text{in}}>0$ for all $i\in B_n$.
Using the definition of $a_i$,  \eqref{eq_sum_o} can be written as 
\begin{equation}\label{sum_o2}
\sum_{i\in B_n} a_i  |r^i_{\text{in}}-r^i_{\text{out}}| > c_0.
\end{equation}
Note that $\sum_{i\in B_n} |r^i_{\text{in}}-r^i_{\text{out}}|$ satisfies the upper bound
\begin{equation}\label{sum_r_diff}
\sum_{i\in B_n} |r^i_{\text{in}}-r^i_{\text{out}}| \leq \frac{|U|}{2\pi r_{\text{in}}} =: M,
\end{equation}
which follows from 
\[
|U| = \sum_{i\in B_n} |D_i| = \pi \sum_{i\in B_n} |r^i_{\text{in}}-r^i_{\text{out}}| \underbrace{(r^i_{\text{in}}+r^i_{\text{out}})}_{>2r_{\text{in}}}.
\]
Combining \eqref{sum_o2} and \eqref{sum_r_diff} gives
\begin{equation}\label{sum_r2}
\sum_{i\in B_n} \mathbbm{1}_{a_i > \frac{c_0}{2M}} |r^i_{\text{in}}-r^i_{\text{out}}| \geq \sum_{i\in B_n} \Big(a_i-\frac{c_0}{2M}\Big)|r^i_{\text{in}}-r^i_{\text{out}}| \geq \frac{c_0}{2},
\end{equation}
where the first inequality follows from $\mathbbm{1}_{a_i > \frac{c_0}{2M}} \geq a_i-\frac{c_0}{2M}$ (recall that $a_i \in(0, 1)$), and the second inequality follows from subtracting $\frac{c_0}{2M}$ times  \eqref{sum_r_diff} from  \eqref{sum_o2}.

Let 
\[K_n := \Big\{i \in B_n: a_i > \frac{c_0}{2M}\Big\}.
\] Using this definition and the fact that $|D_i| > 2\pi r_{\text{in}} |r^i_{\text{in}}-r^i_{\text{out}}| $, \eqref{sum_r2} can be rewritten as
\begin{equation}\label{sum_temp}
\sum_{i\in K_n} |D_i| > 2\pi r_{\text{in}}  \sum_{i\in K_n} |r^i_{\text{in}}-r^i_{\text{out}}| \geq \pi r_{\text{in}}c_0.
\end{equation}
Now we take a sufficiently large $n$, and discuss two cases (note that different $n$ may lead to different cases):

\textbf{Case 1.} Every $i\in K_n$ satisfies $r^i_{\text{out}}-r^i_{\text{in}} \leq \min\{c_2 (\frac{c_0}{2M})^2 , \frac{ r_{\text{in}} c_0}{4 r_{\text{out}}}\}$. By definition of $K_n$, we have $r^i_{\text{out}}-r^i_{\text{in}} \leq c_2( \frac{c_0}{2M})^2 \leq c_2 a_i^2$ for $i\in K_n$. (The motivation of the second term in the min function will be made clear later.) 
  Then by \eqref{p_eq000}, we have 
\[
\frac{1}{2\pi} |D_i|  - p_i \big|_{\inb D_i} \geq  c_1 a_i^2 |D_i| \geq  \frac{c_1 c_0^2}{4M^2} |D_i| \quad\text{ for all }i\in K_n.
\] 
Since $K_n$ is a subset of $B_n$ (and recall that  $\alpha_i \geq \omega_{\min}>0$ for all $i \in B_n$), we have the following lower bound for $T_n$:
\begin{equation}\label{sum_temp2}
\begin{split}
T_n &\geq \omega_{\min}^2 \sum_{i,j\in K_n,\, j\prec i}  \Big( \frac{1}{2\pi} |D_i| - p_i\big|_{\inb D_i}\Big)|D_j| \\
&\geq \omega_{\min}^2 \sum_{i,j\in K_n,\, j\prec i}   \frac{c_1 c_0^2}{4M^2} |D_i||D_j| 
 =  \omega_{\min}^2 \sum_{i,j\in K_n,\, i\neq j}   \frac{c_1 c_0^2}{8M^2} |D_i||D_j|\\
 &= \omega_{\min}^2  \frac{c_1 c_0^2}{16 M^2} \left(\bigg(\sum_{i\in K_n} |D_i|\bigg)^2 - \sum_{i\in K_n}|D_i|^2 \right) 
\end{split}
\end{equation}
Note that the second term in the min function in the assumption gives
\[
\max_{i\in K_n}|D_i| < 2\pi r_{\text{out}} (r^i_{\text{out}}-r^i_{\text{in}}) \leq \frac{\pi r_{\text{in}} c_0}{2} \leq \frac{1}{2} \sum_{i\in K_n} |D_i|,
\]
where we use \eqref{sum_temp} in the last inequality.
 Applying this to the right hand side of \eqref{sum_temp2} gives 
\[
T_n \geq \omega_{\min}^2\frac{c_1 c_0^2}{16 M^2}\cdot \frac{1}{2} \bigg(\sum_{i\in K_n} |D_i|\bigg)^2 \geq \omega_{\min}^2 \frac{c_1 c_0^2}{32 M^2} (\pi  r_{\text{in}} c_0)^2.\]

\textbf{Case 2.}  If Case 1 is not true, then there must be some $i_0 \in K_n$ satisfying $r^{i_0}_{\text{out}}-r^{i_0}_{\text{in}} > \min\{c_2 (\frac{c_0}{2M})^2 , \frac{ r_{\text{in}} c_0}{4 r_{\text{out}}}\}=: c_3,$ which leads to
\[
|o^{i_0}_{\text{in}} - o^{i_0}_{\text{out}}| = a_{i_0} (r^{i_0}_{\text{out}}-r^{i_0}_{\text{in}} ) >\frac{c_0 c_3}{2M} =: l.
\]
Although this set $D_{i_0}$ is likely not thin enough for us to apply Lemma~\ref{talenti_3}, since $|o^{i_0}_{\text{in}} - o^{i_0}_{\text{out}}| $ is bounded below by a positive constant independent of $n$, we can still use Lemma~\ref{talenti_4} to conclude that $\frac{1}{2\pi} |D_{i_0}|  - p_{i_0} \big|_{\inb D_{i_0}} \geq c_4$ for some $c_4>0$ only depending on $r_{\text{in}}, r_{\text{out}}$ and $l$.
This leads to
\[
T_n \geq \sum_{{i_0}\succ j} \omega_{\min} \alpha_j c_4 |D_j| \geq \omega_{\min} c_4 \hspace*{-0.3cm}\sum_{D_j\subset \interior(\Gamma_{\text{in}})} \hspace*{-0.3cm} \alpha_j  |D_j| \geq \omega_{\min} c_4 \cdot \frac{1}{2} \int_{\interior(\Gamma_{\text{in}})} \omega dx,
\]
where the last inequality follows from the fact that for all sufficiently large $n$, the definition of $\omega_n$ gives $\sum_{D_j\subset \interior(\Gamma_{\text{in}})}\alpha_j |D_j| = \int_{\interior(\Gamma_{\text{in}})} \omega_n dx \geq \frac{1}{2} \int_{\interior(\Gamma_{\text{in}})} \omega dx$. Note that the last integral is positive since $\omega>0$ on $\Gamma_{\text{in}}$, and it is clearly independent of $n$.

From the above discussion, for all sufficiently large $n$, regardless of whether we are in Case 1 or 2 for this $n$, we always have that $T_n$ is bounded below by some uniformly positive constant independent of $n$. Therefore taking the $n\to\infty$ limit gives
\[
\liminf_{n\to\infty} \mathcal{I}^n \geq \liminf_{n\to\infty} T_n > 0.
\]
This contradicts $\mathcal{I}^n=0$, therefore finishing the proof of step 5.

\indent \textbf{Step 6.} It remains to show that all connected components of $\supp\,\omega$ are concentric.  
If $\supp\,\omega$ has finitely many connected components, we could proceed similarly as the end of the proof of  Corollary~\ref{cor_multiple_steady}. But since $\supp\,\omega$ may have countably many connected components, we need to use a different argument. 

Let us denote the connected components of $\supp\,\omega$ by $\{U_i\}_{i\in I}$, where $I$ may have countably many elements. Denote their centers by $\{o_i\}_{i\in I}$, their radii by $\{R_i\}_{i\in I}$, and their outer boundaries by $\{\outb U_i\}_{i\in I}$. Without loss of generality, suppose the $x$-coordinates of their centers $\{o^1_i\}_{i\in I}$ are not all identical. 

Among the (possibly infinitely many) collection of circles $\{\outb U_i\}_{i\in I}$, let $\Gamma_r$ be the ``circle with rightmost center'' among them, in the following sense: 

$\bullet$ If there exists some $i_0 \in I$ such that $o^1_{i_0} = \sup_{i\in I} o^1_i$, we define $\Gamma_r := \outb U_{i_0}$. (If the supremum is achieved at more than one indices, we set $i_0$ to be any of them.) 

$\bullet$ Otherwise, take any subsequence $\{i_k\}_{k\in \mathbb{N}} \subset I$ such that $\lim_{k\to\infty} o^1_{i_k} = \sup_{i\in I} o^1_i$. Since $\omega$ has compact support, we can extract a further subsequence (which we still denote by $\{i_k\}_{k\in\mathbb{N}}$), such that both $o_{i_k}$ and $r_{i_k}$ converge as $k\to\infty$, and denote their limit by $O_r \in \mathbb{R}^2$ and $R_r \in \mathbb{R}$. Finally let $\Gamma_r:= \partial B(O_r, R_r)$. 

With the above definition, we point out that $f := \omega * \mathcal{N}= \text{const}$ on $\Gamma_r$.  Note that in both cases above, we can find a sequence of level set components of $\omega$ that converges to $\Gamma_r$, in the sense that the Hausdorff distance between the two sets goes to 0. Since  $f= \text{const}$ on each level set component of $\omega$, continuity of $f$ gives that $f = \text{const}$ on $\Gamma_r$. 

Let $f_i(x) := (\omega 1_{U_i}) * \mathcal{N}$ for $i\in I$; note that by definition we have $f = \sum_{i\in I} f_i$. Lemma~\ref{lem_newton} gives the following:

(a) For all $x \in (\interior( \outb U_i))^c$, we have $f_i(x) = \frac{1}{2\pi} \big(\int_{U_i} \omega dx\big) \ln |x-o_i|$.

(b) If $U_i$ is doubly-connected, then $f_i =\text{const}$ in $\interior (\inb U_i)$, where the constants are different for different $i$.

Note that for any $i\in I$,  $U_i$ must be either nested inside $\Gamma_r$, or have $\Gamma_r$ nested in its hole. (By a slight abuse of notation, we use $i\prec \Gamma_r$ and $i\succ \Gamma_r$ to denote these two relations.) Let $\Gamma_{r}^R := (O_r^1+R_r, O_r^2)$ and $\Gamma_{r}^L:=(O_r^1-R_r, O_r^2)$ be the rightmost/leftmost point of the circle $\Gamma_r$. Note that  (b) implies $f_i(\Gamma_{r}^R) = f_i(\Gamma_{r}^L)$ for all $i\succ \Gamma_r$, whereas  (a) gives the following for all $i\prec \Gamma_r$:
\[
f_i(\Gamma_{r}^R) = \frac{\int_{U_i} \omega dx}{2\pi} \ln |\Gamma_{r}^R-o_i| \geq  \frac{\int_{U_i} \omega dx}{2\pi} \ln |\Gamma_{r}^L-o_i| =  f_i(\Gamma_{r}^L),
\] 
where the inequality follows from that $|O_r^1+R_r - o^1_i| \geq |O_r^1-R_r - o^1_i|$, which is a consequence of $o^1_i \leq O_r^1$ due to our choice of $O_r$. (Also note that $\Gamma_{r}^R$ and $\Gamma_{r}^L$ have the same $y$-coordinate.)

As a result, summing over all $i\in I$ gives $f(\Gamma_{r}^R) \geq f(\Gamma_{r}^L)$, where the equality is achieved if and only if $o_i^1 = O_r$ for all $i \prec \Gamma_r$. Now we discuss two cases:

\textbf{Case 1.} There is some $i\prec \Gamma_r$ with $o_i^1 < O_r$. In this case the above discussion gives $f(\Gamma_{r}^R) > f(\Gamma_{r}^L)$, which directly leads to a contradiction to $f =\text{const}$ on $\Gamma_r$.

\textbf{Case 2.} If case 1 does not hold, then let us define $\Gamma_l$ as a ``circle with leftmost center'' among $\{\outb U_i\}_{i\in I}$ in the same way as $\Gamma_r$.
 Then we must have $O_l^1 < O_r^1$, and since case 1 does not hold (i.e. all $i\prec \Gamma_r$ satisfy that $o_i^1 = O_r$), we must have $\Gamma_l \succ \Gamma_r$. By definition of $\Gamma_r$, there exists some $U_{i_0}$ whose outer boundary is sufficiently close to $\Gamma_r$ and center sufficiently close to $O_r$. As a result, $i_0 \prec \Gamma_l$ and $o_{i_0}^1 > O_l^1$.
 
Let $\Gamma_l^L$ and $\Gamma_l^R$ be the leftmost/rightmost point of $\Gamma_l$. A parallel argument as above then gives that 
$
f_i(\Gamma_{l}^L) \geq f_i(\Gamma_{l}^R)
$ for all $i\in I$. Since we have found an $i_0 \prec \Gamma_l$ with $o_{i_0}^1 > O_l^1$, we have $f_{i_0}(\Gamma_{l}^L) > f_{i_0}(\Gamma_{l}^R)$, thus summing over all $i\in I$ gives the strict inequality $f(\Gamma_{l}^L) >f(\Gamma_{l}^R)$, contradicting with $f =\text{const}$ on $\Gamma_l$. 

In both cases above we have obtained a contradiction, thus $\{o_i\}_{i\in I}$ must  have the same $x$-coordinate. An identical argument shows that their $y$-coordinate must also be identical, thus $\{U_i\}_{i\in I}$ are concentric. Since $\omega$ is known to be radial within each $U_i$ (about its own center) in step 1--5, the proof is now finished.
\end{proof}

\color{black}

In the next corollary, we show that the above proof for stationary smooth solutions can be extended (with some modifications) to show radial symmetry of nonnegative rotating smooth solutions with $\Omega < 0$.

\begin{cor}\label{cor_conti_rotating}
Let $\omega(x,t)=\omega_0(R_{\Omega t}x)$ be a nonnegative compactly-supported uniformly-rotating  solution of  2D Euler with angular velocity $\Omega< 0$. Then $\omega_0$ is radially symmetric about the origin.
\end{cor}
\begin{proof} The proof is very similar to the proof of Theorem~\ref{theorem3}, and we only highlight the differences. 
Let $\{\omega_n\}$ be the same approximation for $\omega_0$ as in step 2 of Theorem~\ref{theorem3}. We consider the same setting as in \eqref{eq42} and \eqref{def_i_n}, except with $f(x)$ replaced by $f_\Omega(x):=\omega_0*\mathcal{N}-\frac{\Omega}{2}|x|^2.$ From the assumption on $\omega_0$, we have that $f_\Omega$ is a constant on each level set component of $\omega_0$. Thus the same computations in \eqref{eq20} give $\mathcal{I}^{n}=0$ for all $n>1$.

On the other hand, we have
\begin{align}\label{I^n identity}
\mathcal{I}^n=\int_{\R^2}\omega_n \nabla \vphi^n \cdot \nabla \left(\omega_0*\mathcal{N}\right)dx+({-\Omega})\int_{\R^2}\omega_n\nabla \vphi^n\cdot xdx =: \mathcal{I}^n_1 + \underbrace{(-\Omega)}_{\geq 0}\mathcal{I}^n_2.
\end{align}
The same argument as in \eqref{ineq_i2_rotate} of Theorem~\ref{thm_omega<0}  gives that $\mathcal{I}^n_2\geq 0$. As for $\mathcal{I}^n_1$, in step 3 -- step 5 of the proof of Theorem~\ref{theorem3}, we have already shown that $\liminf_{n\to\infty} \mathcal{I}^n_1\geq 0$, and the equality is achieved if and only if each connected component of $\{\omega_0>0\}$ is radially symmetric up to a translation, and they are all nested. 


Let us decompose $\supp\, \omega_0$ into (possibly infinitely many) connected components $\cup_{i\in I} U_i$, with centers $\{o_i\}_{i\in I}$. Our goal is to show $o_i \equiv (0,0)$ for $i\in I$. Note that it suffices to show that their $x$-coordinates satisfy $\sup_{i\in I} o_i^1\leq 0$. Once we prove this, a parallel argument gives $\inf_{i\in I}o_i^1 \geq 0$, which implies $o_i^1\equiv 0$ for $i\in I$, and the same can be done for the $y$-coordinate.

Towards a contradiction, suppose $\sup_{i\in I} o_i^1 > 0$. We can then define $\Gamma_r$ in the same way as step 6 of the proof of Theorem~\ref{theorem3}, i.e. it is the ``circle with rightmost center'' among $\{\outb U_i\}_{i\in I}$, and its center $O_r$ satisfies $O_r^1=\sup_{i\in I} o_i^1>0$. Since $f_\Omega=\text{const}$  along each level set component of $\omega_0$, we again have that $f_\Omega = \text{const}$ on $\Gamma_r$. Let $\Gamma_r^R$ and $\Gamma_r^L$ be the rightmost/leftmost points on $\Gamma_r$. Note that their distances to the origin satisfy $|\Gamma_r^R| > |\Gamma_r^L|$, where the strict inequality is due to the assumption $O_r^1>0$.

Let us define $f_i(x) = (\omega_0 1_{U_i})*\mathcal{N}$ for $i\in I$, and note that $f_\Omega = (\sum_{i\in I} f_i) - \Omega |x|^2$. The properties (a,b) in step 6 of Theorem~\ref{theorem3} still hold for $f_i$, thus we have $f_i(\Gamma_r^R) \geq f_i(\Gamma_r^L)$ for all $i\in I$. This leads to
\[
f_\Omega(\Gamma_r^R) = \Big( \sum_{i\in I} f_i(\Gamma_r^R)\Big) + \underbrace{(-\Omega)}_{>0}\big| \Gamma_r^R\big|^2 > \Big( \sum_{i\in I} f_i(\Gamma_r^L)\Big) + (-\Omega)\big| \Gamma_r^L\big|^2 = f_\Omega(\Gamma_r^L),
\]
contradicting the fact that $f_\Omega \equiv \text{const}$ on $\Gamma_r$.  \color{black}
\end{proof}

\subsection{Proofs of the quantitative lemmas}
\label{sec21}

Before the proof of Lemma~\ref{talenti_2}, let us state two lemmas that we will use in the proof. The first one is a quantitative version of the isoperimetric inequality obtained by Fusco, Maggi and Pratelli \cite{Fusco-Maggi-Pratelli:stability-faber-krahn}. 
\begin{lem}[{c.f.~\cite[Section 1.2]{Fusco-Maggi-Pratelli:stability-faber-krahn}}] \label{Fraenkel2} Let $E\subseteq \R^{2}$ be a bounded domain. Then there is some constant $c\in (0,1),$ such that
\begin{align*}
P(E)\ge 2\sqrt{\pi}\sqrt{|E|}\left(1+c\mathcal{A}(E)^{2}\right),
\end{align*} 
where $P(E)=\mathcal{H}^{1}(\partial E)$ denotes the perimeter of $E$.
\end{lem}

The second lemma is a simple result relating the Fraenkel asymmetry of a set $E$ with its subsets $U$. 
\begin{lem}[c.f.~{\cite[Lemma 4.4]{Craig-Kim-Yao:aggregation-newtonian}}]\label{CKY_lem4.4} Let $E\subseteq\R^{2}$ be a bounded domain. For all $U\subseteq E$ satisfying $|U|\ge|E|(1-\frac{\mathcal{A}(E)}{4}),$ we have
\begin{align*}
\mathcal{A}(U)\ge\frac{\mathcal{A}(E)}{4}.
\end{align*}
\end{lem}

\begin{proof}[\textup{\textbf{Proof of Lemma~\ref{talenti_2}}}] The proof of the Lemma~\ref{talenti_2} is similar to \cite[Proposition~4.5]{Craig-Kim-Yao:aggregation-newtonian} obtained by Craig, Kim and the last author. For the sake of completeness, we give a proof below.  Let $g(k)$, $D_{k}$ and $\tilde{D}_{k}$ be defined as in Proposition~\ref{talenti2}, let $\tilde D = D \cup \overline{h}$ and define $p_h := p|_{\partial h}$. We start by following the proof of Proposition~\ref{talenti2}, except that after obtaining \eqref{eq28}, instead of using the isoperimetric inequality, we use the stability version in Lemma~\ref{Fraenkel2} to control $P(\tilde D_k)$. This gives
\begin{align*}
g'(k)\big(g(k)+|h|\mathbbm{1}_{p_h>k}\big)&\le-\frac{1}{2}P(\tilde{D}_{k})^{2}\\
&\le-2\pi |\tilde{D}_{k}|\left(1+c\mathcal{A}(\tilde{D}_{k})^{2}\right)^{2}\\
&\le -2\pi\big(g(k)+|h|\mathbbm{1}_{p_h>k}\big)\left(1+c\mathcal{A}(\tilde{D}_{k})^{2}\right).
\end{align*}



Hence it follows from Lemma~\ref{CKY_lem4.4} that
\begin{align}\label{eq41}
g'(k)\le-2\pi\Big(1+c\frac{\mathcal{A}(\tilde D)^{2}}{16}\Big) \ \text{ for all $k$ such that } |\tilde{D}_{k}|\ge |\tilde D|\Big(1-\frac{\mathcal{A}(\tilde D)}{4}\Big).
\end{align}
We claim that
\begin{align}\label{eq39}
g(k)\le |D|-2\pi\Big(1+c\frac{\mathcal{A}(\tilde D)^{2}}{16}\Big)k \quad\text{ for }k \leq \min\Big\{p_h,\frac{\mathcal{A}(\tilde D)|\tilde D|}{16\pi}\Big\}.
\end{align}
Towards a contradiction, suppose there is $k_{0}\leq  \min\left\{p_h,\frac{\mathcal{A}(\tilde D)|\tilde D|}{16\pi}\right\}$ such that \eqref{eq39} is violated.
Since $1+c\frac{\mathcal{A}(\tilde D)^{2}}{16}\le 2$, we have
\begin{align*}
g(k_{0})>|D|-4\pi k_{0}\ge|D|-\frac{\mathcal{A}(\tilde D)|\tilde D|}{4},
\end{align*}
therefore
\begin{align*}
|\tilde{D}_{k_{0}}|&=g(k_{0})+|h|\\
&>|D|-\frac{\mathcal{A}(\tilde D)|\tilde D|}{4}+|h|\\
&=|\tilde D|-\frac{\mathcal{A}(\tilde D)|\tilde D|}{4}=|\tilde D|\Big(1-\frac{\mathcal{A}(\tilde D)}{4}\Big).
\end{align*}
Hence for all $k\in(0, k_{0}]$, $g'(k)$ satisfies the inequality \eqref{eq41}. Thus we have
\begin{align*}
g(k_{0})&\le\int_{0}^{k_{0}}-2\pi\Big(1+c\frac{\mathcal{A}(\tilde D)^{2}}{16}\Big)dk+|D|\\
&=|D|-2\pi\Big(1+c\frac{\mathcal{A}(\tilde D)^{2}}{16}\Big)k_{0},
\end{align*}
contradicting our assumption on $k_0$.

Finally, to control $p_h$, we discuss two cases below, depending on which one in the minimum function in \eqref{eq39} is smaller. For simplicity, we denote $A:=\frac{\mathcal{A}(\tilde D)|\tilde D|}{16\pi}$ and $B:=c\frac{\mathcal{A}(\tilde D)^{2}}{16}$.\\
\newline
\indent \textbf{Case 1:} $p_h\le A$. In this case \eqref{eq39} holds for all $k\le p_h$, thus
\begin{align*}
&0\le g(p_h)\le |D|-2\pi\left(1+B\right)p_h,
\end{align*}
implying
\[
 p_h\le \frac{|D|}{2\pi\left(1+B \right)}\le\frac{|D|}{2\pi}(1-c_{0}),
\]
for some constant $c_{0}$ which only depends on $\mathcal{A}(\tilde D)$. \\
\newline
\indent \textbf{Case 2:} $p_h>A$. In this case \eqref{eq39} gives $g(A)\le |D|-2\pi(1+B)A$ and we use a crude bound for $k\ge A$ that is $g'(k)\le -2\pi$. Therefore for $k>A$,
\[
\begin{split}
g(k)&=\int_{A}^{k}g'(k)dk+g(A)\le -2\pi(k-A)+|D|-2\pi(1+B)A\\
&=|D|-2\pi k -2\pi AB\\
&\le |D|\Big(1-\frac{\mathcal{A}(\tilde D)}{8}B\Big)-2\pi k,
\end{split}
\]where the last inequality follows from $A>\frac{\mathcal{A}(\tilde D)|D|}{16\pi}$.
Plugging in $k=p_h$ gives
\[0\le g(p_h)\le |D|\Big(1-\frac{\mathcal{A}(\tilde D)}{8}B\Big)-2\pi p_h,
\]
leading to \[
 p_h\le\frac{|D|}{2\pi}(1-c_{0}),
\]
again $c_{0}$ only depends on $\mathcal{A}(\tilde D)$.
\end{proof}

Next we prove Lemma~\ref{talenti_3}.


\begin{proof}[\textup{\textbf{Proof of Lemma~\ref{talenti_3}}}]
Without loss of generality, we can assume that $o_1=(0,0)$ and $o_2=(a\epsilon,0)$. To estimate $ p|_{\partial B_1}$, we decompose $p$ into 
\[
p=p\big|_{\partial B_1} g + u,
\] where $g$ satisfies 
\begin{equation}\label{g_def}
  \begin{cases}
  \lap g=0 & \text{ in $D$}\\
  g=1 & \text{ on $\partial B_1$}\\
  g=0 & \text{ on $\partial B_2$},
  \end{cases}
  \end{equation} 
  and $u$ satisfies
  \begin{equation}\label{u_def}
  \begin{cases}
  \lap u=-2 & \text{ in $D$}\\
  u=0 & \text{ on $\partial D$.}
  \end{cases}
  \end{equation} 
Using this decomposition as well as the definition of $p$, we have
  \[
  -2|B_1|=\int_{\partial B_1}\nabla p\cdot \vec{n}d\sigma= p\big|_{\partial B_1} \int_{\partial B_1}\nabla g\cdot \vec{n}d\sigma+\int_{\partial B_1}\nabla u\cdot \vec{n}d\sigma,
  \] 
 where $\vec{n}$ is the outer-normal of $B_1$ throughout this proof. Thus
 \begin{equation}\label{p_in_temp}
   p\big|_{\partial B_1} =\frac{1}{\int_{\partial B_1}\nabla g\cdot \vec{n}d\sigma}\left(-2\pi-\int_{\partial B_1}\nabla u\cdot \vec{n}d\sigma\right).
\end{equation}
To estimate $p\big|_{\partial B_1} $, it remains to estimate the two integrals in \eqref{p_in_temp}.

 The function $g$ can be explicitly constructed using the conformal mapping from $D$ to a perfect annulus centered at 0.  Consider the M\"obius map $h:\mathbb{C}\to\mathbb{C}$ given by
 \[
 h(z):=\frac{z+b}{1+bz},
 \] 
where $b\in \mathbb{R}$ will be fixed soon. Note that the unit circle and the real line are both invariant under $h$, and $\partial B_2$ is mapped to some circle centered on the real line. In order to make $h(\partial B_2)$ centered at 0, since the left/right endpoints of $\partial B_2$ are $\pm(1+\epsilon)+a\epsilon$, we look for $b\in \mathbb{R}$ that solves
 \begin{equation}\label{b_def}
 h(1+\epsilon+a\epsilon)=-h(-1-\epsilon+a\epsilon).
 \end{equation}
 Plugging the definition of $h$ into the above equation, we know that $b$ is a root of the quadratic polynomial
 \[
 f(b) := b^2 -\frac{2 + (1-a^2)\epsilon}{a} b + 1.
 \]
 Clearly, for $0<a<1$, $f$ has two positive roots whose product is 1, thus one is in $(0,1)$ and the other in $(1,+\infty)$. We define $b$ to be the root in $(0,1)$. One can easily check that $f(a)<0$, and $f(\frac{a}{2})>0$ if $a^2 > 2(1-a^2)\epsilon$, which is true due to our assumption $a^2>64\epsilon$. Thus for all $\epsilon\in (0,\frac{a^2}{64})$ we have
 \[
 0<\frac{a}{2} < b < a < 1.
 \]
 
Note that $h$ is holomorphic in $\mathbb{C}$ except at the two singularity points $-b$ and $-\frac{1}{b}$. We have already shown that $-b \in B_1$, thus it is outside of $D$. Next we will show that $-\frac{1}{b}\in B_2^c$, thus is also outside of $D$. To see this,  note  that 
 \[
 \dfrac{-1-\epsilon+a\epsilon+b}{1+b(-1-\epsilon+a\epsilon)} = h(-1-\epsilon+a\epsilon) = - h(1+\epsilon+a\epsilon)=-\frac{1+\epsilon+a\epsilon+b}{1+b(1+\epsilon+a\epsilon)}<0,
 \] where the inequality follows from the fact that $a,b,\epsilon>0$. Since the numerator of the left hand side is already known to be negative due to $a,b\in(0,1)$, its denominator must be positive, leading to 
$ -\frac{1}{b}<-1-\epsilon+a\epsilon$, i.e. $-\frac{1}{b} \in B_2^c$.
 
 Now we define $g:\mathbb{R}^2 \setminus \{(-b,0) \cup (-1/b,0)\}\to\mathbb{R}$ as
\[  g(x):=-\frac{1}{\log|h(1+\epsilon+a\epsilon)|}\log|h(z)|+1 \quad\text{ for $z=x_1+ix_2$}.
 \]
 Let us check that $g$ indeed satisfies \eqref{g_def}: first note that $g$ satisfies the boundary conditions in \eqref{g_def}, since $h$ maps $D$ to a perfect annulus centered at the origin, whose inner boundary is $\partial B_1$. In addition, $g$ is harmonic in $\mathbb{R}^2 \setminus \{(-b,0) \cup (-1/b,0)\}$, thus harmonic in $D$.

Using the explicit formula of $g$, we have
 \[
 \Delta g(x) = -\frac{2\pi}{\log|h(1+\epsilon+a\epsilon)|} \Big( \delta_{(-b,0)}(x) - \delta_{\left(-\frac{1}{b},0\right)}(x) \Big) 
 \]
 in the distribution sense. We can then apply the divergence theorem to $g$ in $B_1$, and compute the integral containing $g$ in \eqref{p_in_temp} explicitly as
 \begin{equation}\label{g_int}
  \int_{\partial B_1}\nabla g\cdot \vec{n}d\sigma =-\frac{2\pi}{\log|h(1+\epsilon+a\epsilon)|}.
 \end{equation}
As for the integral containing $u$ in \eqref{p_in_temp}, we  compare $u$ with a radial barrier function 
\[
w(x):=-2(|x|-1)(|x|-1-2\epsilon),
\] which satisfies $w=0$ on $\partial B_1$ and $w>0$ on $\partial B_2$. 
Note that
\begin{align*}
\lap w=\left(\partial_{rr}+\frac{1}{r}\partial_{r}\right)w=-8+\frac{4+4\epsilon}{r}\le-2 \quad\text{ in }D,
\end{align*}
where we used that $\epsilon \in (0,\frac12)$ and $r>1$ in $D$ in the last inequality.
Thus $w-u$ is superharmonic in $D$ and nonnegative on $\partial D$, 
which allows us to apply the classical maximum principle to obtain $u\le w$ in $\overline{D}$. Combining this with the fact that $u = w = 0$ on $\partial B_1$, we have
\[\nabla u(x)\cdot \vec{n}(x)\leq \nabla w(x)\cdot \vec{n}(x)=\frac{d}{dr}w(r)\Big|_{r=1}=4\epsilon \quad\text{ for all }x\in \partial B_1,
\]hence
  \begin{equation}\label{int_u}
  \int_{\partial B_1}\nabla u\cdot \vec{n}d\sigma\le 8\pi\epsilon.
  \end{equation} 
  Plugging \eqref{g_int} and \eqref{int_u} into \eqref{p_in_temp}, we obtain
\[  p\big|_{\partial B_1}\le\log(|h(1+\epsilon+a\epsilon)|)(1+4\epsilon).
\]  Since $\log s\le s-1$ for $s>1$, it follows that
\[
\begin{split}  \log|h(1+\epsilon+a\epsilon)|&\le h(1+\epsilon+a\epsilon)-1=\frac{1+\epsilon+a\epsilon+b}{1+b(1+\epsilon+a\epsilon)}-1\\
  &=\epsilon\left(1+\frac{a-2b-ab-b\epsilon-ab\epsilon}{1+b(1+\epsilon+a\epsilon)}\right)\\
  &\le \epsilon\left(1 - \frac{ a b}{4}\right)\leq  \epsilon\left(1-\frac{a^2}{8}\right) ,
\end{split}
\]  where we used $b>\frac{a}{2}$ to obtain the last two inequalities. Finally, using that $\epsilon < \frac{a^2}{64}$, we have
  \[
  p\big|_{\partial B_1}\le \epsilon\left(1-\frac{a^2}{8}\right) \left(1+\frac{a^2}{16}\right) \leq \epsilon\left(1-\frac{1}{16}a^2\right) < \frac{|D|}{2\pi} \left(1-\frac{1}{16}a^2\right),
  \]
  where in the last step we use that $|D| = \pi(1+\epsilon)^2 - \pi > 2\pi \epsilon$. This finishes the proof of the lemma.
\end{proof}


Finally we give the proof of Lemma~\ref{talenti_4}.

\begin{proof}[\textup{\textbf{Proof of Lemma~\ref{talenti_4}}}]
Without loss of generality, we can assume that $o_{2}$ is the origin.  Let $\beta:=p|_{\partial B_r}$. From the proof of Proposition~\ref{talenti2},  we already know that $g'(k)\le -2\pi$, where $g(k):=|\left\{ x\in D : p(x)>k\right\}|$. This implies that $g(k)\ge -2\pi (k-\beta).$ Therefore we have
\begin{align*}
\int_D pdx \geq \int_0^\beta g(k)dk\ge \int_0^\beta -2\pi(k-\beta)dk=\pi \beta^2.
\end{align*}
On the other hand, the same computation in the proof of Lemma~\ref{p_bd} gives
\begin{align*}
\beta |B_r|+\int_D pdx=\frac12 \int_D |\nabla p|^2dx\le \frac12 \int_D |x|^2dx.
\end{align*}
Since 
\begin{align*}
\frac12 \int_D |x|^2dx&=\frac12 \left(\int_{B_R}|x|^2dx-\int_{B_r}|x|^2dx\right)\\
&=\frac{|D|^2}{4\pi}+\frac{|D||B_r|}{2\pi}+\frac{|B_r|^2}{4\pi}-\frac12 \int_{B_r}|x|^2dx\\
&=\frac{|D|^2}{4\pi}+\frac{|D||B_r|}{2\pi}-\frac{l^2|B_r|}{2},
\end{align*}
it follows that
\begin{equation}\label{quadratic_ineq}
\pi \beta^2+\beta |B_r|\le \frac{|D|^2}{4\pi}+\frac{|D||B_r|}{2\pi}-\frac{l^2|B_r|}{2}.
\end{equation}
By solving the quadratic inequality \eqref{quadratic_ineq}, we find that 
\begin{align*}
\beta\le \frac{|D|}{2\pi}(1-c_0),
\end{align*}
for some constant $c_0$ which only depends on $\delta_1$, $\delta_2$ and $l$.
\end{proof}

\section{Radial symmetry for stationary/rotating gSQG solutions with $\Omega\leq 0$}

In this section, we consider the family of gSQG equations with $0 < \alpha < 2$, and study the symmetry property for rotating patch/smooth solutions with angular velocity $\Omega \leq 0$. 

Let us deal with patch solutions first. As we have discussed in the introduction, we cannot expect a non-simply-connected patch $D$ with $\Omega\leq 0$ to be radial, due to the non-radial examples in \cite{delaHoz-Hassainia-Hmidi:doubly-connected-vstates-gsqg, GomezSerrano:stationary-patches} for $\alpha\in(0,2)$. For a simply-connected patch $D$, the constant on the right hand side of \eqref{eq_rotating20} is the same on $\partial D$, which motivates us to consider Question~\ref{question2} in the introduction. The goal of this section is to prove Theorem~\ref{thmC}, which gives an affirmative answer to Question~\ref{question2} for the whole range $\alpha\in[0,2)$.

Our results are not limited to the Riesz potentials $K_{\alpha,d}$  in \eqref{def_k_alpha}; in fact, we only need the potential being radially increasing and not too singular at the origin. Below we state our assumption on the potential $K$, which covers the whole range of $K_{\alpha,d}$ with $\alpha \in [0,2)$.

\medskip
\noindent\textbf{(HK)} 
Let $K \in C^1(\mathbb{R}^d \setminus \{0\})$ be radially symmetric with $K'(r)>0$ for all $r>0$. (Here we denote $K(x) = K(r)$ by a slight abuse of notation.) Also assume there is some $\delta>0$ such that $K'(r)\leq r ^ {-d-1+\delta}$ for all $0< r\leq 1$.
\medskip


Our proof is done by a variational approach, which relies on a continuous Steiner symmetrization argument in a similar spirit as \cite{Carrillo-Hittmeir-Volzone-Yao:nonlinear-aggregation-symmetry-asymptotics}.

\subsection{Definition and properties of continuous Steiner symmetrization}

Below we define the continuous Steiner symmetrization for a bounded open set $D \subset \mathbb{R}^d$ with respect to the direction $e_1 = (1,0,\dots,0)$, which can be easily adapted to any other direction in $\mathbb{R}^d$. The definition is the same as \cite[Section 2.2.1]{Carrillo-Hittmeir-Volzone-Yao:nonlinear-aggregation-symmetry-asymptotics}, which we briefly outline below for completeness.

For a \emph{one-dimensional} open set $U\subset\mathbb{R}$, we define its continuous Steiner symmetrization $M^\tau[U]$ as follows. If $U=(a,b)$ is an open interval, then $M^\tau[U]$ shifts the midpoint of this interval towards the origin with velocity 1, while preserving the length of interval. That is,
\[
M^\tau[U] := \begin{cases}
\big(a - \tau\, \text{sgn}(\frac{a+b}{2}), b-\tau \,\text{sgn}(\frac{a+b}{2})\big) & \text{ for } 0\leq \tau < \frac{|a+b|}{2},\\
(-\frac{b-a}{2}, \frac{b-a}{2}) &  \text{ for } \tau \geq \frac{|a+b|}{2}.
\end{cases}
\]
If $U=\cup_{i=1}^N U_i$ is a finite union of open intervals, then $M^\tau[U]$ is defined by $\cup_{i=1}^N M^\tau[U_i]$, and as soon as two intervals touch each other, we merge them into one interval as in \cite[Definition 2.10(2)]{Carrillo-Hittmeir-Volzone-Yao:nonlinear-aggregation-symmetry-asymptotics}. Finally, if $U=\cup_{i=1}^\infty U_i$ is a countable union of open intervals, we define $M^\tau[U]$ as a limit of $M^\tau[\cup_{i=1}^N U_i]$ as $N\to\infty$  as in \cite[Definition 2.10(3)]{Carrillo-Hittmeir-Volzone-Yao:nonlinear-aggregation-symmetry-asymptotics}. See \cite[Figure 1]{Carrillo-Hittmeir-Volzone-Yao:nonlinear-aggregation-symmetry-asymptotics} for an illustration of $M^\tau[U]$.

Next we move on to higher dimensions. 
We denote a point $x\in\mathbb{R}^d$ by $(x_1, x')$, where $x' = (x_2,\dots,x_d) \in\mathbb{R}^{d-1}$. For a bounded domain $D \subset \mathbb{R}^d$ and any $x'\in\mathbb{R}^{d-1}$, we define the \emph{section} of $D$ with respect to the direction $x_1$ as
\[
D_{x'} := \{x_1 \in \mathbb{R}: (x_1, x')\in D\},
\]
which is an open set in $\mathbb{R}$.  If the section $D_{x'}$ is a single open interval centered at 0 for all $x' \in \mathbb{R}^{d-1}$, then we say the set $D$ is \emph{Steiner symmetric} about the hyperplane $\{x_1=0\}$. Note that this definition is stronger than being symmetric about $\{x_1=0\}$.  For example, an annulus in $\mathbb{R}^2$ is symmetric about $\{x_1=0\}$, but not Steiner symmetric about it.

Finally, for any $\tau>0$, the \emph{continuous Steiner symmetrization} of $D\subset\mathbb{R}^d$ is defined as
\[
S^\tau[D] := \{(x_1, x') \in \mathbb{R}^d: x_1 \in M^\tau[D_{x'}]\},
\]
with $M^\tau$ given above being the continuous Steiner symmetrization for one-dimensional open sets. See Figure~\ref{fig_steiner} for a comparison of the sets $D$ and $S^\tau[D]$  for small $\tau>0$.

\begin{figure}[h!]
\begin{center}
\includegraphics[scale=0.7]{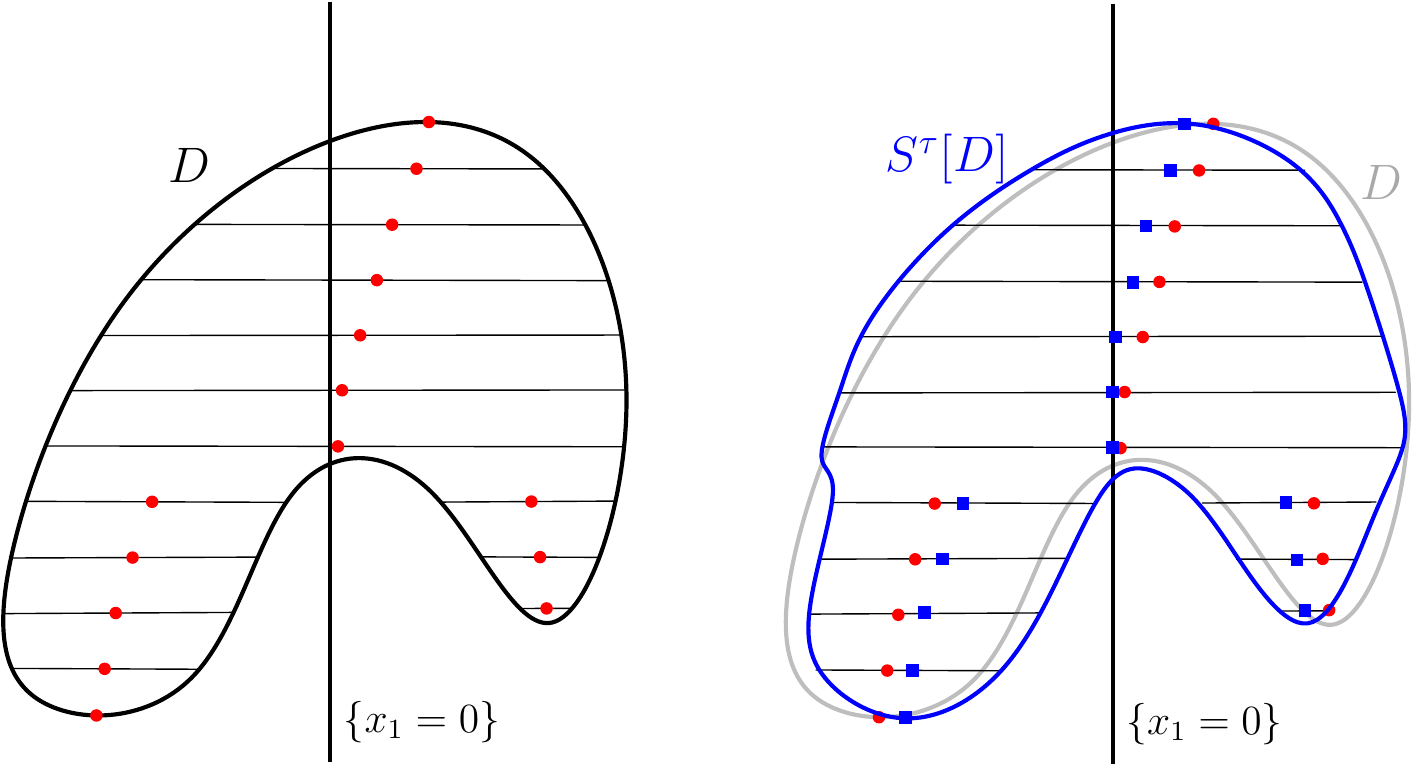}
\caption{Illustration of the continuous Steiner symmetrization $S^\tau[D]$ for a set $D\subset \mathbb{R}^2$. The left figure is the set $D$, with the midpoints of all subintervals of its 1D section highlighted in red circles. The right figure shows the set $S^\tau[D]$ for some small $\tau>0$, with the new midpoints denoted by blue squares. \label{fig_steiner}}
\end{center}
\end{figure}

One can easily check that $S^\tau[D]$ satisfies the following properties.
\begin{lem}\label{lem_steiner}
For any bounded open set $D\subset\mathbb{R}^d$, its continuous Steiner symmetrization $S^\tau[D]$ satisfies the following properties:
\begin{enumerate}
\item[(a)] $|S^\tau[D]| = |D|$ for any $\tau> 0$, where $|\cdot|$ denotes the Lebesgue measure in $\mathbb{R}^d$.
\item[(b)] $(S^\tau[D] ) \triangle D \subset B^\tau[D]$ for any $\tau>0$, where $\triangle$ is the symmetric difference between the two sets, and $B^\tau[D]$ is the $\tau$-neighborhood of $\partial D$, given by
\begin{equation}\label{def_b_tau}
B^\tau[D] := \{x \in \mathbb{R}^d: \textup{dist}(x,\partial D) \leq  \tau\}.
\end{equation}

\end{enumerate}
\end{lem}

\begin{proof}
(a) is a direct consequence of the fact that $|M^\tau [U]| = |U|$ for any open set $U\subset\mathbb{R}$ and $\tau>0$  \cite[Lemma 2.11(b)]{Carrillo-Hittmeir-Volzone-Yao:nonlinear-aggregation-symmetry-asymptotics}. 
To prove (b), one can start with the one-dimensional version: For any bounded open set $U \subset \mathbb{R}$, we have
$
 M^\tau[U] \triangle U \subset \{x\in \mathbb{R}: \text{dist}(x, \partial U)\leq \tau\},
$
which follows from the fact that the intervals move with velocity at most 1. Thus for any bounded open set $D\subset \mathbb{R}^d$,
\[
\begin{split}
S^\tau[D]\triangle D &= \{(x_1, x')\in\mathbb{R}^d: x_1 \in M^\tau[D_{x'}] \triangle D_{x'}\} \\
& \subset \{(x_1,x'): \text{dist}(x_1,\partial (D_{x'})) \leq \tau\}\\
& \subset B^\tau[D],
\end{split}
\]
finishing the proof. 
\end{proof}

\subsection{Simply-connected patch solutions with $\Omega\leq0$}

We assume that $D\subset \mathbb{R}^d$ satisfies the following condition. 

\medskip
\noindent\textbf{(HD)} $D\subset \mathbb{R}^d$ is a bounded domain, and there exists some $M>0$ depending on $D$, such that $|B^\tau[D]|\leq M\tau$ for all sufficiently small $\tau>0$, where $B^\tau[D]$ is given in \eqref{def_b_tau}.

It can be easily checked that for $d\geq 2$, any bounded domain $D$ with Lipschitz continuous boundary satisfies condition \textbf{(HD)}. In fact, for $d=2$, we will show any domain $D\subset\mathbb{R}^2$ with a rectifiable boundary satisfies \textbf{(HD)}, with a precise bound
\begin{equation}\label{ineq_2d}
|B^\tau[D]| \leq 2|\partial D| \tau \quad\text{ for all }\tau\geq 0,
\end{equation}
where $|\partial D|$ is the total length of $\partial D$. Let us first prove \eqref{ineq_2d} holds for any polygon $P \subset \mathbb{R}^2$. Erect two polygons at distance $\tau$ from $P$ and the transversal sides being bisectors of the inner angles of $P$ (see Figure \ref{fig_poly}). It is clear that $B^\tau[P]$ is contained in the trapezoidal region, which has area no more than $2|\partial P| \tau$. Finally, this can be extended to the general case by approximating any rectifiable curve by polygons.


\begin{figure}[h!]
\begin{center}
\includegraphics[scale=1]{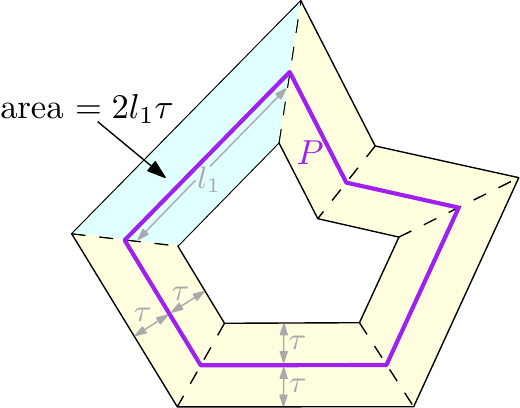}
\caption{Illustration of the polygon $P$ and the underlying trapezoidal region (the whole colored region). Here the blue trapezoid has area $2 l_1 \tau$ ($l_1$ is the corresponding side length in $P$), and summing over all edges gives a total area $2 |\partial P| \tau$. Since the trapezoids may intersect for large $\tau$, the whole trapezoidal region has area no more than $2|\partial P| \tau$. \label{fig_poly}}
\end{center}
\end{figure}


Below we state our main theorem of this section, which is slightly more general than Theorem~\ref{thmC}.

\begin{thm}\label{thm-generic} Let $D\subset\mathbb{R}^d$ and $K\in C^1(\mathbb{R}^d\setminus\{0\})$ satisfy the conditions \textup{\textbf{(HD)} and \textup{\textbf{(HK)}}} respectively.  Let $g\in C^1(\mathbb{R}^d)$ be a radial function with $g'(r)>0$ for all $r> 0$.

If $D$ satisfies that
\begin{equation}\label{stat_eq}
1_D * K- \frac{\Omega}{2} g(x) = \text{const} \quad\text{ on }\partial D
\end{equation}
 for some $\Omega \leq 0$ (where the constant is the same on all connected components of $\partial D$), then $D$ is
a ball. Moreover, the ball is centered at the origin if $\Omega<0$.
\end{thm}

\begin{rmk}\label{rmk_origin}

(1) Note that $D$ does not need to be simply-connected in Theorem~\ref{thm-generic}. However, since the constant on the right hand side of \eqref{stat_eq} is assumed to be the same on all connected components of $\partial D$, comparing with \eqref{eq_rotating20}, Theorem~\ref{thm-generic} only implies that all simply-connected patches with $\Omega\leq 0$ must be a disk.

(2) In the case $\Omega = 0$, the problem is translation invariant, so in the proof we assume without loss of generality that the center of mass of $D$ is at the origin.
\end{rmk}

\begin{proof}
We prove it by contradiction. Without loss of generality, we assume $D$ is not Steiner symmetric about the hyperplane $\{x_1=0\}$. Let $D^\tau := S^\tau[D]$ be the continuous Steiner symmetrization of $D$ at time $\tau>0$. By Lemma~\ref{lem_steiner}(b), we have 
\begin{equation}\label{sym_diff}
D^\tau \triangle D \subset B^\tau[D],
\end{equation} where $B^\tau$ is defined in \eqref{def_b_tau}.
Let us consider the functional
\[
\mathcal{E}[D] := \underbrace{\int_{\mathbb{R}^d} \int_{\mathbb{R}^d} 1_D(x) 1_D(y) K(x-y)dxdy}_{=: \mathcal{I}[D]} +\underbrace{(- \Omega)\int_{\mathbb{R}^d} g(x) 1_D(x) dx}_{=: \mathcal{V}[D]}.
\]
We will use two different ways to compute $\frac{d^+}{d\tau} \mathcal{E}[D^\tau] \Big|_{\tau = 0}$, where $\frac{d^+}{d\tau}$ denotes the right derivative. On the one hand, using the equation \eqref{stat_eq} and the regularity assumptions on $D, K$ and $g$, we aim to show that 
\begin{equation}\label{diff1}\frac{d^+}{d\tau} \mathcal{E}[D^\tau] \Big|_{\tau = 0} = 0.
\end{equation} Instead of directly taking the derivative, we consider the finite difference
\[
\mathcal{E}[D^\tau] - \mathcal{E}[D] = \underbrace{\int_{\mathbb{R}^d}2 (1_{D^\tau}-1_{D}) \left(1_D * K - \frac{\Omega}{2} g(x) \right) dx }_{=:I_1}+ \underbrace{\int_{\mathbb{R}^d} ( 1_{D^\tau}-1_D) ( (1_{D^\tau}-1_D) * K) dx}_{=: I_2},
\]
where in the equality we used that $\int 1_D(1_{D^\tau}*K) dx = \int 1_{D^\tau}(1_D*K)dx$ for any radial kernel $K$.

Let us control the term $I_1$ first. First note that \eqref{sym_diff} implies that the integrand is supported in $B_\tau[D]$. 
  Next we claim that \textbf{(HK)} implies $1_D * K - \frac{\Omega}{2}g \in C^{0,\delta'}(\mathbb{R}^d)$ 
  for  $\delta':= \min\{\delta,1\}$, where $C^{0,1}$ stands for Lipschitz continuity. 
   The proof is a simple potential theory estimate, which we provide below for completeness. For any $x,z \in \mathbb{R}^d$, 
\[
\begin{split}
|(1_D * K)(x+z) &- (1_D * K)(x)| = \left|\int_{\mathbb{R}^d} 1_D(x-y) (K(y+z)-K(y)) dy\right|\\
&\leq \int_{|y|<2|z|} |K(y+z)-K(y)| dy + \int_{|y|>2|z|} 1_D(x-y) |K(y+z)-K(y)| dy\\
&=: J_1 + J_2. 
\end{split}
\]
For $J_1$, a crude estimate gives
\[
J_1 \leq  \int_{|y|<2|z|} |K(y+z)|+|K(y)| dy \leq 2  \int_{|y|<3|z|} |K(y)| dy \leq C(d)|z|^\delta,
\]
where in the last step we used that \textbf{(HK)} implies $|K(y)| \leq C |y|^{-d+\delta}$ for $|y|\leq 1$. For $J_2$, note that \textbf{(HK)} and the mean-value theorem gives 
\[
|K(y+z)-K(y)| \leq C|y|^{-d-1+\delta}|z| \quad\text{ for all }|y|>2|z|,\]
and plugging it into the integral gives $J_2 \leq C(d,|D|) |z|^{\delta}$. Putting the estimates for $J_1$ and $J_2$ together gives that
 $1_D * K \in C^{0,\delta'}(\mathbb{R}^d)$ for $\delta' = \min\{\delta,1\}$, and
combining this with the assumption $g\in C^1(\mathbb{R}^d)$ gives $1_D * K- \frac{\Omega}{2}g \in C^{0,\delta'}(\mathbb{R}^d)$.

In addition, by \eqref{stat_eq}, 
We have $1_D * K - \frac{\Omega}{2} g(x) \equiv C_0$ on $\partial D$ for some constant $C_0$. Thus we have
\[
\Big|1_D * K - \frac{\Omega}{2} g(x)- C_0\Big| \leq C(\delta', d,|D|)\tau^{\delta'} \quad\text{ in }B^\tau[D]
\]
for some constant $C>0$, where we used the H\"older continuity of $1_D * K - \frac{\Omega}{2} g$ and the definition of $B^\tau[D]$. This leads to
\[
|I_1| \leq 2|B^\tau[D]|~ \sup_{x\in B^\tau[D]}\Big|1_D * K - \frac{\Omega}{2} g(x)- C_0\Big| \leq C(\delta', d,|D|) M \tau^{1+\delta'},
\]
where in the first inequality we used that $\int_{B_\tau} (1_D - 1_{D^\tau}) C_0 dx = 0$, which follows from Lemma~\ref{lem_steiner}(a); and in the second inequality we used \textbf{(HD)}.

Next we control $I_2$ by the crude bound
\[
\begin{split}
|I_2| \leq \int_{\mathbb{R}^d} 1_{B^\tau[D]} |(1_{B^\tau[D]} * K)|dx \leq |B^\tau[D]|\, \|1_{B^\tau[D]} * K\|_\infty \leq M\tau \|(1_{ B^\tau[D]})^* * K\|_\infty,
\end{split}
\]
where the last step follows from the Hardy--Littlewood inequality, where $(1_{ B^\tau[D]})^*$ is the radial decreasing rearrangement of $1_{B^\tau[D]}$. By \textbf{(HD)}, $(1_{ B^\tau[D]})^*$ is a characteristic function of a ball whose radius is bounded by $C(d) (M\tau)^{1/d}$, thus
\[
\|(1_{ B^\tau[D]})^* * K\|_\infty \leq \int_0^{C(d)  (M\tau)^{1/d}} |K(r)| \, \omega_d r^{d-1}dr \leq  \int_0^{C(d)  (M\tau)^{1/d}} \omega_d r^{-1+\delta}dx \leq C(d) (M\tau)^{\frac{\delta}{d}},
\]
and plugging it into the $I_2$ estimate gives
\[
|I_2| \leq C(d) M^{\frac{d+\delta}{d}} \tau^{1+\frac{\delta}{d}}.
\]

Putting the estimates of $I_1$ and $I_2$ together directly yields
\[
\frac{|\mathcal{E}[D^\tau] - \mathcal{E}[D]|}{\tau} \leq C(\delta',d,M,|D|) \tau^{\min\{\frac{\delta}{d}, \delta'\}},
\]
and since $\delta>0$ we have $\frac{d^+}{d\tau} \mathcal{E}[D^\tau] \Big|_{\tau = 0} = 0$.

Now, we use another way to calculate $\frac{d^+}{d\tau} \mathcal{E}[D^\tau]\big|_{\tau=0} $. Let us deal with the $\Omega<0$ case first. Since $K$ is radial and increasing in $r$, it has been shown in \cite[Corollary 2]{Brock:continuous-steiner-symmetrization} and \cite[Theorem 3.7]{Lieb-Loss:analysis} that  the interaction energy $\mathcal{I}[D^\tau]=\int_{D^\tau}\int_{D^\tau} K(x-y) dxdy$ is non-increasing along the continuous Steiner symmetrization, leading to

\begin{equation}\nonumber
\frac{d^+}{d\tau}\mathcal{I}[D^\tau]\leq  0 \quad\text{ for all }\tau\geq 0.
\end{equation}
For the other term $\mathcal{V}[D^\tau]=(-\Omega)\int_{D^\tau} g(x) dx$, by the assumptions that $g'(r)>0$ for all $r>0$ and  $D$ is not Steiner symmetric about $\{x_1=0\}$, 
we can use \cite[Lemma 2.22]{Carrillo-Hittmeir-Volzone-Yao:nonlinear-aggregation-symmetry-asymptotics} to show, for $\Omega<0$,

\begin{equation}\nonumber
\frac{d^+}{d\tau} \mathcal{V}[D^\tau]\bigg|_{\tau=0} = (-\Omega)\frac{d^+}{d\tau}\int_{D^{\tau}} g(x) dx\bigg|_{\tau=0}<0.
\end{equation}
Adding them together gives
\[
\frac{d^+}{d\tau}\mathcal{E}[D^\tau]\bigg|_{\tau=0} < 0
\]leading to a contradiction with \eqref{diff1}.

In the $\Omega =0$ case, recall that we assume that the center of mass of $D$ is at the origin. Thus if $D$ is not Steiner symmetric about $\{x_1=0\}$, the same proof as \cite[Proposition 2.15]{Carrillo-Hittmeir-Volzone-Yao:nonlinear-aggregation-symmetry-asymptotics} gives that $\mathcal{I}[D]$ must be decreasing to the first order for a short time, leading to

\begin{equation}\nonumber
\frac{d^+}{d\tau}\mathcal{E}[D^\tau]\bigg|_{\tau=0}  = \frac{d^+}{d\tau}\int_{D^{\tau}}\int_{D^{\tau}} K(|x-y|) dx dy\bigg|_{\tau=0}< 0,
\end{equation}
again contradicting \eqref{diff1}.
We point out that although the proposition was stated for continuous densities, the same proof works for the patch setting. In addition, although \cite{Carrillo-Hittmeir-Volzone-Yao:nonlinear-aggregation-symmetry-asymptotics} only dealt with the kernels no more singular than Newtonian potential, the proof indeed holds for all kernels $K$ satisfying \textbf{(HK)}: see \cite[Theorem 6]{Carrillo-Hoffmann-Mainini-Volzone:ground-states-diffusion-regime} for an extension to all Riesz potentials $K_{\alpha,d}$ with $\alpha \in (0,2)$.
\end{proof}

The above theorem immediately leads to the following result concerning simply-connected stationary/rotating patch solution with $\Omega\leq 0$.

\begin{thm}\label{thm-euler}Let $D \subset \mathbb{R}^2$ be a bounded, simply-connected domain with rectifiable boundary. 
If $1_{D}$ is a V-state for \eqref{gSQG} for some $\alpha\in[0,2)$ with angular velocity $\Omega\leq 0$, then $D$ must be a disk.
In addition, the disk must be centered at the origin if $\Omega<0$. 
\end{thm}
\begin{proof}
We have $1_D * K - \frac{\Omega}{2} |x|^2 = C$ for some constant $C$ on $\partial D$.
For the Euler equation, $K=\frac{1}{2\pi}\ln |x|$. For the g-SQG equation. $K=-C_\alpha|x|^{-\alpha}$. In both cases, the proof follows from Theorem \ref{thm-generic}.
\end{proof}

\begin{rmk}
As we discussed in the beginning of this subsection, in the case of gSQG with $\alpha \in (0,2)$, Theorem \ref{thm-euler} is not true if the simply connected assumption is dropped, due to the non-radial patches in \cite{delaHoz-Hassainia-Hmidi:doubly-connected-vstates-gsqg, GomezSerrano:stationary-patches} bifurcating from annuli. 
\end{rmk}

\subsection{Smooth solutions with simply-connected level sets with $\Omega\leq0$}

The rest of this section is devoted to the smooth setting. We will show that any nonnegative smooth rotating solution  of the Euler or gSQG equation with angular velocity $\Omega \leq 0$ must be radial, under the additional assumption that all the super level-sets $U^h$
\begin{equation}\label{def_Uh}
U^h:=\{x\in \mathbb{R}^d: \omega(x) > h\}
\end{equation}
are simply-connected for any $h>0$. We believe that the simply-connected assumption is necessary, since it is likely that the bifurcation arguments from annuli in \cite{delaHoz-Hassainia-Hmidi:doubly-connected-vstates-gsqg, GomezSerrano:stationary-patches} can be extended to the smooth setting as well, using a similar argument as in \cite{Castro-Cordoba-GomezSerrano:uniformly-rotating-smooth-euler} or \cite{Castro-Cordoba-GomezSerrano:global-smooth-solutions-sqg}.

\begin{thm}\label{thm-smooth}
Let $\omega (x) \in C^1(\mathbb{R}^2)$ be nonnegative and compactly supported. In addition, assume the super level-set $U^h$ as in \eqref{def_Uh} is simply connected for all $h\in(0,\sup\omega)$.
Assume $K$ satisfies \textup{\textbf{(HK)}}.
If for some $\Omega\leq 0$, we have
\begin{equation}\label{thm2_eq}
\omega*K-\frac{\Omega}{2} |x|^2=C_0(h) \quad\text{ on } \partial U^h~ \text{ for all }  h \in(0, \sup\omega),
\end{equation}then $\omega$ is radially decreasing up to a translation. Moreover, it is centered at the origin if $\Omega < 0$.
\end{thm}

\begin{proof}
We prove it by contradiction. For the $\Omega < 0$ case, without loss of generality, we assume $\omega$ is not symmetric decreasing about the line $x_1=0$. For the $\Omega=0$ case, similar to Remark~\ref{rmk_origin}, without loss of generality we assume the center of mass is at the origin, and then we assume $\omega$ is not symmetric decreasing about the line $x_1=0$.

For any $\tau\geq 0$, we define the continuous Steiner symmetrization $\omega^\tau(x)$ in the same way as \cite[Definition 2.12]{Carrillo-Hittmeir-Volzone-Yao:nonlinear-aggregation-symmetry-asymptotics}: 
\[
\omega^\tau(x):=\int_0^{h_0}1_{S^\tau[U^h]}(x)\,dh,
\]where $h_0:=\sup \omega$, and $S^\tau [U^h]$ is the continuous Steiner symmetrization of the super level set $U^h$ at time $\tau\geq 0$.

Consider the energy functional 
\[ 
\mathcal{E}[\omega]:=\underbrace{\int_{\mathbb{R}^2}\int_{\mathbb{R}^2}\omega(x)\omega(y)K(x-y)dxdy}_{=:\mathcal{I}[\omega]}+\underbrace{(-\Omega)\int_{\mathbb{R}^2}\omega(x)|x|^2dx}_{=:\mathcal{V}[\omega]}.
\]
We proceed similarly as in Theorem~\ref{thm-generic} to compute $\frac{d^+}{d\tau} \mathcal{E}[\omega^\tau]$ in two different ways. We first rewrite the finite difference $\mathcal{E}[\omega^\tau] - \mathcal{E}[\omega]$ as
\begin{eqnarray}\nonumber
&&\mathcal{E}[\omega^\tau] - \mathcal{E}[\omega] \\\nonumber
&=& \int_{\mathbb{R}^2}2 (\omega^\tau(x)-{\omega(x)}) \left(\omega * K - \frac{\Omega}{2} |x|^2 \right) dx +\iint_{\mathbb{R}^2\times \mathbb{R}^2} (\omega^\tau(x) - \omega(x)) ( \omega^\tau(y)-\omega(y) ) K(x-y) dxdy\\
&=:& I_1 + I_2.
\end{eqnarray}
Since $\omega \in C_c^1(\mathbb{R}^2)$ and $K$ satisfies \textbf{(HK)} (hence is locally integrable), one can easily check that  
$
\omega*K-\frac{\Omega}{2}|x|^2$ is Lipschitz in 
$\tilde D:= \{x\in\mathbb{R}^2: \text{dist}(x,\text{supp\,} \omega)\leq 1\}
.$
 Note that we have $\text{supp~} \omega^\tau \in \tilde D$ for all $\tau\in[0, 1]$. Combining this fact with the assumption \eqref{thm2_eq},  there exists $C_1>0$ independent of $h$, such that
\begin{equation}\label{temp432}
\Big|(\omega*K)(x)-\frac{\Omega}{2}|x|^2-C_0(h)\Big|\leq C_1 \tau \quad \text{on }S^\tau[U^{h}]\triangle U^h ~~\text{ for all }h\in (0,h_0).
\end{equation}
Let us first rewrite $I_1$ as
\[
I_1=2\int_0^{h_0}\int_{\mathbb{R}^2}\Big(1_{S^\tau[U^h]}(x)-1_{U^h}(x)\Big)\Big((\omega*K)(x)-\frac{\Omega}{2}|x|^2\Big)\,dx\,dh.
\]
By Lemma~\ref{lem_steiner}(a), we have $\int_{\mathbb{R}^2}  (1_{S^\tau[U^h]}(x)-1_{U^h}(x)) dx =0$ for all $h\in(0,h_0)$. Thus we can control $I_1$ as
\begin{equation}
\begin{split}
|I_1|&=\bigg|2\int_0^{h_0}\int_{\mathbb{R}^2}\Big(1_{S^\tau[U^h]}(x)-1_{U^h}(x)\Big)\Big((\omega*K)(x)-\frac{\Omega}{2}|x|^2-C_0(h)\Big)dxdh\bigg|\\
&\leq 2C_1\tau \int_0^{h_0} \Big|(S^\tau[U^h])\triangle U^h\Big| dh\\
&\leq 2C_1\tau \int_0^{h_0} 2|\partial U^h|\tau \,dh\\
&= 4C_1 \tau^2 \int_{\text{supp~}\omega}|\nabla \omega| dx \leq C(\omega) \tau^2.
\end{split}
\end{equation}
Here in the second line we used \eqref{temp432}; in the third line we used  Lemma~\ref{lem_steiner}(b) and the property \eqref{ineq_2d} in two dimensions; and in the fourth line we used the co-area formula and the fact that $\omega\in C_c^1$.



We next move on to $I_2$. Since $|\nabla\omega|$ is bounded, Lemma~\ref{lem_steiner}(b) leads to 
\[
\begin{split}
|\omega^\tau (x)-\omega(x)| &=\left| \int_0^\infty 1_{S^\tau[U^h]}(x)-1_{U^h}(x) dh \right|\\
& \leq \|\nabla \omega\|_{L^\infty}\tau \quad\text{ for all }x\in \mathbb{R}^2,
\end{split}
\]
and $\text{supp}~\omega^\tau \in \tilde D$ for all $\tau\in [0,1]$.
Thus
\[
\begin{split}
|I_2|&\leq \|\omega^\tau - \omega\|_{L^1} \|(\omega^\tau-\omega) * K\|_{L^\infty}\\
&\leq \|\omega^\tau - \omega\|_{L^1} \|\omega^\tau - \omega\|_{L^\infty} \int_{\tilde D} |K(x)| dx\\
&\leq C(\omega)\tau^2.
\end{split}
\]

Combining the estimates for $I_1$ and $I_2$ gives $\mathcal{E}[\omega^\tau]-\mathcal{E}[\omega]\leq C(\omega)\tau^2$ for all $\tau\in[0,1]$, thus
\begin{equation}\label{thm2_ineq2}
\frac{d^+}{d\tau}\mathcal{E}[\omega^\tau]\bigg|_{\tau=0} = \frac{d^+}{d\tau}(I_1+I_2)\bigg|_{\tau=0}=0.
\end{equation}
On the other hand, we compute $\frac{d^+}{d\tau}\mathcal{E}[\omega^\tau]\big|_{\tau=0}$ in a different way as
\begin{eqnarray}\nonumber
&&\frac{d^+}{d\tau}\mathcal{E}[\omega^\tau]\bigg|_{\tau=0}=\frac{d^+}{d\tau}(\mathcal{I}[\omega^\tau] + \mathcal{V}[\omega^\tau])\bigg|_{\tau=0}.
\end{eqnarray}
In the $\Omega<0$ case, similarly as in Theorem~\ref{thm-generic}, we have $\mathcal{I}[\omega^\tau]$ is non-increasing along the continuous Steiner symmetrization by  \cite[Corollary 2]{Brock:continuous-steiner-symmetrization} and \cite[Theorem 3.7]{Lieb-Loss:analysis}, thus
\begin{equation}\nonumber
\frac{d^+}{d\tau}\mathcal{I}[\omega^\tau]\leq  0 \quad\text{ for all }\tau>0.
\end{equation}
For  $\mathcal{V}[\omega^\tau]$, by the assumption that $\omega$ is not symmetric decreasing about $\{x_1=0\}$, 
we again use \cite[Lemma 2.22]{Carrillo-Hittmeir-Volzone-Yao:nonlinear-aggregation-symmetry-asymptotics} to show, for $\Omega<0$,
\begin{equation}\nonumber
\frac{d^+}{d\tau}\mathcal{V}[\omega^\tau] = (-\Omega)\frac{d^+}{d\tau}\int_{\mathbb{R}^2}  \omega(x) |x|^2 dx\bigg|_{\tau=0}<0.
\end{equation}
Adding them together gives $
\frac{d^+}{d\tau}\mathcal{E}[D^\tau]\big|_{\tau=0} < 0$, contradicting \eqref{thm2_ineq2}.

In the $\Omega =0$ case, we assume that the center of mass of $\omega$ is at the origin. Thus if $\omega$ is not symmetric  decreasing about $\{x_1=0\}$, the same proof as \cite[Proposition 2.15]{Carrillo-Hittmeir-Volzone-Yao:nonlinear-aggregation-symmetry-asymptotics} gives that $\mathcal{I}[D]$ must be decreasing to the first order for a short time (again, the proof holds for all kernels $K$ satisfying \textbf{(HK)}; see \cite[Theorem 6]{Carrillo-Hoffmann-Mainini-Volzone:ground-states-diffusion-regime} for extensions to Riesz kernels $K_{\alpha,d}$ with $\alpha\in(0,2)$). This gives $
\frac{d^+}{d\tau}\mathcal{E}[D^\tau]\big|_{\tau=0} < 0,$ again contradicting \eqref{thm2_ineq2}.
\end{proof}

The above theorem immediately gives the following corollary concerning the V-states for the Euler and gSQG equations.

\begin{cor}\label{thm-euler-smooth}
Assume  $\omega (x) \in C^1(\mathbb{R}^2)$ is a nonnegative, compactly supported V-state satisfying the Euler equation or the gSQG equation for some $\alpha\in (0,2)$ with $\Omega\leq 0$. In addition, assume the super level-set $U^h$ as in \eqref{def_Uh} is simply connected for all $h\in(0,\sup\omega)$. Then $\omega$ must be radially decreasing if $\Omega<0$, and radially decreasing up to a translation if $\Omega=0$.
\end{cor}
\begin{proof}
For the Euler equation, $K=\frac{1}{2\pi}\ln|x|$. For the gSQG equation. $K=-C_\alpha|x|^{-\alpha}$. In both cases, the proof follows from Theorem \ref{thm-smooth}.
\end{proof}

\section{Radial symmetry of rotating gSQG solutions with $\Omega>\Omega_\alpha$}

In this section, we focus on rotating gSQG patches with area $\pi$ and $\alpha \neq 0$. As we discussed in the introduction, for $\alpha\in[0,2)$, there exist rotating patches bifurcating from the unit disk at angular velocities  $\Omega_{m}^{\al} = 2^{\al-1}\frac{\Gamma(1-\al)}{\Gamma\left(1-\frac{\al}{2}\right)^2}\left(\frac{\Gamma\left(1+\frac{\al}{2}\right)}{\Gamma\left(2-\frac{\al}{2}\right)} - \frac{\Gamma\left(m+\frac{\al}{2}\right)}{\Gamma\left(m+1-\frac{\al}{2}\right)}\right)$, where $\Omega_m^\alpha$ is increasing in $m$ for any fixed $\alpha \in [0,2)$. Let us denote $\Omega_\alpha := \lim_{m\to\infty}\Omega^{\alpha}_{m}$. If $\alpha \in (0,1)$ we have that 
\begin{equation}\label{def_omega_alpha}
\Omega_\alpha = 2^{\al-1}\frac{\Gamma(1-\al)}{\Gamma\left(1-\frac{\al}{2}\right)^2}\frac{\Gamma\left(1+\frac{\al}{2}\right)}{\Gamma\left(2-\frac{\al}{2}\right)}.
\end{equation}
Note that $\Omega_\alpha$ is a continuous function of $\alpha$ for $\alpha \in (0,1)$, with $\Omega_0 = \frac{1}{2}$, and $\Omega_\alpha = +\infty$ for all $\alpha \in [1,2)$.

A natural question is whether there can be rotating patches with area $\pi$ with $\Omega \geq \Omega_\alpha$ for $\alpha \in (0,1)$. Note that the area constraint is necessary for all $\alpha>0$: if $D$ is a rotating gSQG patch for $\alpha\in(0,2)$ with angular velocity $\Omega$, then one can easily check that its scaling $\lambda D = \{\lambda x: x\in D\}$ is a rotating patch with angular velocity $\lambda^{-\alpha}\Omega$.

In Theorem~\ref{clockwise rotating_patch}, we showed that for the 2D Euler case ($\alpha=0$), all rotating patches with $\Omega \geq \Omega_0 = \frac{1}{2}$ must be a disk. In this section, our goal is to show that all \emph{simply-connected} rotating patches with area $\pi$ with $\Omega \geq \Omega_\alpha$ for $\alpha \in (0,1)$ must be a disk. Whether there exist non-simply-connected or disconnected rotating patches with $\Omega \geq \Omega_\alpha$ for $\alpha \in (0,1)$ is still an open question.

Below is the main theorem of this section. Recall that for $\alpha \in (0,2)$, $K_\alpha = -C_\alpha |x|^{-\alpha}$  is the fundamental solution for $-(-\Delta)^{-1+\frac{\alpha}{2}}$, where $C_\alpha = \frac{1}{2\pi} \frac{\Gamma(\frac{\alpha}{2})}{2^{1-\alpha}\Gamma(1-\frac{\alpha}{2})}$.

\medskip
\begin{thm}\label{thm_large_omega}
Let $D\subset \mathbb{R}^2$ be a bounded, simply-connected patch with $C^1$ boundary. Let us denote $R := \max_{x \in D}|x|$. 
Assume that $D$ is a uniformly rotating patch with angular velocity $\Omega$ of the gSQG equation with $\alpha \in (0,1)$, i.e.,
\begin{equation}\label{eq_const}
1_D * K_\alpha - \frac{\Omega}{2} |x|^2 = C \quad\text{ on } \partial D.
\end{equation} 
Let $\Omega_c(R):=R^{-\al}\Omega_\alpha.$
If $\Omega \geq \Omega_c(R)$, then $D$ must coincide with $B(0,R)$.
\end{thm}

\medskip
\begin{rmk} (a)
Note that all sets $D\subset\mathbb{R}^2$ with area $\pi$ must have $R\geq 1$. In this case we have $\Omega_c(R)\leq \Omega_\alpha$, thus Theorem~\ref{thm_large_omega} immediately implies that all simply-connected rotating patches with area $\pi$ and $\Omega \geq \Omega_\alpha$ must be a disk. 

(b) Note that the constant $\Omega_\alpha$ is sharp, since there exist patches bifurcating from a disk of radius 1 at velocities $\Omega_m^\alpha$, which can get arbitrarily close to $\Omega_\alpha$ as $m\to\infty$ \cite[Theorem 1.4]{Hassainia-Hmidi:v-states-generalized-sqg}.
\end{rmk}

\begin{proof}Towards a contradiction, assume that $D \neq B(0,R)$.
Let $x_0 \in \partial D$ be the farthest point from 0. Then we have that $D \subset B(0,R)$, and let $U:= B(0,R)\setminus D$. See Figure~\ref{fig_alpha} for an illustration of $U$ and $x_0$. Then \eqref{eq_const} can be rewritten as
\begin{equation}\label{eq_const2}
1_U *K_\alpha =  1_{B(0,R)} * K_\alpha - \frac{\Omega}{2}|x|^2 - C \quad\text{ on } \partial D.
\end{equation}

\begin{figure}[h!]
\begin{center}
\includegraphics[scale=1.1]{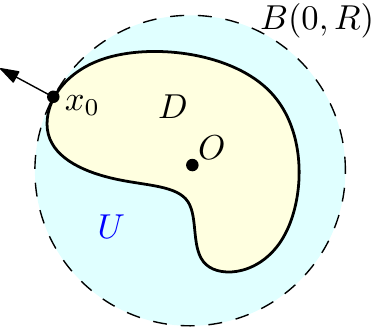}
\caption{Illustration of the set $U$ and the point $x_0$.\label{fig_alpha}}
\end{center}
\end{figure}

The key idea of this proof is to use two different ways to compute $\nabla (1_U * K_\alpha)(x_0)  \cdot x_0$, and obtain a contradiction if $\Omega \geq \Omega_c(R)$. On the one hand, 
\begin{equation}\label{eq_temp1}
\nabla (1_U * K_\alpha)(x_0)  \cdot x_0 = \alpha C_\alpha \int_U \frac{(x_0-y)\cdot x_0}{|x_0-y|^{\alpha+2}}dy > 0,
\end{equation}
where we used the fact that $(x_0-y)\cdot x_0 > 0$ for all $y\in U\subset B(0,R)$ since the two vectors point to the same halfplane.

On the other hand, we claim the following properties hold for $1_U *K_\alpha$:

1. $\Delta (1_U * K_\alpha) < 0$ in $D$.

2. Along $\partial D$, the minimum of $1_U * K_{\alpha}$ is achieved at $x_0$.

To show property 1, using the fact that $K_\alpha = -C_\alpha |x|^{-\alpha}$  is the fundamental solution for $-(-\Delta)^{-1+\frac{\alpha}{2}}$, we have $1_U * K_\alpha = -(-\Delta)^{-1+\frac{\alpha}{2}} 1_U$, thus $\Delta (1_U *K_{\alpha})  = (-\Delta)^{\alpha/2} 1_U$. Thus for any $x\in D$, using the singular integral definition of the fractional Laplacian \cite[Theorem 1.1, Definition (e)]{Kwasnicki:definitions-fractional-laplacian}
and the fact that $1_U \equiv 0$ in $D$, we have
\[
(-\Delta)^{\alpha/2} 1_U(x) = C_1(\alpha) \int_{\mathbb{R}^2}\frac{1_U(x)-1_U(y)}{|x-y|^{2+\alpha}} dy = C_1(\alpha) \int_{\mathbb{R}^2}\frac{0-1_U(y)}{|x-y|^{2+\alpha}} dy < 0 \quad\text{ for }x\in D
\]
for some constant $C_1(\alpha)>0$. Note that despite the denominator being singular, the integral indeed converges for all $x\in D$, due to the fact that $D$ is open and the integrand is identically zero in $D$ which yields
\[
\Delta (1_U * K_\alpha)(x) =  (-\Delta)^{\alpha/2} 1_U(x) < 0 \text{ in } D.
\]

 We now move on to property 2. Due to \eqref{eq_const2} and the fact that $x_0$ is the outmost point on $\partial D$, it suffices to show that the radial function $1_{B(0,R)} *K_\alpha - \frac{\Omega}{2}|x|^2$ is non-increasing in $|x|$ for all $\Omega\geq \Omega_c(R)$. We prove this in Proposition~\ref{prop_radial} right after this theorem.

The above claims allow us to apply the maximum principle to $1_U *K_\alpha$, which yields that the minimum of $1_U * K_\alpha$ in $\overline D$ is also achieved at $x_0$, thus 
\[
\nabla (1_U * K_\alpha)(x_0) \cdot \vec n(x_0) \leq 0,
\] where $\vec n(x_0)$ is the outer normal of $D$ at $x_0$. Since $\vec n(x_0) = x_0/|x_0|$, the above inequality contradicts with \eqref{eq_temp1}. As a result, $D$ must coincide with $B(0,R)$.
\end{proof}

Now we prove the proposition that was used in the proof of the above theorem. 
\begin{prop}\label{prop_radial}
For a fixed $\alpha \in (0,1)$ and $R>0$, let $\Omega_c(R)$ be the smallest number such that 
\[g_R(x) := 1_{B(0,R)}*K_\alpha - \frac{\Omega_c}{2}|x|^2
\] is non-increasing in $|x|$. Then we have $\Omega_c(R) = R^{-\alpha} \Omega_\alpha$, with $\Omega_\alpha$ given in \eqref{def_omega_alpha}.
 \end{prop}

\begin{proof}
Recall that $K_\alpha = -C_\alpha |x|^{-\alpha}$ with  $C_\alpha=\frac{1}{2\pi}\frac{\Gamma(\frac{\alpha}{2})}{2^{1-\alpha}\Gamma(1-\frac{\alpha}{2})}$. Since $|x|^2$ and $1_{B(0,R)} * K_\alpha$ are both radially symmetric and increasing in $|x|$, we have 
\begin{align*}
\Omega_{c}(R)=2C_\alpha \sup_{|x_1|\neq |x_2|}\frac{\displaystyle  \int_{B(0,R)}|x_2-y|^{-\alpha}dy-\int_{B(0,R)}|x_1-y|^{-\alpha}dy}{|x_1|^2-|x_2|^2}.
\end{align*}
Let us denote the fraction above by $F(x_1, x_2)$. 
 We claim that the $\sup_{|x_1|\neq |x_2|} F(x_1, x_2)$ is attained when $|x_1|=R$, and $|x_2| \to R$.

%
%
%
%
To prove the claim, we first compute $I(x) := \int_{B(0,R)}|x-y|^{-\alpha}dy$. Taking the Fourier transform:
\begin{align*}
 I(x) = C R^{2-\al} \int_{0}^{\infty} r^{a-2} J_1(r) J_0\left(\frac{|x| r}{R}\right) dr, 
\end{align*}
where $C$ is some positive constant. By Sonine-Schafheitlin's formula \cite[p. 401]{Watson:treatise-bessel-functions} and by continuity, we obtain
\begin{align*}
 I(x) = \left\{
\begin{array}{cc}
C R^{2-\al} 2^{\alpha-2}\frac{\Gamma\left(\frac{\al}{2}\right)}{\Gamma\left(2-\frac{\al}{2}\right)} \,_2 F_1\left(\frac{\alpha}{2}-1,\frac{\al}{2},1,\frac{|x|^2}{R^2}\right) & \text{ if } |x| \leq R \\
C R^{2-\al}2^{\alpha-2}|x|^{-\alpha}R^{\al}\frac{\Gamma\left(\frac{\al}{2}\right)}{\Gamma\left(1-\frac{\al}{2}\right)} \,_2 F_1\left(\frac{\al}{2},\frac{\al}{2},2,\frac{R^2}{|x|^2}\right) & \text{ if } |x| > R. \\
\end{array}
\right.
\end{align*}

By the mean value theorem, it is enough to check that $\min J(z) = J(R^2)$, where
\begin{align*}
 J(z) & = \left\{
\begin{array}{cc}
 \frac{d}{dz}\left(\,_2 F_1\left(\frac{\alpha}{2}-1,\frac{\al}{2},1,\frac{z}{R^2}\right)\right) & \text{ if } z \leq R^2 \\
\frac{d}{dz}\left(\left(1-\frac{\al}{2}\right)z^{-\frac{\alpha}{2}}R^{\al} \,_2 F_1\left(\frac{\al}{2},\frac{\al}{2},2,\frac{R^2}{z}\right)\right) & \text{ if } z > R^2 \\
\end{array}
\right. \\
& = \left\{
\begin{array}{cc}
 \frac{\al(\al-2)}{4} \frac{1}{R^2}\,_2 F_1\left(\frac{\al}{2}, 1+\frac{\al}{2}, 2, \frac{z}{R^2}\right) & \text{ if } z \leq R^2 \\
 \frac{\al(\al-2)}{4} z^{-1-\frac{\al}{2}} R^{\al} \,_2 F_1\left(\frac{\al}{2}, 1+\frac{\al}{2}, 2, \frac{R^2}{z}\right) & \text{ if } z > R^2 \\
\end{array}
\right. 
\end{align*}

Writing the series expansion (respectively at $z = 0$ and $z = \infty$) of the hypergeometric series:
\begin{align*}
 \frac{\al(\al-2)}{4} \frac{1}{R^2} \,_2 F_1\left(\frac{\al}{2}, 1+\frac{\al}{2}, 2, z\right) & = \frac{1}{R^2} \sum_{n=0}^{\infty} \frac{\Gamma\left(\frac{\al}{2}+n\right)\Gamma\left(n+1+\frac{\al}{2}\right)}{\Gamma\left(\frac{\al}{2}-1\right)\Gamma\left(\frac{\al}{2}\right)\Gamma(1+n)\Gamma(2+n)} \left(\frac{z}{R^2}\right)^n \\
\frac{\al(\al-2)}{4} z^{-1-\frac{\al}{2}} R^{\al} \,_2 F_1\left(\frac{\al}{2}, 1+\frac{\al}{2}, 2, \frac{1}{z}\right) & = \left(\frac{1}{z}\right)^{\frac{\al}{2}} R^{\al-2}\sum_{n=1}^{\infty} 
\frac{\Gamma\left(\frac{\al}{2}+n\right)\Gamma\left(n-1+\frac{\al}{2}\right)}{ \Gamma\left(\frac{\al}{2}-1\right) \Gamma\left(\frac{\al}{2}\right)\Gamma(n)\Gamma(n+1)}\left(\frac{R^2}{z}\right)^n,
\end{align*}
which are both minimized at $z = R^2$ since every coefficient is negative. This proves the claim.

The claim immediately implies
\begin{equation}\label{omegac}
\Omega_c(R)=-\frac{C_\alpha}{R}\frac{d}{d|x|} \int_{B(0,R)}|x-y|^{-\alpha}dy\bigg|_{|x|=R},
\end{equation}
Where $\frac{d}{d|x|}$ denotes the derivative in the radial variable (recall that $\int_{B(0,R)}|x-y|^{-\alpha}dy$ is radially symmetric). To compute the derivative at $|x|=R$, we can simply compute the partial derivative in the $x_1$ direction at the point $(R,0)$:
\begin{equation}\label{I1I2}
\begin{split}
& \frac{\partial}{\partial x_1}  \int_{B(0,R)}|x-y|^{-\alpha}dy\bigg|_{x = (R,0)} = -\alpha\int_{B(0,R)}\left((R-y_1)^2+y_2^2\right)^{-\frac{\alpha}{2}-1}(R-y_1)dy_1dy_2\\
&=-2\int_{0}^{R}\big((R-y_1)^2+{y_2}^2\big)^{-\frac{\alpha}{2}}\Big|_{y_1=-\sqrt{R^2-{y_2}^2}}^{y_1=\sqrt{R^2-{y_2}^2}}dy_2\\
&=-2^{1-\frac{\alpha}{2}} R^{1-\al}\left(\int_{0}^{1}\left(1-\sqrt{1-u^2}\right)^{-\frac{\alpha}{2}}d u -\int_{0}^{1}\left(1+\sqrt{1-u^2}\right)^{-\frac{\alpha}{2}}d u \right)\\
&=-2^{1-\frac{\alpha}{2}}R^{1-\al}\left(\int_0^\frac{\pi}{2}(1-\cos\theta)^{-\frac{\alpha}{2}}\cos\theta\, d\theta-\int_0^{\frac{\pi}{2}}(1+\cos\theta)^{-\frac{\alpha}{2}}\cos\theta \, d\theta\right)\\
&=:- 2^{1-\frac{\alpha}{2}} R^{1-\al}(I_{1}-I_{2}),
\end{split}
\end{equation}
where in the third line we used the identity $(R\pm \sqrt{R^2-y_2^2})^2 + y_2^2 = 2R^2(1 \pm \sqrt{1-(R^{-1}y_2)^2})$, as well as the substitution $u=R^{-1}y_2$.

Using a substitution $\theta = 2\beta$, we rewrite $I_1$ as
\begin{align*}
I_1& = 2\int_0^{\frac{\pi}{4}} (1-\cos(2\beta))^{-\frac{\alpha}{2}}\cos(2\beta)\, d\beta = 2^{1-\frac{\alpha}{2}}\int_{0}^\frac{\pi}{4}(\sin\beta)^{-\alpha}(1-2\sin^2\beta)\,d\beta.
\end{align*}

Likewise, the substitution $\theta = \pi-2\beta$ allows us to rewrite $-I_2$ as
\begin{align*}
-I_2&=2\int_{\frac{\pi}{4}}^{\frac{\pi}{2}}(1-\cos(2\beta))^{-\frac{\alpha}{2}}\cos(2\beta)\, d\beta =2^{1-\frac{\alpha}{2}} \int_{\frac{\pi}{4}}^{\frac{\pi}{2}}(\sin\beta)^{-\alpha}(1-2\sin^2\beta)\,d\beta.
\end{align*}
Adding the above two identities for $I_1$ and $-I_2$ together gives
\begin{align*}
I_1-I_2&=2^{1-\frac{\alpha}{2}}\int_0^\frac{\pi}{2}
(\sin\beta)^{-\alpha}(1-2\sin^2\beta)\, d\beta\\
&=2^{1-\frac{\alpha}{2}}\left( \frac{1}{2}B\Big(\frac{1-\alpha}{2},\frac{1}{2}\Big)-B\Big(\frac{3-\alpha}{2},\frac{1}{2}\Big)\right )\\
&=2^{-\frac{\alpha}{2}}\frac{\Gamma(\frac{1-\alpha}{2})\Gamma(\frac{1}{2})}{\Gamma(1-\frac{\alpha}{2})}-2^{1-\frac{\alpha}{2}}\frac{\Gamma(\frac{3-\alpha}{2})\Gamma(\frac{1}{2})}{\Gamma(2-\frac{\alpha}{2})},
\end{align*}
where $B$ stands for the beta function. Here the second identity follows from the property that $B(x,y) = 2\int_0^{\pi/2} (\sin\theta)^{2x-1} (\cos\theta)^{2y-1}d\theta$, and the third line follows from the property that $B(x,y) = \frac{\Gamma(x)\Gamma(y)}{\Gamma(x+y)}$.  According to the properties of the gamma function $\Gamma(z+1)=z\Gamma(z)$ and $\Gamma(z)\Gamma(z+\frac{1}{2})=2^{1-2z} \sqrt{\pi}\Gamma(2z)$, we have
\begin{equation}
I_1-I_2=2^{-1+\frac{\alpha}{2}} \frac{\alpha}{2-\alpha} \cdot \frac{2\pi\Gamma(1-\alpha)}{\Gamma(1-\frac{\alpha}{2})^2}.
\end{equation}
Finally, plugging this into \eqref{I1I2} and \eqref{omegac} gives
\begin{align*}
\Omega_c(R)&= R^{-\alpha} C_\alpha 2^{1-\frac{\alpha}{2}}  (I_1 - I_2)\\
&=R^{-\al}\frac{1}{2\pi}\frac{\Gamma(\frac{\alpha}{2})}{2^{1-\alpha}\Gamma(1-\frac{\alpha}{2})} \frac{\alpha}{2-\alpha}\bigg(\frac{2\pi\Gamma(1-\alpha)}{\Gamma(1-\frac{\alpha}{2})^2}\bigg)\\
&=R^{-\al}\frac{2^{\alpha-1}\Gamma(1-\alpha)\Gamma(\frac{\alpha}{2}+1)}{\Gamma(1-\frac{\alpha}{2})^2\Gamma(2-\frac{\alpha}{2})} = R^{-\al} \Omega_\alpha,
\end{align*}
finishing the proof.
\end{proof}


At the end of this section, we point out that Theorem~\ref{thm_large_omega} directly gives the following quantitative estimate: if a simply-connected patch $D$ rotates with angular velocity $\Omega \in (0,\Omega_\alpha)$ that is very close to $\Omega_\alpha$, then $D$ must be very close to a disk in terms of symmetric difference.

\begin{cor}Assume $0<\alpha<1$. 
Let $D$ be a rotating patch with area $\pi$ and angular velocity $\Omega \in (0, \Omega_\alpha)$, and let $B$ be the unit disk. Then we have
\[
|D \triangle B| \leq 2\pi \left(\Big(\frac{\Omega_\alpha}{\Omega}\Big)^{2/\alpha}-1\right).
\] 
Note that for a fixed $\alpha\in(0,1)$, the right hand side goes to 0 as $\Omega\nearrow\Omega_\alpha$. 

\end{cor}

\begin{proof}
Denote $R := \max_{x \in D}|x|$. If $D$ is a rotating patch with angular velocity $\Omega$ and is not a disk, Theorem~\ref{thm_large_omega} gives that $\Omega \leq R^{-\alpha}\Omega_\alpha$, which gives that $ R \leq (\frac{\Omega_\alpha}{\Omega})^{1/\alpha}$. Thus $D \subset B(0, (\frac{\Omega_\alpha}{\Omega})^{1/\alpha})$, which implies that the symmetric difference $D\triangle B$ satisfies
\[
|D\triangle B| = 2|D\setminus B| \leq 2 \left|B\Big(0, \big(\frac{\Omega_\alpha}{\Omega}\big)^{1/\alpha}\Big) \setminus B\right| = 2\pi \left(\Big(\frac{\Omega_\alpha}{\Omega}\Big)^{2/\alpha}-1\right).
\]
\end{proof}


\section*{Acknowledgements}

JGS was partially supported by NSF through Grant NSF DMS-1763356. JP was partially supported by NSF through Grants NSF DMS-1715418, and NSF CAREER Grant DMS-1846745. JS was partially supported by NSF through Grant NSF DMS-1700180. YY was partially supported by NSF through Grants NSF DMS-1715418, and NSF CAREER Grant DMS-1846745.

\bibliographystyle{abbrv}
\bibliography{references}

\def\cprime{$'$}
\begin{thebibliography}{10}

\bibitem{Alexander-Kim-Yao:quasistatic-evolution-crowd}
D.~Alexander, I.~Kim, and Y.~Yao.
\newblock Quasi-static evolution and congested crowd transport.
\newblock {\em Nonlinearity}, 27(4):823--858, 2014.

\bibitem{Arnold:geometrie-differentielle-dimension-infinie}
V.~Arnol{\cprime}d.
\newblock Sur la g\'eom\'etrie diff\'erentielle des groupes de {L}ie de
  dimension infinie et ses applications \`a l'hydrodynamique des fluides
  parfaits.
\newblock {\em Ann. Inst. Fourier (Grenoble)}, 16(fasc. 1):319--361, 1966.

\bibitem{Arnold:apriori-estimate-hydrodynamic-stability}
V.~I. Arnol{\cprime}d.
\newblock An a priori estimate in the theory of hydrodynamic stability.
\newblock {\em Izv. Vys\v s. U\v cebn. Zaved. Matematika}, 1966(5 (54)):3--5,
  1966.

\bibitem{Arnold-Khesin:topological-methods-hydrodynamics}
V.~I. Arnold and B.~A. Khesin.
\newblock {\em Topological methods in hydrodynamics}, volume 125 of {\em
  Applied Mathematical Sciences}.
\newblock Springer-Verlag, New York, 1998.

\bibitem{Beichman-Denisov:stability-rectangular-strip}
J.~Beichman and S.~Denisov.
\newblock 2{D} {E}uler equation on the strip: stability of a rectangular patch.
\newblock {\em Comm. Partial Differential Equations}, 42(1):100--120, 2017.

\bibitem{Bertozzi-Constantin:global-regularity-vortex-patches}
A.~L. Bertozzi and P.~Constantin.
\newblock Global regularity for vortex patches.
\newblock {\em Comm. Math. Phys.}, 152(1):19--28, 1993.

\bibitem{Brock:continuous-steiner-symmetrization}
F.~Brock.
\newblock Continuous {S}teiner-symmetrization.
\newblock {\em Math. Nachr.}, 172:25--48, 1995.

\bibitem{Burbea:motions-vortex-patches}
J.~Burbea.
\newblock Motions of vortex patches.
\newblock {\em Lett. Math. Phys.}, 6(1):1--16, 1982.

\bibitem{Cao-Liu-Wei:regularization-point-vortices}
D.~Cao, Z.~Liu, and J.~Wei.
\newblock Regularization of point vortices pairs for the {E}uler equation in
  dimension two.
\newblock {\em Arch. Ration. Mech. Anal.}, 212(1):179--217, 2014.

\bibitem{Cao-Peng:planar-vortex-patch-steady}
D.~Cao, S.~Peng, and S.~Yan.
\newblock Planar vortex patch problem in incompressible steady flow.
\newblock {\em Adv. Math.}, 270:263--301, 2015.

\bibitem{Cao-Wang:nonlinear-stability-patches-bounded-domains}
D.~Cao, J.~Wan, and G.~Wang.
\newblock Nonlinear orbital stability for planar vortex patches.
\newblock {\em Proc. Amer. Math. Soc.}, 147(2):775--784, 2019.

\bibitem{Carrillo-Hittmeir-Volzone-Yao:nonlinear-aggregation-symmetry-asymptotics}
J.~A. Carrillo, S.~Hittmeir, B.~Volzone, and Y.~Yao.
\newblock Nonlinear {A}ggregation-{D}iffusion {E}quations: {R}adial {S}ymmetry
  and {L}ong {T}ime {A}symptotics.
\newblock {\em Arxiv preprint arXiv:1603.07767, to appear in Invent. Math.},
  2016.

\bibitem{Carrillo-Hoffmann-Mainini-Volzone:ground-states-diffusion-regime}
J.~A. Carrillo, F.~Hoffmann, E.~Mainini, and B.~Volzone.
\newblock Ground states in the diffusion-dominated regime.
\newblock {\em Calc. Var. Partial Differential Equations}, 57(5):Art. 127, 28,
  2018.

\bibitem{Carrillo-Mateu-Mora-Rondi-Scardia-Verdera:dislocation-ellipses}
J.~A. Carrillo, J.~Mateu, M.~G. Mora, L.~Rondi, L.~Scardia, and J.~Verdera.
\newblock The ellipse law: {K}irchhoff meets dislocations.
\newblock {\em Arxiv preprint arXiv:1703.07013}, 2017.

\bibitem{Castro-Cordoba-GomezSerrano:existence-regularity-vstates-gsqg}
A.~Castro, D.~C{\'o}rdoba, and J.~G{\'o}mez-Serrano.
\newblock Existence and regularity of rotating global solutions for the
  generalized surface quasi-geostrophic equations.
\newblock {\em Duke Math. J.}, 165(5):935--984, 2016.

\bibitem{Castro-Cordoba-GomezSerrano:analytic-vstates-ellipses}
A.~Castro, D.~C{\'o}rdoba, and J.~G{\'o}mez-Serrano.
\newblock Uniformly rotating analytic global patch solutions for active
  scalars.
\newblock {\em Annals of PDE}, 2(1):1--34, 2016.

\bibitem{Castro-Cordoba-GomezSerrano:global-smooth-solutions-sqg}
A.~Castro, D.~C{\'o}rdoba, and J.~G{\'o}mez-Serrano.
\newblock Global smooth solutions for the inviscid {S}{Q}{G} equation.
\newblock {\em Memoirs of the AMS}, 2017.
\newblock To appear.

\bibitem{Castro-Cordoba-GomezSerrano:uniformly-rotating-smooth-euler}
A.~Castro, D.~C\'{o}rdoba, and J.~G\'{o}mez-Serrano.
\newblock Uniformly rotating smooth solutions for the incompressible 2{D}
  {E}uler equations.
\newblock {\em Arch. Ration. Mech. Anal.}, 231(2):719--785, 2019.

\bibitem{Chemin:persistance-structures-fluides-incompressibles}
J.-Y. Chemin.
\newblock Persistance de structures g\'eom\'etriques dans les fluides
  incompressibles bidimensionnels.
\newblock {\em Ann. Sci. \'Ecole Norm. Sup. (4)}, 26(4):517--542, 1993.

\bibitem{Choffrut-Sverak:local-structure-steady-euler}
A.~Choffrut and V.~{\v{S}}ver{\'a}k.
\newblock Local structure of the set of steady-state solutions to the 2{D}
  incompressible {E}uler equations.
\newblock {\em Geom. Funct. Anal.}, 22(1):136--201, 2012.

\bibitem{Choffrut-Szekelyhihi:weak-solutions-stationary-euler}
A.~Choffrut and L.~Sz{\'e}kelyhidi, Jr.
\newblock Weak solutions to the stationary incompressible {E}uler equations.
\newblock {\em SIAM J. Math. Anal.}, 46(6):4060--4074, 2014.

\bibitem{Choksi-Neumayer-Topaloglu:anisotropic-liquid-drop-models}
R.~Choksi, R.~Neumayer, and I.~Topaloglu.
\newblock Anisotropic liquid drop models.
\newblock {\em Arxiv preprint arXiv:1810.08304}, 2018.

\bibitem{Constantin-La-Vicol:remarks-gavrilov-stationary}
P.~Constantin, J.~La, and V.~Vicol.
\newblock Remarks on a paper by {G}avrilov: {G}rad-{S}hafranov equations,
  steady solutions of the three dimensional incompressible {E}uler equations
  with compactly supported velocities, and applications.
\newblock {\em arXiv preprint arXiv:1903.11699}, 2019.

\bibitem{Constantin-Titi:evolution-nearly-circular-vortex-patches}
P.~Constantin and E.~S. Titi.
\newblock On the evolution of nearly circular vortex patches.
\newblock {\em Comm. Math. Phys.}, 119(2):177--198, 1988.

\bibitem{Cordoba-Cordoba-Gancedo:uniqueness-sqg-patch}
A.~C\'ordoba, D.~C\'ordoba, and F.~Gancedo.
\newblock Uniqueness for {SQG} patch solutions.
\newblock {\em Trans. Amer. Math. Soc. Ser. B}, 5:1--31, 2018.

\bibitem{Craig-Kim-Yao:aggregation-newtonian}
K.~Craig, I.~Kim, and Y.~Yao.
\newblock Congested aggregation via {N}ewtonian interaction.
\newblock {\em Arch. Ration. Mech. Anal.}, 227(1):1--67, 2018.

\bibitem{delaHoz-Hassainia-Hmidi:doubly-connected-vstates-gsqg}
F.~de~la Hoz, Z.~Hassainia, and T.~Hmidi.
\newblock Doubly connected {V}-states for the generalized surface
  quasi-geostrophic equations.
\newblock {\em Arch. Ration. Mech. Anal.}, 220(3):1209--1281, 2016.

\bibitem{delaHoz-Hassainia-Hmidi-Mateu:vstates-disk-euler}
F.~de~la Hoz, Z.~Hassainia, T.~Hmidi, and J.~Mateu.
\newblock An analytical and numerical study of steady patches in the disc.
\newblock {\em Anal. PDE}, 9(7):1609--1670, 2016.

\bibitem{Hmidi-delaHoz-Mateu-Verdera:doubly-connected-vstates-euler}
F.~de~la Hoz, T.~Hmidi, J.~Mateu, and J.~Verdera.
\newblock Doubly connected {$V$}-states for the planar {E}uler equations.
\newblock {\em SIAM J. Math. Anal.}, 48(3):1892--1928, 2016.

\bibitem{Deem-Zabusky:vortex-waves-stationary}
G.~S. Deem and N.~J. Zabusky.
\newblock Vortex waves: Stationary "{V}-states", interactions, recurrence, and
  breaking.
\newblock {\em Physical Review Letters}, 40(13):859--862, 1978.

\bibitem{Elcrat-Fornberg-Miller:stability-vortices-cylinder}
A.~Elcrat, B.~Fornberg, and K.~Miller.
\newblock Stability of vortices in equilibrium with a cylinder.
\newblock {\em Journal of Fluid Mechanics}, 544:53--68, 2005.

\bibitem{Fraenkel:book-maximum-principles-symmetry-elliptic}
L.~E. Fraenkel.
\newblock {\em An introduction to maximum principles and symmetry in elliptic
  problems}, volume 128 of {\em Cambridge Tracts in Mathematics}.
\newblock Cambridge University Press, Cambridge, 2000.

\bibitem{Fusco-Maggi-Pratelli:stability-faber-krahn}
N.~Fusco, F.~Maggi, and A.~Pratelli.
\newblock Stability estimates for certain {F}aber-{K}rahn, isocapacitary and
  {C}heeger inequalities.
\newblock {\em Ann. Sc. Norm. Super. Pisa Cl. Sci. (5)}, 8(1):51--71, 2009.

\bibitem{Gancedo:existence-alpha-patch-sobolev}
F.~Gancedo.
\newblock Existence for the {$\alpha$}-patch model and the {QG} sharp front in
  {S}obolev spaces.
\newblock {\em Adv. Math.}, 217(6):2569--2598, 2008.

\bibitem{Garcia-Hmidi-Soler:non-uniform-vstates-euler}
C.~Garc\'ia, T.~Hmidi, and J.~Soler.
\newblock Non uniform rotating vortices and periodic orbits for the
  two-dimensional {E}uler {E}quations.
\newblock {\em arXiv preprint arXiv:1807.10017}, 2018.

\bibitem{Gavrilov-stationary-euler-3d}
A.~V. Gavrilov.
\newblock A steady {E}uler flow with compact support.
\newblock {\em Geom. Funct. Anal.}, 29(1):190--197, 2019.

\bibitem{Gavrilov:stationary-euler-helix}
A.~V. Gavrilov.
\newblock A steady smooth {E}uler flow with support in the vicinity of a helix.
\newblock {\em arXiv preprint arXiv:1906.07465}, 2019.

\bibitem{Gerschgorin:eigenvalues-theorem}
S.~A. Gershgorin.
\newblock {\"U}ber die {A}bgrenzung der {E}igenwerte einer {M}atrix.
\newblock {\em Bulletin de l'Acad\'emie des Sciences de l'URSS. Classe des
  sciences math\'ematiques et naturelles}, (6):749--754, 1931.

\bibitem{Gidas-Ni-Nirenberg:symmetry-maximum-principle}
B.~Gidas, W.~M. Ni, and L.~Nirenberg.
\newblock Symmetry and related properties via the maximum principle.
\newblock {\em Comm. Math. Phys.}, 68(3):209--243, 1979.

\bibitem{GomezSerrano:stationary-patches}
J.~G\'{o}mez-Serrano.
\newblock On the existence of stationary patches.
\newblock {\em Adv. Math.}, 343:110--140, 2019.

\bibitem{GomezSerrano-Park-Shi-Yao:nonradial-stationary-solutions}
J.~G\'omez-Serrano, J.~Park, J.~Shi, and Y.~Yao.
\newblock On nonradial stationary solutions of the 2{D} {E}uler equations.
\newblock 2019.
\newblock In preparation.

\bibitem{Guo-Hallstrom-Spirn:dynamics-unstable-kirchhoff-ellipse}
Y.~Guo, C.~Hallstrom, and D.~Spirn.
\newblock Dynamics near an unstable {K}irchhoff ellipse.
\newblock {\em Communications in Mathematical Physics}, 245(2):297--354, 2004.

\bibitem{Hamel-Nadirashvili:shear-flow-euler-strip-halfspace}
F.~Hamel and N.~Nadirashvili.
\newblock Shear flows of an ideal fluid and elliptic equations in unbounded
  domains.
\newblock {\em Comm. Pure Appl. Math.}, 70(3):590--608, 2017.

\bibitem{Hamel-Nadirashvili:liouville-euler}
F.~Hamel and N.~Nadirashvili.
\newblock A {L}iouville theorem for the {E}uler equations in the plane.
\newblock {\em Arch. Ration. Mech. Anal.}, 233(2):599--642, 2019.

\bibitem{Han-Lu-Zhu:characterization-balls-bessel-potentials}
X.~Han, G.~Lu, and J.~Zhu.
\newblock Characterization of balls in terms of {B}essel-potential integral
  equation.
\newblock {\em J. Differential Equations}, 252(2):1589--1602, 2012.

\bibitem{Hassainia-Hmidi:v-states-generalized-sqg}
Z.~Hassainia and T.~Hmidi.
\newblock On the {V}-states for the generalized quasi-geostrophic equations.
\newblock {\em Comm. Math. Phys.}, 337(1):321--377, 2015.

\bibitem{Hassainia-Masmoudi-Wheeler:global-bifurcation-vortex-patches}
Z.~Hassainia, N.~Masmoudi, and M.~Wheeler.
\newblock Global bifurcation of rotating vortex patches.
\newblock {\em Comm. Pure Appl. Math.}, 2019.
\newblock To appear.

\bibitem{Hmidi:trivial-solutions-rotating-patches}
T.~Hmidi.
\newblock On the trivial solutions for the rotating patch model.
\newblock {\em J. Evol. Equ.}, 15(4):801--816, 2015.

\bibitem{Hmidi-Mateu:bifurcation-kirchhoff-ellipses}
T.~Hmidi and J.~Mateu.
\newblock Bifurcation of rotating patches from kirchhoff vortices.
\newblock {\em Discrete and Continuous Dynamical Systems}, 36(10):5401--5422,
  2016.

\bibitem{Hmidi-Mateu:existence-corotating-counter-rotating}
T.~Hmidi and J.~Mateu.
\newblock Existence of corotating and counter-rotating vortex pairs for active
  scalar equations.
\newblock {\em Comm. Math. Phys.}, 350(2):699--747, 2017.

\bibitem{Hmidi-Mateu-Verdera:rotating-vortex-patch}
T.~Hmidi, J.~Mateu, and J.~Verdera.
\newblock Boundary regularity of rotating vortex patches.
\newblock {\em Archive for Rational Mechanics and Analysis}, 209(1):171--208,
  2013.

\bibitem{Hmidi-Mateu-Verdera:rotating-doubly-connected-vortices}
T.~Hmidi, J.~Mateu, and J.~Verdera.
\newblock On rotating doubly connected vortices.
\newblock {\em Journal of Differential Equations}, 258(4):1395 -- 1429, 2015.

\bibitem{Izosimov-Khesin:characterization-steady-solutions-2d-euler}
A.~Izosimov and B.~Khesin.
\newblock Characterization of steady solutions to the 2{D} {E}uler equation.
\newblock {\em Int. Math. Res. Not. IMRN}, (24):7459--7503, 2017.

\bibitem{Kamm:thesis-shape-stability-patches}
J.~R. Kamm.
\newblock {\em Shape and stability of two-dimensional uniform vorticity
  regions}.
\newblock PhD thesis, California Institute of Technology, 1987.

\bibitem{Kirchhoff:vorlesungen-math-physik}
G.~Kirchhoff.
\newblock {\em Vorlesungen {\"u}ber mathematische {P}hysik}, volume~1.
\newblock Teubner, 1874.

\bibitem{Kiselev-Yao-Zlatos:local-regularity-sqg-patch-boundary}
A.~Kiselev, Y.~Yao, and A.~Zlato\v{s}.
\newblock Local regularity for the modified {SQG} patch equation.
\newblock {\em Comm. Pure Appl. Math.}, 70(7):1253--1315, 2017.

\bibitem{Koch-Nadirashvili-Seregin-Sverak:liouville-navier-stokes}
G.~Koch, N.~Nadirashvili, G.~A. Seregin, and V.~\v{S}ver\'{a}k.
\newblock Liouville theorems for the {N}avier-{S}tokes equations and
  applications.
\newblock {\em Acta Math.}, 203(1):83--105, 2009.

\bibitem{Kwasnicki:definitions-fractional-laplacian}
M.~Kwa\'{s}nicki.
\newblock Ten equivalent definitions of the fractional {L}aplace operator.
\newblock {\em Fract. Calc. Appl. Anal.}, 20(1):7--51, 2017.

\bibitem{Lieb-Loss:analysis}
E.~H. Lieb and M.~Loss.
\newblock {\em Analysis}, volume~14 of {\em Graduate Studies in Mathematics}.
\newblock American Mathematical Society, Providence, RI, second edition, 2001.

\bibitem{Long-Wang-Zeng:concentrated-steady-vortex-patches}
Y.~Long, Y.~Wang, and C.~Zeng.
\newblock Concentrated steady vorticities of the {E}uler equation on 2-d
  domains and their linear stability.
\newblock {\em J. Differential Equations}, 266(10):6661--6701, 2019.

\bibitem{Love:stability-ellipses}
A.~E.~H. Love.
\newblock On the {S}tability of certain {V}ortex {M}otions.
\newblock {\em Proc. London Math. Soc.}, 25(1):18--42, 1893.

\bibitem{Lu-Zhu:overdetermined-riesz-potential}
G.~Lu and J.~Zhu.
\newblock An overdetermined problem in {R}iesz-potential and fractional
  {L}aplacian.
\newblock {\em Nonlinear Anal.}, 75(6):3036--3048, 2012.

\bibitem{Luo-Shvydkoy:2d-homogeneous-euler}
X.~Luo and R.~Shvydkoy.
\newblock 2{D} homogeneous solutions to the {E}uler equation.
\newblock {\em Comm. Partial Differential Equations}, 40(9):1666--1687, 2015.

\bibitem{Luo-Shvydkoy:addendum-homogeneous-euler}
X.~Luo and R.~Shvydkoy.
\newblock Addendum: 2{D} homogeneous solutions to the {E}uler equation.
\newblock {\em Comm. Partial Differential Equations}, 42(3):491--493, 2017.

\bibitem{LuzzattoFegiz-Williamson:efficient-numerical-method-steady-uniform-vortices}
P.~Luzzatto-Fegiz and C.~H.~K. Williamson.
\newblock An efficient and general numerical method to compute steady uniform
  vortices.
\newblock {\em Journal of Computational Physics}, {230}({17}):{6495--6511},
  2011.

\bibitem{Maury-RoudneffChupin-Santambrogio:macroscopic-crowd-gradient-flow}
B.~Maury, A.~Roudneff-Chupin, and F.~Santambrogio.
\newblock A macroscopic crowd motion model of gradient flow type.
\newblock {\em Math. Models Methods Appl. Sci.}, 20(10):1787--1821, 2010.

\bibitem{Maury-RoudneffChupin-Santambrogio-Venel:handling-congestion-crowd}
B.~Maury, A.~Roudneff-Chupin, F.~Santambrogio, and J.~Venel.
\newblock Handling congestion in crowd motion modeling.
\newblock {\em Netw. Heterog. Media}, 6(3):485--519, 2011.

\bibitem{Morgan:ball-minimizes-grativational-energy}
F.~Morgan.
\newblock A round ball uniquely minimizes gravitational potential energy.
\newblock {\em Proc. Amer. Math. Soc.}, 133(9):2733--2735, 2005.

\bibitem{Musso-Pacard-Wei:stationary-solutions-euler}
M.~Musso, F.~Pacard, and J.~Wei.
\newblock Finite-energy sign-changing solutions with dihedral symmetry for the
  stationary nonlinear {S}chr\"{o}dinger equation.
\newblock {\em J. Eur. Math. Soc. (JEMS)}, 14(6):1923--1953, 2012.

\bibitem{Nadirashvili:stationary-2d-euler}
N.~Nadirashvili.
\newblock On stationary solutions of two-dimensional {E}uler equation.
\newblock {\em Arch. Ration. Mech. Anal.}, 209(3):729--745, 2013.

\bibitem{Reichel:balls-riesz-potentials}
W.~Reichel.
\newblock Characterization of balls by {R}iesz-potentials.
\newblock {\em Ann. Mat. Pura Appl. (4)}, 188(2):235--245, 2009.

\bibitem{Renault:relative-equlibria-holes-sqg}
C.~Renault.
\newblock Relative equilibria with holes for the surface quasi-geostrophic
  equations.
\newblock {\em J. Differential Equations}, 263(1):567--614, 2017.

\bibitem{Rodrigo:evolution-sharp-fronts-qg}
J.~L. Rodrigo.
\newblock On the evolution of sharp fronts for the quasi-geostrophic equation.
\newblock {\em Comm. Pure Appl. Math.}, 58(6):821--866, 2005.

\bibitem{Saffman-Szeto:equilibrium-shapes-equal-uniform-vortices}
P.~Saffman and R.~Szeto.
\newblock {Equilibrium shapes of a pair of equal uniform vortices}.
\newblock {\em {Physics of Fluids}}, {23}({12}):{2339--2342}, {1980}.

\bibitem{Serrin:symmetry-moving-plane}
J.~Serrin.
\newblock A symmetry problem in potential theory.
\newblock {\em Arch. Rational Mech. Anal.}, 43:304--318, 1971.

\bibitem{Sideris-Vega:stability-L1-patches}
T.~C. Sideris and L.~Vega.
\newblock Stability in {$L^1$} of circular vortex patches.
\newblock {\em Proc. Amer. Math. Soc.}, 137(12):4199--4202, 2009.

\bibitem{Smets-VanSchaftingen:desingularization-vortices-euler}
D.~Smets and J.~Van~Schaftingen.
\newblock Desingularization of vortices for the {E}uler equation.
\newblock {\em Arch. Ration. Mech. Anal.}, 198(3):869--925, 2010.

\bibitem{Talenti:rearrangements}
G.~Talenti.
\newblock Elliptic equations and rearrangements.
\newblock {\em Ann. Scuola Norm. Sup. Pisa Cl. Sci. (4)}, 3(4):697--718, 1976.

\bibitem{Tang:nonlinear-stability-vortex-patches}
Y.~Tang.
\newblock Nonlinear stability of vortex patches.
\newblock {\em Trans. Amer. Math. Soc.}, 304(2):617--638, 1987.

\bibitem{Turkington:stationary-vortices}
B.~Turkington.
\newblock On steady vortex flow in two dimensions. {I}, {II}.
\newblock {\em Comm. Partial Differential Equations}, 8(9):999--1030,
  1031--1071, 1983.

\bibitem{Turkington:corotating-vortices}
B.~Turkington.
\newblock Corotating steady vortex flows with {$N$}-fold symmetry.
\newblock {\em Nonlinear Anal.}, 9(4):351--369, 1985.

\bibitem{Wan:stability-rotating-vortex-patches}
Y.~H. Wan.
\newblock The stability of rotating vortex patches.
\newblock {\em Comm. Math. Phys.}, 107(1):1--20, 1986.

\bibitem{Wan-Pulvirenti:stability-circular-patches}
Y.~H. Wan and M.~Pulvirenti.
\newblock Nonlinear stability of circular vortex patches.
\newblock {\em Comm. Math. Phys.}, 99(3):435--450, 1985.

\bibitem{Watson:treatise-bessel-functions}
G.~N. Watson.
\newblock {\em A {T}reatise on the {T}heory of {B}essel {F}unctions}.
\newblock Cambridge University Press, Cambridge, England; The Macmillan
  Company, New York, 1944.

\bibitem{Wu-Overman-Zabusky:steady-state-Euler-2d}
H.~M. Wu, E.~A. Overman, II, and N.~J. Zabusky.
\newblock Steady-state solutions of the {E}uler equations in two dimensions:
  rotating and translating {$V$}-states with limiting cases. {I}. {N}umerical
  algorithms and results.
\newblock {\em J. Comput. Phys.}, 53(1):42--71, 1984.

\end{thebibliography}

\begin{tabular}{ll}
\textbf{Javier G\'omez-Serrano} & \textbf{Jia Shi}\\
{\small Department of Mathematics} & {\small Department of Mathematics}\\
{\small Princeton University} & {\small Princeton University} \\
{\small 610 Fine Hall, Washington Rd,} & {\small A09 Fine Hall, Washington Rd,}\\
{\small Princeton, NJ 08544, USA} & {\small Princeton, NJ 08544, USA}\\
 {\small e-mail: jg27@math.princeton.edu} & {\small e-mail: jiashi@math.princeton.edu}\\[0.3cm]
\textbf{Jaemin Park} & \textbf{Yao Yao} \\
{\small School of Mathematics, Georgia Tech} & {\small School of Mathematics, Georgia Tech}\\
{\small 686 Cherry Street, Atlanta, GA 30332} & {\small 686 Cherry Street, Atlanta, GA 30332}\\
{\small Email: jpark776@gatech.edu} & {\small Email: yaoyao@math.gatech.edu}\\
   & \\
\end{tabular}

%
%
%
%
%
%
%
%

\end{document}